\documentclass[a4paper,11pt]{article}
\usepackage[utf8]{inputenc} 
\usepackage[english]{babel}
\usepackage[T1]{fontenc}   
\usepackage{graphicx}
\usepackage{comment}
\usepackage{amsthm,amsmath,amsfonts,stmaryrd,mathrsfs,amssymb}
\usepackage{mathtools}
\usepackage{enumitem}
\usepackage{physics}
\usepackage{chemarrow}
\usepackage{tikz}
\usepackage[top=2.5cm, bottom=2.5cm, left=2cm, right=2cm]{geometry}
\usepackage{subfig}
\usepackage{multirow}
\usepackage{array}
\usepackage{bigints}
\usepackage{etoolbox}
\usepackage{mdframed}
\usepackage{color}
\usepackage[backend=bibtex,style=trad-plain,abbreviate=false,doi=false,isbn=false,url=false,useprefix=false]{biblatex}
\DeclareSourcemap{
	\maps[datatype=bibtex]{
		\map[overwrite=true]{
			\step[fieldsource=fjournal]
			\step[fieldset=journal, origfieldval]
		}
	}
}

\renewbibmacro*{journal}{%
	\printfield{journaltitle}%
	\setunit{\subtitlepunct}%
	\printfield{journalsubtitle}}

\setenumerate[1]{label=(\roman*), font = \normalfont}
\makeatletter
\newcommand{\mylabel}[2]{#2\def\@currentlabel{#2}\label{#1}}
\makeatother

\newcommand{\ensemblenombre}[1]{\mathbb{#1}}
\newcommand{\N}{\ensemblenombre{N}}
\newcommand{\Z}{\ensemblenombre{Z}}
\newcommand{\Q}{\ensemblenombre{Q}}
\newcommand{\R}{\ensemblenombre{R}}

\renewcommand{\P}{\mathbb{P}}
\newcommand{\E}{\mathbb{E}}
\newcommand{\D}{\mathbb{D}}
\newcommand{\M}{\mathbb{M}}
\newcommand{\K}{\mathbb{K}}
\newcommand{\bU}{\mathbb{U}}
\newcommand{\Ec}[1]{\mathbb{E} \left[#1\right]}

\newcommand{\Pp}[1]{\mathbb{P} \left(#1\right)}

\newcommand{\Ecsq}[2]{\mathbb{E} \left[#1\mathrel{}\middle|\mathrel{}#2\right]}

\newcommand{\Ppsq}[2]{\mathbb{P} \left(#1\mathrel{}\middle|\mathrel{}#2\right)}

\newcommand{\Var}[1]{\mathbb{V}\left(#1\right)}

\newcommand{\intervalle}[4]{\mathopen{#1}#2
                            \mathclose{}\mathpunct{},#3
                            \mathclose{#4}}
\newcommand{\intervalleff}[2]{\intervalle{[}{#1}{#2}{]}}
\newcommand{\intervalleof}[2]{\intervalle{(}{#1}{#2}{]}}

\newcommand{\intervalleoo}[2]{\intervalle{(}{#1}{#2}{)}}
\newcommand{\intervalleentier}[2]{\intervalle\llbracket{#1}{#2}
                               \rrbracket}
       
\newcommand{\petito}[1]{o\mathopen{}\left(#1\right)}

\newcommand{\grandO}[1]{O\mathopen{}\left(#1\right)}

\newcommand{\enstq}[2]{\left\lbrace#1\mathrel{}\middle|\mathrel{}#2\right\rbrace}

\newcommand{\ind}[1]{\mathbf{1}_{\left\lbrace #1 \right\rbrace}}  
\newcommand{\cT}{\mathcal{T}}
\newcommand{\cB}{\mathcal{B}}
\newcommand{\cA}{\mathcal{A}}
\newcommand{\cL}{\mathcal{L}}
\newcommand{\cS}{\mathcal{S}}
\newcommand{\cP}{\mathcal{P}}
\newcommand{\cG}{\mathcal{G}}
\newcommand{\cU}{\mathcal{U}}

\newcommand{\cC}{\mathcal{C}}

\newcommand{\bT}{\mathbb{T}}

\newcommand{\cH}{\mathcal{H}}
\newcommand{\cF}{\mathcal{F}}

\newcommand{\cD}{\mathcal{D}}

\newcommand{\ttT}{\mathtt{T}}
\newcommand{\ttP}{\mathtt{P}}

\newcommand{\bnu}{\boldsymbol\nu}

\newcommand{\bB}{\mathsf{B}}
\newcommand{\bRho}{\rho}
\newcommand{\bD}{\mathsf{D}}
\newcommand{\bNu}{\nu}

\newcommand{\bi}{\mathbf{i}}
\newcommand{\bj}{\mathbf{j}}

\newcommand{\sL}{\mathscr{L}}

\newcommand{\sG}{\mathscr{G}}
\newcommand{\Gam}[1]{\Gamma \left(#1\right)}

\DeclareMathOperator{\haut}{ht}
\DeclareMathOperator{\diam}{diam}

\DeclareMathOperator{\dist}{d}

\DeclareMathOperator{\Ball}{B}
\DeclareMathOperator{\supp}{supp}

\DeclareMathOperator{\wrt}{WRT}
\DeclareMathOperator{\pa}{PA}

\DeclareMathOperator{\GEM}{GEM}
\DeclareMathOperator{\ML}{ML}
\DeclareMathOperator{\MLMC}{MLMC}

\newmdtheoremenv{theorem}{Theorem}
\newmdtheoremenv{proposition}[theorem]{Proposition}

\newtheorem{lemma}[theorem]{Lemma}

\newtheorem{remark}[theorem]{Remark}

\newtheorem{hypothesismx}{Hypothesis}
\newenvironment{hypothesis}[1]
{\renewcommand\thehypothesismx{#1}\hypothesismx}
{\endhypothesismx}

\begin{document}
	\title{Growing random graphs with a preferential attachment structure}
\author{Delphin Sénizergues}
\maketitle
\abstract{
	The aim of this paper is to develop a method for proving almost sure convergence in Gromov-Hausodorff-Prokhorov topology for a class of models of growing random graphs that generalises Rémy's algorithm for binary trees. We describe the obtained limits using some iterative gluing construction that generalises the famous line-breaking construction of Aldous' Brownian tree. 

In order to do that, we develop a framework in which a metric space is constructed by gluing smaller metric spaces, called \emph{blocks}, along the structure of a (possibly infinite) discrete tree. 
Our growing random graphs seen as metric spaces can be understood in this framework, that is, as evolving blocks glued along a growing discrete tree structure. Their scaling limit convergence can then be obtained by separately proving the almost sure convergence of every block and verifying some relative compactness property for the whole structure. 
For the particular models that we study, the discrete tree structure behind the construction has the distribution of an affine preferential attachment tree or a weighted recursive tree. We strongly rely on results concerning those two models and their connection, obtained in the companion paper \cite{senizergues_geometry_2019}. 
}

\section{Introduction}

We start by introducing a particular model of growing random graphs, which we call \emph{generalised Rémy's algorithm}, as an example of models that is handled by our methods.
Other models are discussed at the end of the introduction.

\subsection{A generalised version of Rémy's algorithm}
  
\begin{figure}
	\centering
	\includegraphics[height=2cm]{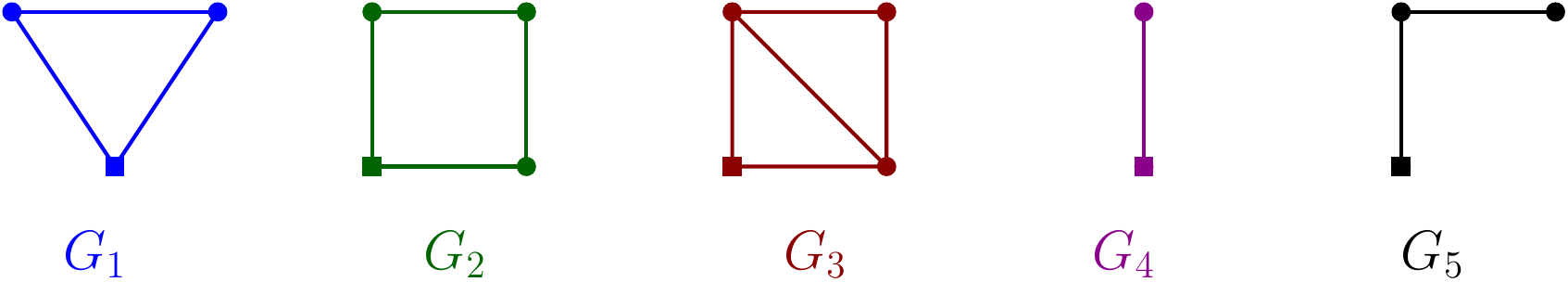}
	\caption{An example of a sequence of graphs used to run the algorithm, the root of each graph is represented by a square vertex}\label{growing:fig:exemple graphes gn}
\end{figure}

Consider $(G_n)_{n\geq 1}$ a sequence of finite connected rooted graphs and construct the sequence $(H_n)_{n\geq 1}$ recursively as follows. Let $H_1=G_1$. Then, for any $n\geq 1$, conditionally on the structure $H_n$ already constructed, take an edge in $H_n$ uniformly at random , split it into two edges by adding a vertex "in the middle" of this edge, and glue a copy of $G_{n+1}$ to the structure by identifying the root vertex of $G_{n+1}$ with the newly created vertex. Call the obtained graph $H_{n+1}$.

When all the graphs $(G_n)_{n\geq 1}$ are equal to the single-edge graph, we obtain the so-called Rémy's algorithm, which produces for each $n$ a uniform planted binary tree with $n$ leaves (if the leaves are labelled for example). Remark that this kind of generalisation of the algorithm has already been studied for particular sequences $(G_n)_{n\geq 1}$, namely for $(G_n)_{n\geq 1}$ constant equal to the star-graph with $k-1$ branches in \cite{haas_scaling_2015}, where the authors show that the obtained trees converge in the scaling limit to some fragmentation tree; and in \cite{ross_scaling_2018}, for $G_n$ being equal to the single edge graph for every $n \equiv 1 \mod \ell$, for some $\ell\geq 2$, and equal to single-vertex graph whenever $n \not\equiv 1 \mod \ell$.

We see the graphs $(H_n)_{n\geq 1}$ as measured metric spaces, by considering their set of vertices endowed with the usual graph distance and the uniform measure on vertices. It is well-known \cite{curien_stable_2013} that the sequence of trees created through the standard Rémy's algorithm with distances rescaled by $n^{-1/2}$ converges almost surely in Gromov-Hausdorff-Prokhorov topology to a constant multiple of Aldous' Brownian tree.
We give here an analogous result, under some conditions on the sequence $(G_n)_{n\geq 1}$, which ensures that the graphs $(H_n)_{n\geq 1}$ appropriately rescaled converge almost surely in the Gromov-Hausdorff-Prokhorov topology to a random compact metric space.

\begin{proposition}\label{growing:prop:convergence rémy généralisé}
Call $(a_n)_{n\geq 1}$ the respective numbers of edges in the graphs $(G_n)_{n\geq 1}$. Suppose there exists $c>0$ and $0\leq c'<\frac{1}{c+1}$ and $\epsilon>0$ such that, as $n\rightarrow\infty$,
	\begin{equation*}
	\sum_{i=1}^{n}a_i=c\cdot n \cdot \left(1+\grandO{n^{-\epsilon}}\right), \qquad and \qquad  a_n\leq n^{c'+\petito{1}},
	\end{equation*}
	then we have the following convergence, almost surely in GHP topology
	\begin{equation}\label{growing:convergence generalised remy}
	(H_n,n^{\frac{-1}{c+1}}\cdot\dist_{\mathrm{gr}},\mu_{\mathrm{unif}})\underset{n\rightarrow\infty}{\longrightarrow} (\cH,\dist,\mu).
	\end{equation}
\end{proposition}
As for Aldous' Brownian tree, the limiting random compact metric space $(\cH,\dist,\mu)$, which depends on the whole sequence $(G_n)_{n\geq 1}$, can be described as the result of an \emph{iterative gluing construction}, as defined in \cite{senizergues_random_2019}. It is a natural extension of the famous \emph{line-breaking construction} invented by Aldous \cite{aldous_continuum_1991}, but with branches that are allowed to be more complex than just segments.
\paragraph{A line-breaking construction.}
The construction of $(\cH,\dist,\mu)$ is described as follows. We first run some increasing time-inhomogeneous Markov chain $(\mathsf M_n^\mathbf{a})_{n\geq 1}$ which takes values in $\R_+$, and whose law  depends only on the sequence $\mathbf{a}=(a_n)_{n\geq 1}$. 
 
Cut the semi-infinite line $\R_+$ at the values $\mathsf M_1^\mathbf{a}, \mathsf M_2^\mathbf{a}\dots$ taken by the chain. This creates an ordered sequence of segments with length $\mathsf M_1^\mathbf{a},(\mathsf M_2^\mathbf{a}-\mathsf M_1^\mathbf{a}),(\mathsf M_3^\mathbf{a}-\mathsf M_2^\mathbf{a}),\dots$.
Now for any $n\geq 1$
\begin{enumerate}
	\item  Cut the $n$-th segment into $a_n$ sub-segments by throwing $a_n-1$ uniform points on it, and call $(L_{n,1},L_{n,2},\dots,L_{n,a_n})$ the respective length of the obtained sub-segments.
	\item Take the graph $G_n$ and replace every edge $e_k\in\{e_1,\dots ,e_{a_n}\}$, where the edges of $G_n$ are labelled in an arbitrary order, with a segment of length $L_{n,k}$. Call the result $\cG_n$. 
\end{enumerate}
Now, start from $\cH_1:=\cG_1$ and recursively when $\cH_n$ is already constructed, sample a point according to the length measure on $\cH_n$ and identify the root of $\cG_{n+1}$ to the chosen point. 
The space $\cH$ is obtained as the completion of the increasing union
\[\cH=\overline{\bigcup_{n\geq 1}\cH_n}.\]
The measure $\mu$ is the weak limit of the normalised length measure carried by the $\cH_n$'s.

\subsection{Metric spaces glued along a tree structure}
We introduce a general framework that allows us to handle objects that are defined as the result of gluing together metric spaces along a discrete tree structure. Consider the Ulam tree with its usual representation as
\begin{equation}
\bU=\bigcup_{n\geq 0} \N^n.
\end{equation}
We say that $\cD=\left(\cD(u)\right)_{u\in\bU}$ is a \emph{decoration} on the Ulam tree if for any $u\in \bU$,
\begin{equation*}
\cD(u)=(D_u,d_u,\rho_u,(x_{ui})_{i\geq 1}),
\end{equation*}
is a compact rooted metric space, with underlying set $D_u$, distance $\dist_u$, rooted at a point $\rho_u$ and endowed with a sequence $(x_{ui})_{i\geq 1}\in D_u$.
Then for any such decoration $\cD$, we make sense of the following metric space $\sG(\cD)$, which is informally what we get if we take the disjoint union $\bigsqcup_{u\in\bU}D_u$ and identify every root $\rho_{ui}\in D_{ui}$ to the distinguished point $x_{ui}\in D_u$ for every $u\in\bU$ and every $i\in\N$, and take the metric completion of the obtained metric space. 

This setting also encompasses the case where we only glue a finite number of blocks along a plane tree. 
If $\tau$ is a plane tree, it can be natural to consider a decoration $\cD=(\cD(u))_{u\in \tau}$ which is only defined on the vertices of $\tau$ and is such that for all $u\in\tau$, the block $\cD(u)=(D_u,d_u,\rho_u,(x_{ui})_{1\leq i\leq \deg_\tau^+(u)})$ is only endowed with a finite number of distinguished points that corresponds to the number $\deg_{\tau}^+(u)$ of children of $u$ in $\tau$.
In this case we automatically extend $\cD$ by letting $\cD(u)$ be the one-point space $(\star,0,\star,(\star)_{i\geq 1})$ for all $u\notin \tau$ and by letting $x_{ui}=\rho_i$ for all $u\in\tau$ and $i> \deg_{\tau}^+(u)$.
Thanks to this identification, we always consider decorations that are defined on the whole Ulam tree $\bU$ and for which all the blocks have infinitely many distinguished points.

\paragraph{Convergence of metric spaces glued along $\bU$.}
For a sequence $(\cD_n)_{n\geq 1}$ of decorations, we have a sufficient condition for the convergence of the sequence $(\sG(\cD_n))_{n\geq 1}$ in the Gromov-Hausdorff topology. 
Indeed, we will see in Theorem~\ref{growing:thm:continuité du recollement} that it suffices that for every $u\in\bU$, we have the convergence $\cD_n(u)\rightarrow\cD_\infty(u)$ for some decoration $\cD_\infty$ in the some appropriate infinitely pointed Gromov-Hausdorff topology and that the sequence of decorations $(\cD_n)_{n\geq 1}$ satisfies the following relative compactness property
\begin{equation*}
\underset{\theta \text{ finite}}{\inf_{\theta\subset \bU}} \sup_{u\in\bU}\left( \underset{v\notin \theta}{\sum_{v\prec u}} \sup_{n\geq 1}\diam(\cD_n(v))\right)=0,
\end{equation*}
to get $\sG(\cD_n)\rightarrow \sG(\cD_\infty)$ in the Gromov-Hausdorff topology as $n\rightarrow\infty$. With some appropriate assumptions, we can also endow these metric spaces with measures and get a similar statement in Gromov-Hausdorff-Prokhorov topology. 
We recall the definition and some properties of those topologies in Section~\ref{growing:subsubsec:ghp topology}.

\paragraph{Scaling limit for the generalised Rémy algorithm.}
\begin{figure}
	\centering
	\begin{tabular*}{15cm}{ccc}
		\subfloat[A realisation of $H_5$\label{growing:subfig:a realisation of H5}]{\includegraphics[height=3.5cm]{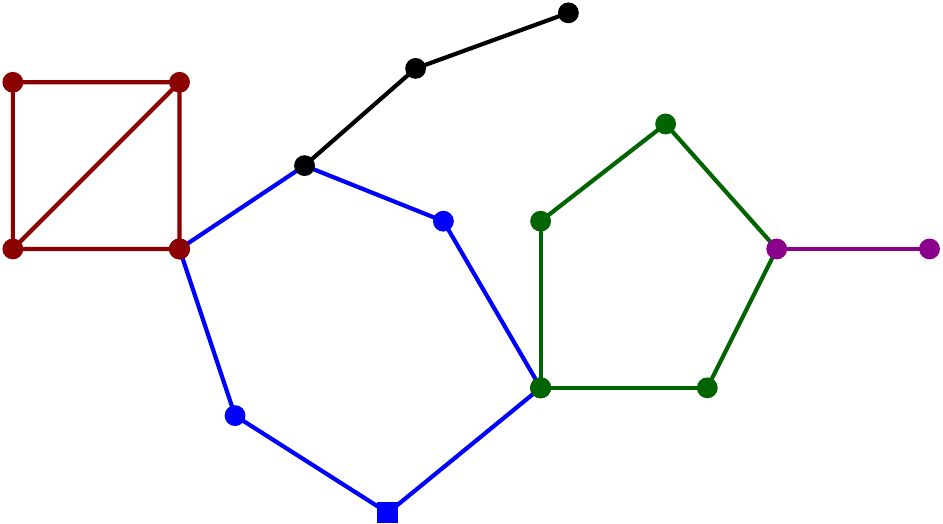}}&
		\subfloat[The tree $\ttP_5$\label{growing:subfig:the tree P5}]{\includegraphics[height=2.5cm]{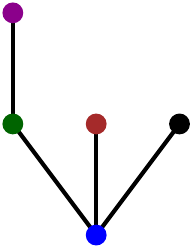}} &
		\subfloat[The graph $H_5$ seen as the gluing of a decoration\label{growing:subfig:the graph H5 as a gluing}]{\includegraphics[height=6cm]{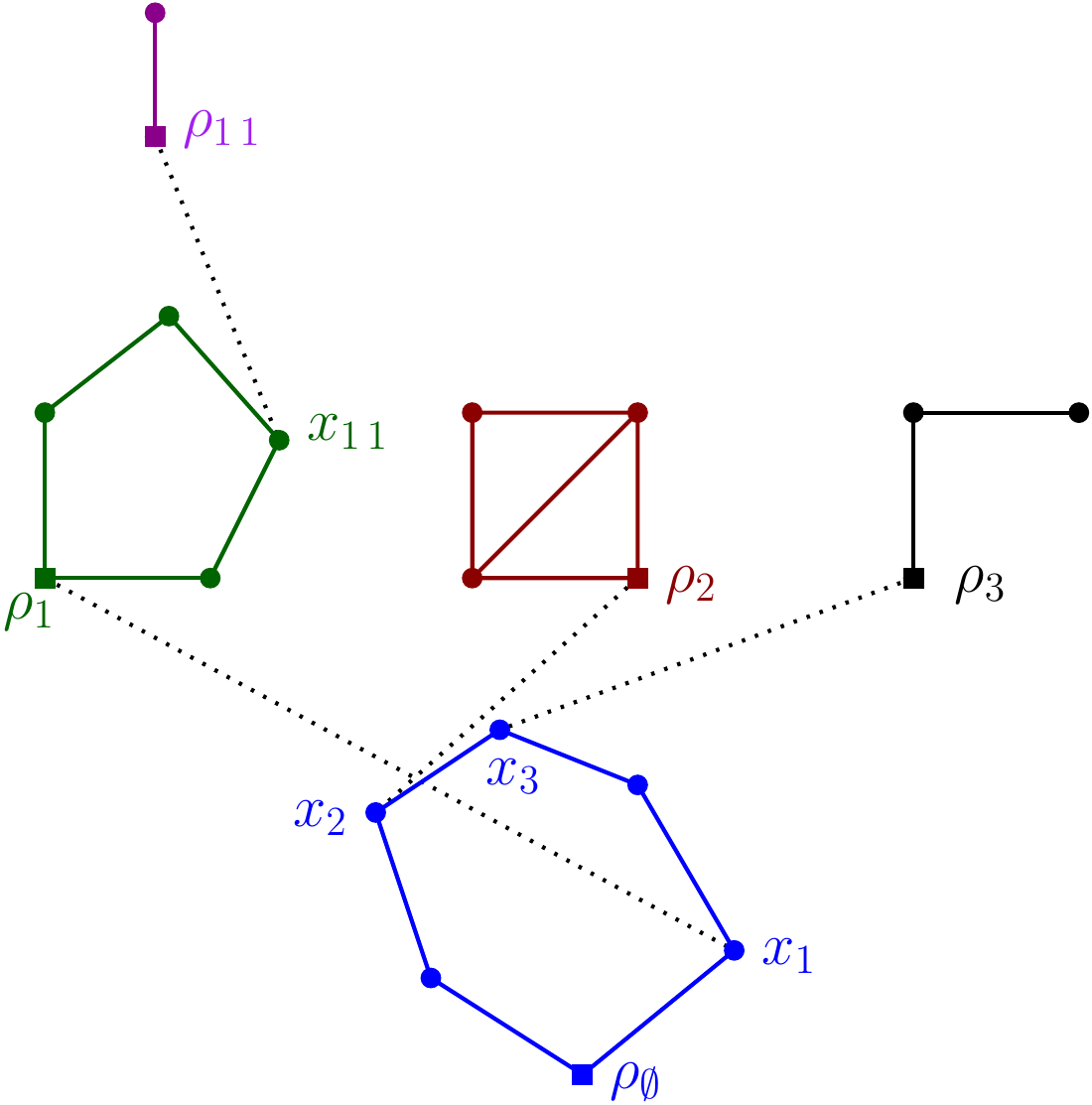}}
	\end{tabular*}
	\caption{Decomposition of $H_5$ as the gluing of some decoration}\label{growing:fig:decomposition le long d'un arbre}
\end{figure}
This will allow us to prove Proposition~\ref{growing:prop:convergence rémy généralisé}.
The idea is to interpret $(H_n)_{n\geq 1}$ in this framework by constructing a sequence of decorations $(\cD_n)_{n\geq 1}$, in such a way that for all $n\geq 1$ the graph $H_n$ seen as a metric space coincides with $\sG(\cD_n)$.
That way, the problem of understanding the whole structure of $H_n$ is decomposed into the easier problem of understanding separately all the $\cD_n(u)$ for all $u\in\bU$.

For this particular construction, this is done in the following way: as pictured in Figure~\ref{growing:fig:exemple graphes gn} we colour the vertices and edges of each graph in the sequence $(G_n)_{n\geq 1}$ with distinct colours for different graphs. 
Then we keep those colours in the construction of the graphs $(H_n)_{n\geq 1}$, and use the rule that every time that an edge is split by the algorithm, the two resulting edges have the same colour as the original edge. See Figure~\ref{growing:subfig:a realisation of H5} for a realisation of $H_5$ using the coloured graphs of Figure~\ref{growing:fig:exemple graphes gn}.
We can naturally couple this construction with that of an increasing sequence of plane trees $(\ttP_n)_{n\geq 1}$, in which every one of the $n$ vertices of $\ttP_n$ corresponds to one of the $n$ colours present in $H_n$: two vertices of $\ttP_n$ are linked by an edge if and only if their corresponding colours are adjacent in the graph $H_n$ (the left-to-right order of children in the plane tree is given by the order of creation of the vertices), see Figure~\ref{growing:subfig:the tree P5}.
The distribution of the sequence $(\ttP_n)_{n\geq 1}$ is that of a preferential attachment tree with sequence of initial fitnesses $\mathbf{a}=(a_n)_{n\geq 1}$, whose definition is recalled in Section~\ref{growing:sec:pa and wrt},
and the construction of the decoration $\cD_n$ is pictured in Figure~\ref{growing:subfig:the graph H5 as a gluing} and simply corresponds to decomposing the graph $H_n$ into $n$ pieces of graph with a single colour, glued along the tree $\ttP_n$.

With this particular construction, each process $(\cD_n(u))_{n\geq 1}$ for a fixed $u\in\bU$ only evolves at times $n$ when the degree of $u$ evolves in the tree $(\ttP_n)_{n\geq 1}$ and stays constant otherwise. Also, at times where the block $\cD_n(u)$ evolves, it does so independently of all the other blocks and follows some simple dynamics. This allows us to study the evolution of the processes $(\cD_n(u))_{n\geq 1}$, including their scaling limit, separately for every $u\in\bU$. 

The fact that the limiting metric space can be described using an iterative gluing construction depends crucially on the fact that the distribution of the trees $(\ttP_n)_{n\geq 1}$ can also be expressed as that of a weighted recursive tree (see Section~\ref{growing:sec:pa and wrt} for a definition) using the random sequence $(\mathsf{m}^\mathbf{a}_n)_{n\geq 1}:=(\mathsf M^\mathbf a_n-\mathsf M^\mathbf a_{n-1})_{n\geq 1}$ to which we referred above in definition of the line-breaking construction.

\paragraph{Two families of continuous distributions on decorations.}
Aside from iterative gluing constructions, a example of which we already mentioned, we define the family of \emph{self-similar} decorations.
Under some assumptions, the distribution of the gluing $\sG(\cD)$ of some random self-similar decoration $\cD$ is the unique fixed point of some contraction in an appropriate space of distribution on metric spaces, in the same spirit as the self-similar random trees of Rembart and Winkel in \cite{rembart_recursive_2018}. 
Some distributions on decorations can be belong to both those families, which is often the case for distributions arising as scaling limits of some natural discrete models.

\subsection{Scope of our results and their relation to previous work}
Let us discuss the results proved in this paper and how they are related to the existing literature.
\paragraph{Subcases of the generalised Rémy's algorithm.}
Proposition~\ref{growing:prop:convergence rémy généralisé} already encompasses several models that were already studied using other methods, when specifying particular sequences of graphs $(G_n)_{n\geq 1}$.
\begin{itemize}
	\item Of course, we recover the convergence for the standard Rémy's algorithm whenever $(G_{n})_{n\geq 1}$ is constant and taken to be a single-edge graph.
	\item When $(G_{n})_{n\geq 1}$ is constant equal to a vertex with a single loop, the model is equivalent to the looptree of the linear preferential attachment tree, and we recover the convergence proved in \cite{curien_scaling_2015}.
	\item In \cite{haas_scaling_2015}, Haas and Stephenson study the case where $G_1$ is the single-edge graph and the sequence $(G_n)_{n\geq 2}$ is constant equal to the star-graph with $k-1$ branches, for $k\geq 2$. They describe the scaling limit as a fragmentation tree, as introduced in \cite{haas_genealogy_2004}. In this case, we improve their convergence which was only in probability and give another construction of the limit.
	\item Let us also cite the work of Ross and Wen \cite{ross_scaling_2018}, whose model (depending on a integer-valued parameter $\ell\geq 2$) is obtained by setting $G_{n}$ to be a single-edge graph if $n-1$ is a multiple of $\ell$, and reduced to a single vertex otherwise. We recover their results.
	
	\item In a recent work of Haas and Stephenson \cite{haas_scaling_2019} the authors also study the case where $(G_n)_{n\geq 1}$ is taken as an i.i.d.\ sequence of rooted trees taken from a finite set. They describe the limit as a multi-type fragmentation tree as introduced in \cite{stephenson_exponential_2018}. Again, our result ensures that the convergence is almost sure in the Gromov-Hausdorff-Prokhorov topology and give another construction of the limit.
\end{itemize}
\paragraph{Other models of growing random graphs.}
Our general method can be applied to various models of growth such as Ford's $\alpha$-model \cite{ford_probabilities_2005}, Marchal's algorithm \cite{marchal_note_2008} or their generalisation the $\alpha-\gamma$-growth \cite{chen_new_2009}, possibly started from an arbitrary graph. The same methods apply also for discrete \emph{looptrees} associated to those models (using an appropriate planar embedding) or to planar preferential attachment trees. Notably:
\begin{itemize}
	\item We improve the convergence \cite{haas_continuum_2008,haas_scaling_2012} of Ford trees and $\alpha$-$\gamma$-trees, from convergence in probability to almost sure convergence, and also prove the convergence of their respective discrete looptrees to continuous limits which can be described as the result of iterative gluing constructions, or as self-similar random metric spaces.
	\item We provide a new iterative gluing construction for $\alpha$-stable trees and $\alpha$-stable components, different from the ones appearing in  \cite{goldschmidt_line_2015,goldschmidt_stable_2018}.
	\item We prove a conjecture of Curien et al. \cite{curien_scaling_2015} for the scaling limits of looptrees of planar preferential attachment trees with offset $\delta$, and describe the limit as an iterative gluing construction with circles. 
\end{itemize}
\subsection{Organisation of the paper}
This paper is organised as follows.

We start in Section \ref{growing:sec:gluing metric spaces along the ulam tree} by developing a framework that allows us to define the gluing of infinitely many metric spaces along the structure of the Ulam tree. We prove Theorem~\ref{growing:thm:continuité du recollement} which ensures that this procedure is continuous in some sense with respect to the blocks that we glue together, as soon as they satisfy some relative compactness property.
Then in Section~\ref{growing:sec:pa and wrt} we recall some properties of affine preferential attachment trees and weighted recursive trees that were proved in the companion paper \cite{senizergues_geometry_2019}, and on which our study of sequences of growing graphs in Section~\ref{growing:sec:applications} strongly relies. 
In Section~\ref{growing:sec:distributions on decorations} we present two families of distributions on decorations, the iterative gluing constructions and the self-similar decorations, which can appear as continuous limits of the discrete distributions that we study.  
We derive some of their properties, in particular we give some sufficient condition for the associated metric space obtained under the gluing map $\sG$ to be compact almost surely.
We also provide examples of random decorations that belong to both families of distributions.
Last, in Section~\ref{growing:sec:applications}, we apply the preceding result to obtain scaling limits of some families of growing random graphs. We first start by proving Theorem~\ref{growing:thm:metatheorem}, which is the general case in which our scaling limit results apply. The rest of the section is devoted to applying this theorem to examples of growing random graphs.

Some definitions and results for classical models that are useful to our proofs are recalled or proved in Appendix~\ref{growing:sec:computations}.

\setcounter{tocdepth}{2}
\tableofcontents

\section{Gluing metric spaces along the Ulam tree}\label{growing:sec:gluing metric spaces along the ulam tree}
In this section, we introduce what we call \emph{decorations} on the Ulam tree, which are a families of infinitely pointed compact metric spaces, indexed by the vertices of the Ulam tree. This structure should be thought as a plan that specifies how to construct a metric space by gluing together all those decorations onto one another, along the structure of the Ulam tree. We then provide sufficient conditions that ensures that the resulting metric space is compact and depends continuously on the decorations in a sense that we make precise. 

\subsection{The Ulam tree}
\paragraph{The completed Ulam tree.}
Recall the definition of the Ulam tree $\bU=\bigcup_{n\geq 0} \N^n$ with $\N=\{1,2,\dots\}$. We also introduce the set $\partial\bU=\N^\N$ to which we refer as the \emph{leaves} of the Ulam tree, which we see as the infinite rays joining the root to infinity and let $\overline{\bU}:=\bU\cup \partial\bU$. On this set, we have a natural genealogical order $\preceq$ defined such that $u\preceq v$ if and only if $u$ is a prefix of $v$. From this order we can define for any $u\in\bU$ the subtree descending from $u$ as the set $T(u):=\enstq{v\in \overline{\bU}}{u\preceq v}$. The collection of sets $\{T(u),\ u\in\bU\}$ and $\{\{u\},\ u\in\bU\}$ generate a topology on $\overline{\bU}$, which can also be generated using an appropriate ultrametric distance. Endowed it this distance the set $\overline{\bU}$ is then a separable and complete metric space. 

In the text, we will consider Borel probability measures on this metric space $\overline\bU$. For those measures, there is a simple characterisation of weak convergence given by the following lemma, see \cite[Lemma~6]{senizergues_geometry_2019}.
\begin{lemma}\label{growing:characterisation convergence measures}
	Let $(\pi_n)_{n\geq 1}$ be a sequence of Borel probability measures on $\overline\bU$. Then $(\pi_n)_{n\geq 1}$ converges
	weakly to a probability measure $\pi$ on $\overline\bU$ if and only if for any $u\in\bU$,
	\begin{align*}
	\pi_n(\{u\})\rightarrow \pi(\{u\}) \quad \text{and} \quad  \pi_n(T(u))\rightarrow \pi(T(u))\qquad \text{as }n\rightarrow\infty. 
	\end{align*}
\end{lemma}

\paragraph{Plane trees as subsets of $\bU$.}
Classically, a plane tree $\tau$ is defined as a finite non-empty subset of $\bU$ such that
\begin{enumerate}
	\item if $v\in\tau$ and $v=ui$ for some $i\in\N$, then $u\in \tau$,  
	\item for all $u\in\tau$, there exists a number in $\N\cup \{0\}$, denoted by $\deg^{+}_\tau(u)$, such that for all $i\in \N$,  $ui\in\tau$ iff $i\leq \deg^{+}_\tau(u)$.
\end{enumerate} 
We denote $\bT$ the set planes trees.

\paragraph{Elements of notation.}
Let us define some pieces of notation.
\begin{itemize}
\item Elements of $\overline{\bU}$ are defined as finite or infinite sequences of integers, which we handle as words on the alphabet $\N$. We usually use the symbols $u$ or $v$ to denote elements of this space. 
\item Sometimes we also use a bold letter $\bi$ to denote a finite or infinite word $\bi =i_1i_2\dots$. In this case, for any integer $k$ smaller than the length of $\bi$ we also write $\bi_k=i_1\dots i_k$ the word truncated to its $k$ first letters.
\item For any two $u,v\in\overline{\bU}$, we write $u\wedge v$ the most recent common ancestor of $u$ and $v$.
\item For any $u\in\bU$, the height of $u$ is the unique number $n$ such that $u\in\N^n$. We denote it $\haut(u)$ when $u$ or sometimes also $\abs{u}$. 
\end{itemize}

\subsection{Decorations on the Ulam tree}\label{growing:subsec:decorations on the ulam tree}
We call any function $f:\bU\rightarrow E$ from the Ulam tree to a space $E$ an \emph{$E$-valued decoration} on the Ulam tree.  

\paragraph{Real-valued decorations.}
As a first example, a function $\ell:\bU\rightarrow\R_+$ is a real-valued decoration on the Ulam tree. As this will be useful later on, we introduce the following terminology. We say that $\ell$ is \emph{non-explosive} if 
\begin{equation}
\underset{\theta \text{ finite}}{\inf_{\theta \in \bT}}\sup_{u\in\bU}\left( \underset{v\notin\theta}{\sum_{v\preceq u}}\ell(v)\right)=0.
\end{equation}

\paragraph{Metric space-valued decorations.}
The main objects studied in this paper are metric-space valued decorations $\cD:\bU\rightarrow\M^{\infty\bullet}$, where the set $\M^{\infty\bullet}$ is the set of non-empty compact metric spaces endowed with an infinite sequence of distinguished points, up to isometry (see below for a proper definition).
More precisely
\begin{align*}
\cD:u\mapsto\cD(u)=\left(D_u,d_u,\rho_u,(x_{ui})_{i\geq 1}\right),
\end{align*}
where $D_u$ is a set, $d_u$ is a distance function of $D_u$, and $\rho_u$ and the $(x_{ui})_{i\geq 1}$ are distinguished points of $D_u$. The point $\rho_u$ is called the \emph{root} of $\cD(u)$, and we call $\cD(u)$ a \emph{block} of the decoration. In all the paper, the word "decoration" always means "$(\M^{\infty\bullet})$-valued decoration", unless specified otherwise. 

Let us define a particular element of $\M^{\infty \bullet}$, which we call the \emph{trivial} or \emph{one-point} space $(\{\star\},0,\star,(\star)_{i\geq 1})$. For any decoration $\cD$, the subset $S\subset\bU$ of elements $u$ for which $\cD(u)$ is not trivial is called the \emph{support} of the decoration $\cD$. 
In the rest of the paper we will often consider decorations that are supported on finite plane trees.

For $a>0$, we will use the notation $a\cdot \cD$ to denote the decoration created from $\cD$ by multiplying all the distances in all the blocks by a factor $a$.

\paragraph{The gluing operation.}
\begin{figure}
	\begin{center}
		\begin{tabular}{cc}
			\subfloat[The path between $\bi$ and $\bj$ in the Ulam tree]{\includegraphics[height=8cm]{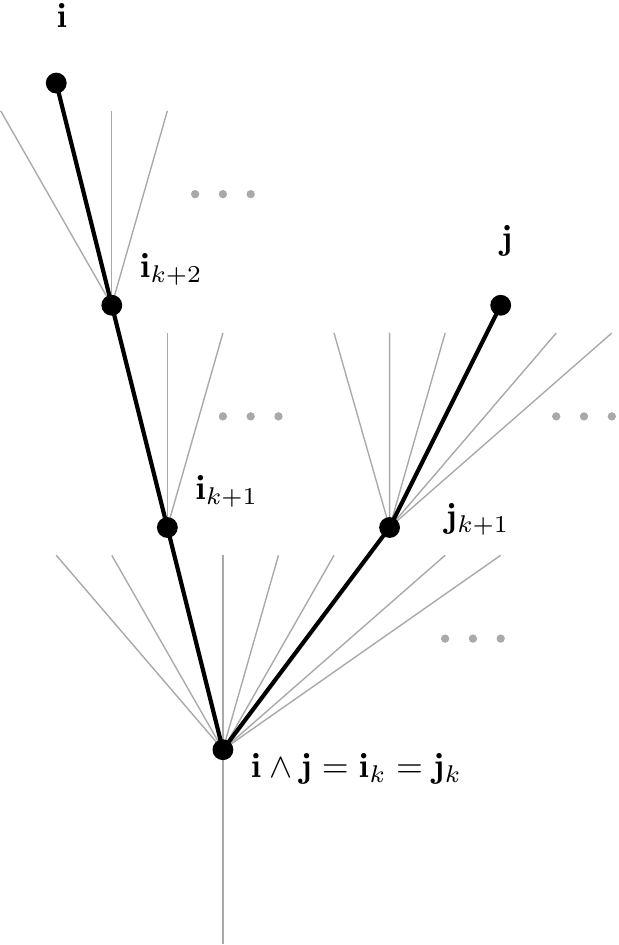}\label{growing:fig:path in ulam tree}} & 			\subfloat[The contribution from every block along the path]{\includegraphics[height=10cm]{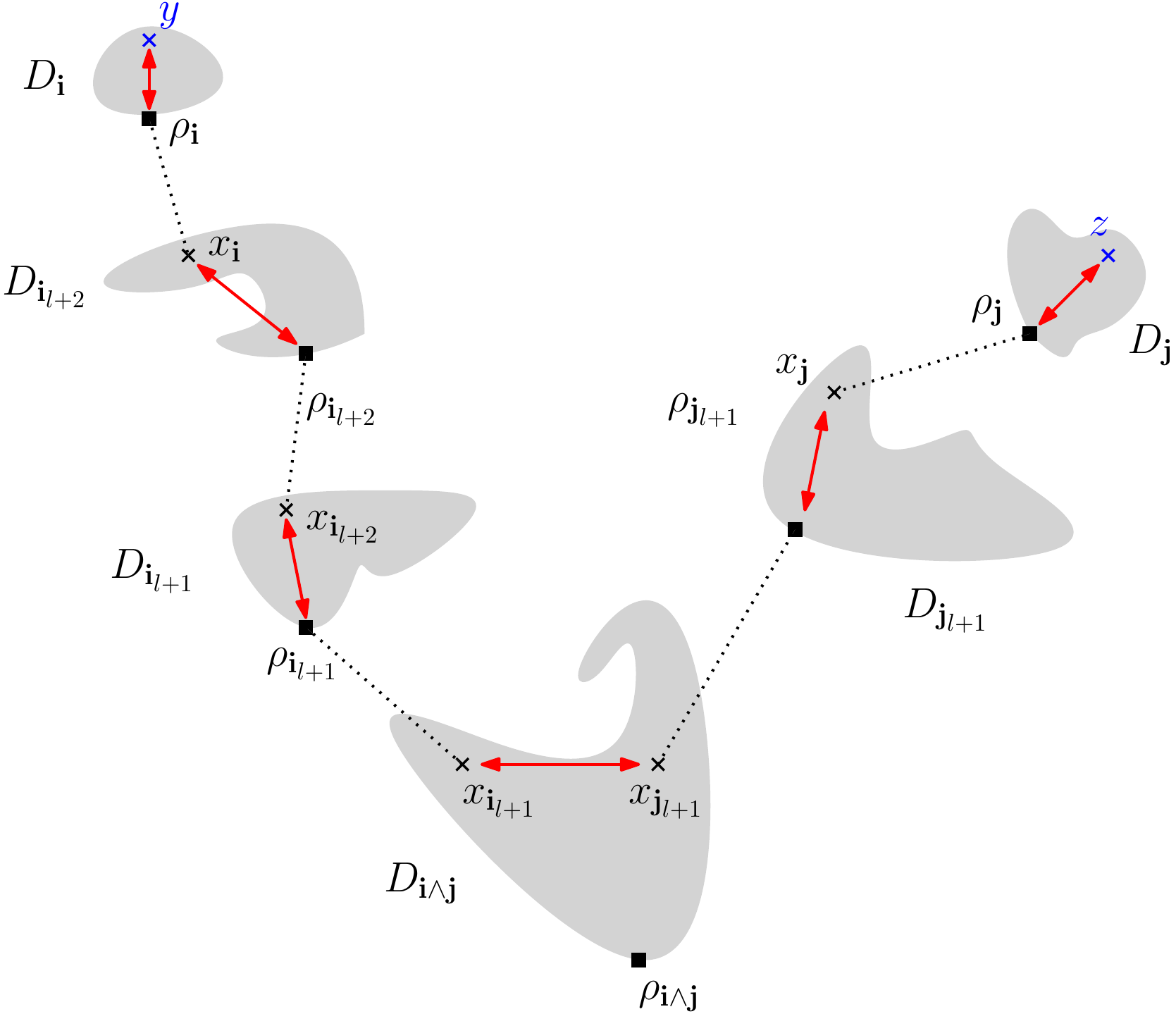}\label{growing:fig:distances in the gluing}}
		\end{tabular}
		\caption{The distance between two points is computed as the sum of the contributions denoted in red, computed using the distance in the corresponding block.}\label{growing:fig:the distance computed by suming} 
	\end{center}
\end{figure}
We define a \emph{gluing operation} $\sG$ on the set of metric space valued decorations $(\M^{\infty\bullet})^\bU$.
For any $\cD=\left(\cD(u)\right)_{u\in\bU}$, we first define a metric space $\sG^*(\cD)$ as
\begin{equation}
\sG^*(\cD)=\left(\bigsqcup_{u\in\bU} D_u\right)/\sim,
\end{equation}
where the equivalence relation $\sim$ is such that for every $u\in \bU$ and $i\in\N$ the root $\rho_{ui}$ of $D_{ui}$ is in relation with the distinguished point $x_{ui}\in D_{u}$. 
The distance $\dist$ on the set $\sG^*(\cD)$ is then the one corresponding to the metric gluing of the blocks along the relation $\sim$, in the sense of \cite{burago_course_2001}.
This distance is defined as follows. For all $\bi=i_1i_2\dots i_n$ and $\bj=j_1j_2\dots j_m$ and points $y\in D_\bi$, $z\in D_\bj$,
\begin{itemize}
\item if $\bi=\bj$ then 
\begin{align*}
\dist(y,z)=\dist(z,y)=d_\bi(y,z),\\
\end{align*}
\item if $\bj\prec \bi$ then 
\begin{align*}
\dist(y,z)=\dist(z,y)=d_\bi(y,x_{\bj_{n+1}})+\sum_{k=n+1}^{m-1}d_{\bj_{k}}(\rho_{\bj_k},x_{\bj_{k+1}})+d_\bj(\rho_\bj,z),
\end{align*}
\item and if $\bi\wedge \bj=\bi_l=\bj_l$ is different from $\bi$ and $\bj$ we let
\end{itemize}
\begin{align*}
	\dist(y,z)=\dist(z,y)	=d_{\bi_k}(x_{\bi_{k+1}},x_{\bj_{k+1}})+&\sum_{k=l+1}^{n-1}d_{\bi_{k}}(\rho_{\bi_k},x_{\bi_{k+1}})+d_\bi(\rho_\bi,z)\\
	&+\sum_{k=l+1}^{m-1}d_{\bj_{k}}(\rho_{\bj_k},x_{\bj_{k+1}})+d_\bj(\rho_\bj,z).
\end{align*}
This last configuration is illustrated in Figure~\ref{growing:fig:the distance computed by suming}.  
We then set
\begin{equation*}
\sG(\cD)=\overline{\sG^*(\cD)},
\end{equation*}
its metric completion for the distance $\dist$. We also let $\sL(\cD)=\sG(\cD)\setminus\sG^*(\cD)$ be its \emph{set of leaves}.

Whenever the associated function $\ell:\bU\rightarrow\R_+$ defined as $u\mapsto \ell(u)=\diam(D_u)$ is non-explosive, it is easy to see that the defined object $\sG(\cD)$ is compact, and it can be approximated by gluing only finitely many blocks of the decoration.

Remark that if $\cD$ is supported on a plane tree $\tau$, then for any $u\in\tau$ the result of the gluing operation does not depend on the distinguished points $(x_{ui})_{i\geq \deg_\tau^+(u)+1}$ of $\cD(u)$ with index greater than $\deg_\tau^+(u)+1$.
\paragraph{Identification of the leaves.}
Suppose that $\cD$ is such that $\sG(\cD)$ is compact. Then there exists a natural map 
\begin{equation}\label{growing:eq:def de iota}
\iota_\cD:\partial \bU\rightarrow \sG(\cD),
\end{equation}
 that maps every leaf of the Ulam-Harris tree to a point of $\sG(\cD)$.
Indeed, for any $\mathbf{i}=i_1i_2\dots \in\partial \bU$, we define 
\begin{align*}
\iota_{\cD}(\mathbf{i})= \lim_{n\rightarrow\infty} x_{\bi_n} \in \sG(\cD),
\end{align*}
and the limit exits because of the compactness of the space. It is then straightforward to see that this map is continuous.

\paragraph{Measure-valued decorations.}
Let $\cD$ be a metric space-valued decoration. 
Suppose that we have a family $\boldsymbol{\nu}:u\mapsto(\nu_u)_{u\in\bU}$ such that $\nu_u$ is a Borel measure on $\cD(u)$, for all $u\in \bU$.
Then we can define a corresponding measure $\nu$ on $\bU$, so that for all $u\in\bU, \ \nu(\{u\})=\nu_u(D_u)$. We define the \emph{support} of $\boldsymbol{\nu}$ as the support of the corresponding measure $\nu$ on $\bU$. 

In this setting, we can define in a natural way a measure on $\sG(\cD)$ by seeing $\sum_{u\in\bU}\nu_u$ as a measure on $\sG(\cD)$, identifying every block as a subspace of $\sG(\cD)$. In this case, we write 
\begin{equation}
\sG\left(\cD,\bnu\right)
\end{equation}
for the corresponding measured metric space. In the case where $\sG(\cD)$ is compact, then the function $\iota_{\cD}:\partial\bU\rightarrow\sG(\cD)$ is well-defined and continuous so that if $\mu$ denotes a measure on $\partial\bU$, then we can consider the push-forward measure $(\iota_{\cD})_*\mu$ on $\sG(\cD)$. In this case we write
\begin{equation}\label{growing:eq:def GDmu}
	\sG(\cD,\mu)=(\sG(\cD),(\iota_{\cD})_*\mu),
\end{equation}
which is a measured metric space.
We can now state the main result of Section~\ref{growing:sec:gluing metric spaces along the ulam tree}.
\begin{theorem}\label{growing:thm:continuité du recollement}
	Suppose that $\left(\cD_n\right)_{n\geq1}$ is a sequence of decorations and that there exists a decoration $\cD_\infty$ such that for every $u\in\bU$,
	\begin{align*}
	\cD_n(u)\underset{n\rightarrow \infty}{\longrightarrow}\cD_\infty(u),
	\end{align*}
	for the infinitely pointed Gromov-Hausdorff-Prokhorov topology and such that the associated real-valued decoration $(\ell \colon u\mapsto \sup_{n\geq 1}\diam(\cD_n(u)))$ is non-explosive.
	Then, the following properties hold.
	\begin{enumerate}
		\item\label{growing:it:convergence gh} We have the convergence 
		\begin{equation*}
		\sG(\cD_n)\longrightarrow\sG\left(\cD_\infty\right) \quad \text{as  $n\rightarrow\infty$ for the Gromov-Hausdorff topology.}
		\end{equation*}
		\item \label{growing:it:convergence ghp} Furthermore, suppose that for all $n\geq 1$, we have $\bnu_n=(\nu_{u,n})_{u\in \bU}$, measures over $\cD_n$ such that the corresponding measures $(\nu_n)_{n\geq 1}$ are probabilities on $\bU$ and converge weakly in $\overline{\bU}$ as $n\rightarrow\infty$ to some probability measure $\nu_\infty$ that only charges $\partial\bU$, then we have the convergence
		\begin{equation*}
		\sG(\cD_n,\bnu_n)\longrightarrow\sG(\cD_\infty,\nu_\infty) \quad \text{as $n\rightarrow\infty$},
		\end{equation*} 
		 for the Gromov-Hausdorff-Prokhorov topology.
	\end{enumerate}
\end{theorem}
The first point of this theorem states that the convergence of a global structure defined as $\sG(\cD_n)$, for some sequence $\cD_n$ of decorations, can be obtained by proving the convergence of every $\cD(u)$, for all $u\in \bU$ (convergence of finite dimensional marginals) with the additional assumption that they satisfy some relative compactness property which is here expressed as the non-explosion condition. 
The second point ensures that if we add measures on our decorations and that those measures converge nicely, then we can improve our convergence to Gromov-Hausdorff-Prokhorov topology on measured metric spaces. We only treat the case where the measure gets "pushed to the leaves" because only this case arises in our applications. A more general statement where $\nu$ is not carried on $\partial\bU$ could be proven under the appropriate assumptions. 

\subsection{Some formal topological arguments}
The aim of this section is to justify and define properly the construction described in the preceding section, in a way that can be adapted to random decorations without any measurability issue. This section is rather technical and can be skipped at first reading.
We begin by recalling some topological facts about the Urysohn universal space, and the so-called Hausdorff/Gromov-Hausdorff/Gromov-Hausdorff-Prokhorov topologies.

\subsubsection{Urysohn space and Gromov-Hausdorff-Prokhorov topology}\label{growing:subsubsec:ghp topology}

\paragraph{Urysohn universal space.}
Let us consider $(\cU,\delta)$ the Urysohn space, and fix a point $\ast \in \cU$. The space $\cU$ is defined as the only Polish metric space (up to isometry) which has the following extension property (see \cite{husek_urysohn_2008} for constructions and basic properties of $\cU$): given any finite metric space $X$, and any point $x\in X$, any isometry from $X\setminus\{x\}$ to $\cU$ can be extended to an isometry from $X$ to $\cU$.
This property ensures in particular that any separable metric space can be isometrically embedded into $\cU$. In what follows we will use the fact that if $(K,\dist,\rho)$ is a rooted compact metric space, there exists an isometric embedding of $K$ into $U$ such that $\rho$ is mapped to $\ast$. It also has a very useful property called compact homogeneity (see \cite[Corollary 1.2]{melleray_geometry_2007}), which ensures that any isometry $\varphi$ between two compact subsets $K$ and $L$ of $\cU$ can be extended to the whole space $\cU$, meaning that there exists a global isometry $\phi$ such that $\varphi$ is just the restriction $\varphi=\phi_{\vert_K}$.  

\paragraph{Hausdorff distance, Lévy-Prokhorov distance.}
For any two compact subsets $A$ and $B$ of the same metric space $(E,d)$, we can define their Hausdorff distance as 
\begin{align*}
\dist_{\mathrm{H}}^E(A,B)=\inf\enstq{\epsilon>0}{A\subset B^{(\epsilon)}, \quad B \subset A^{(\epsilon)}},
\end{align*}
where $A^{(\epsilon)}$ and $B^{(\epsilon)}$ are the $\epsilon$-fattening of the corresponding sets. We denote $\cP(E)$ the set of Borel probability measures on $E$. For any two probability measures $\mu,\nu \in \cP(E)$, we can define their Lévy-Prokhorov distance as
\begin{align*}
\dist_{\mathrm{LP}}^E(\mu, \nu)=\inf \enstq{\epsilon>0}{\forall F \in \mathcal{B}(E), \mu(F)\leq \nu(F^{(\epsilon)})+\epsilon \text{ and } \nu(F)\leq \mu(F^{(\epsilon)})+\epsilon}.
\end{align*}
Whenever the space $E$ is the Urysohn space, we drop the index $E$ in the notation for those distances.
\paragraph{Infinitely pointed Gromov-Hausdorff topology.}
We write $\M^{k\bullet}$ for the space of all equivalence classes of
$(k+1)$-pointed measure metric spaces. We can define the 
Gromov-Hausdorff distance on $\M^{k\bullet}$ by
\begin{align*}
&{\mathrm{d_{GH}}}^{(k)}((X,d,\rho_0,(\rho_1,\ldots,\rho_k)),(X',d',\rho_0,(\rho'_1,\ldots,\rho'_k)))\\
&= \inf_{\phi:X\to \cU,\phi':X'\to \cU} \Big\{ d_\mathrm{H}(\phi(X),\phi'(X)) \vee \max_{0\leq i\leq k} \delta(\phi(\rho_i),\phi'(\rho'_i))\Big\},
\end{align*}
where, as previously, the infimum is over all isometric embeddings $\phi$ and $\phi'$ of $X$ and $X'$ into the Urysohn space $\cU$. %
We write $\M^{\infty\bullet}$ for the space of all (equivalence classes of)
$\infty$-pointed measured metric spaces. We can define the infinitely pointed
Gromov-Hausdorff distance on $\M^{\infty\bullet}$ by
\begin{align*}
&{\mathrm{d_{GH}}}^{(\infty)}((X,d,\rho_0,(\rho_i)_{i\geq 1}),(X',d',\rho_0,(\rho'_i)_{i\geq 1}))\\
&= \sum_{k=1}^{\infty}\frac{1}{2^k}{\mathrm{d_{GH}}}^{(k)}((X,d,\rho_0,(\rho_1,\ldots,\rho_k)),(X',d',\rho_0,(\rho'_1,\ldots,\rho'_k)))\\
(&\leq (\diam X + \diam X')<\infty)
\end{align*}
By abuse of notation, we will also consider (equivalence classes of) finitely pointed compact metric spaces $(X,\dist,\rho,(x_i)_{1\leq i\leq k})$ as elements of $\M^{\infty\bullet}$, by arbitrarily extending the sequence $(x_i)$ with $x_i=\rho$ for all $i\geq k+1$.
\paragraph{Infinitely pointed Gromov-Hausdorff-Prokhorov topology.}
In some of our applications, we work on $\K^{\infty \bullet}$, which is the corresponding space for elements of $\M^{\infty\bullet}$ endowed with a Borel probability measure. In the same way as before, elements of $\K^{\infty \bullet}$ are $5$-tuples $(X,\dist,\rho,(x_i)_{i\geq 1},\mu)$, where $(X,\dist,\rho,(x_i)_{i\geq 1})\in \K^{\infty \bullet}$ and $\mu$ is a finite Borel measure on $X$. Again we set 
\begin{align*}
&{\mathrm{d_{GHP}}}^{(k)}((X,d,\rho_0,(\rho_1,\ldots,\rho_k),\mu),(X',d',\rho_0,(\rho'_1,\ldots,\rho'_k),\mu'))\\
&= \inf_{\phi:X\to \cU,\phi':X'\to \cU} \Big\{ d_\mathrm{H}(\phi(X),\phi'(X)) \vee d_\mathrm{LP}((\phi)_*\mu,(\phi')_*\mu') \vee \max_{0\leq i\leq k} d(\phi(\rho_i),\phi'(\rho'_i))\Big\},
\end{align*}
and
\begin{align*}
&{\mathrm{d_{GHP}}}^{(\infty)}((X,d,\rho_0,(\rho_i)_{i\geq 1},\mu),(X',d',\rho_0,(\rho'_i)_{i\geq 1},\mu'))\\
&= \sum_{k=1}^{\infty}\frac{1}{2^k}{\mathrm{d_{GHP}}}^{(k)}((X,d,\rho_0,(\rho_1,\ldots,\rho_k),\mu),(X',d',\rho_0,(\rho'_1,\ldots,\rho'_k),\mu')).
\end{align*}

\subsubsection{Construction in the appropriate ambient space}
In order to ease the definition of our objects and avoid some measurability issues that may arise when working with abstract equivalence classes of metric spaces, we define a way of only dealing with some particular representatives of those equivalence classes that are compact subsets of the set $\cU$. 
For that matter we define $\K^{\infty\bullet}(\cU)$ the counterpart of $\K^{\infty\bullet}$,
\[\K^{\infty \bullet}(\cU):=\enstq{(K,\delta_{\vert_K},\ast,(\rho_i)_{i\geq 1},\mu)}{\ast\in K\subset \cU, \ K \text{ compact}, \forall i\geq 1, \rho_i\in K, \ \mu\in \cP(\cU), \ \supp(\mu)\subset K},\]
where $\delta_{\vert_K}$ is the distance on $\cU$ restricted to the subset $K$. We set accordingly, 
\begin{align*}
{\mathrm{d_{HP}}}^{(k)}((K,(\rho_1,\ldots,\rho_k),\mu),(K',(\rho'_1,\ldots,\rho'_k),\mu'))
&=d_\mathrm{H}(K,K') \vee d_\mathrm{LP}(\mu,\mu') \vee \max_{1\leq i\leq k} d(\rho_i,\rho'_i),
\end{align*}
and
\begin{align*}
&{\mathrm{d_{HP}}}^{(\infty)}(K,\delta_{\vert_K},\ast,(\rho_i)_{i\geq 1},\mu),(K',\delta_{\vert_{K'}},\ast,(\rho_i')_{i\geq 1},\mu'))\\
&= \sum_{k=1}^{\infty}\frac{1}{2^k}{\mathrm{d_{HP}}}^{(k)}((K,(\rho_1,\ldots,\rho_k),\mu),(K',(\rho'_1,\ldots,\rho'_k),\mu')).
\end{align*}
We define the projection map $\pi:\K^{\infty \bullet}(\cU)\longrightarrow\K^{\infty \bullet}$, such that 
\begin{align*}
\pi((K,\delta_{\vert_K},\ast,(\rho_i)_{i\geq 1},\mu))= \left[(K,\delta_{\vert_K},\ast,(\rho_i)_{i\geq 1},\mu)\right],
\end{align*}
the corresponding equivalence class in $\K^{\infty \bullet}$. This map is surjective by the properties of Urysohn space and continuous because it is obviously $1$-Lipschitz. Using the surjectivity, we know that we can lift any deterministic element of $\K^{\infty \bullet}$ to an element $\K^{\infty\bullet}(\cU)$. %

Actually, we are going to deal with random variables with values in the space $\K^{\infty\bullet}$ and we want to ensure that we can consider versions of those random variables with values in $\K^{\infty\bullet}(\cU)$. In fact, noticing that both sets are Polish spaces, we can use a theorem of measure theory from \cite{lubin_extensions_1974} which ensures that every probability distribution $\tau$ on $\K^{\infty \bullet}$ can be lifted to a probability measure $\sigma$ on $\K^{\infty \bullet}(\cU)$, such that its corresponding push-forward measure by the projection $\pi_*\sigma$ is equal to the probability $\tau$. Hence, whenever we consider a random variable with values in $\K^{\infty \bullet}$, we can always work with a version of our random variable that is embedded in the space $\cU$, and whose root coincides with $\ast$. The same line of reasoning can be made with $\M^{\infty \bullet}$.

From now on, we work with decorations $\cD\in \left(\K^{\infty \bullet}(\cU)\right)^\bU$ by taking a representative for every one of the blocks of the decoration. 

\paragraph{Construction embedded in a space.}
We introduce the following space, in which we will be able to define a representative of the space $\sG(\cD)$ for any decoration $\cD$. 
\[\ell^1(\cU,\bU,\ast):=\enstq{(y_u)_{u\in\bU}\in \cU^{\bU}}{\sum_{u\in\bU} \delta(y_u,\ast)<+\infty }.\]
We endow $\ell^1(\cU,\bU,\ast)$ with the distance $\dist((y_u)_{u\in\bU},(z_u)_{u\in\bU})=\sum_{u\in\bU} \delta(y_u,z_u)$,
which makes it a Polish space. 
\begin{remark}\label{growing:rem:isometries of l1}
	If, for each $u\in\bU$, we are given an isometry $\phi_u:\cU\rightarrow\cU$ such that $ \phi_u(\ast)=\ast$, then we can introduce
	\begin{align*}
	\phi:= \prod_{u\in\bU}\phi_u: \ &\ell^1(\cU,\bU,\ast)\rightarrow \ell^1(\cU,\bU,\ast)\\
	&(y_u)_{u\in\bU}\mapsto \left(\phi_u(y_u)\right)_{u\in\bU}
	\end{align*}
	and $\phi$ is an isometry of the space $\ell^1(\cU,\bU,\ast)$.
\end{remark}
For each $u\in\bU$, we consider a representative of the block $\left(D_u,d_u,\rho_u,(x_{ui})_{i\geq 1}\right)$ that belongs to $\M^{\infty \bullet}(\cU)$, meaning that we see $D_u$ as a subset of $\cU$ and,
\begin{align*}
\left(D_u,d_u,\rho_u,(x_{ui})_{i\geq 1}\right)=\left(D_u,\delta_{\vert_{D_u}},\ast,(x_{ui})_{i\geq 1}\right)
\end{align*}
Then the gluing operation is defined in the following way. Let $\bi=i_1i_2...i_n\in \bU$. For any such $\bi\in\bU$, we define,
\begin{equation*}
\tilde{D_\mathbf{i}}=\enstq{(y_{u})_{u\in\bU}}{y_{\emptyset}=x_{\bi_1}, y_{\bi_1}=x_{\bi_2},\dots, y_{\bi_{n-1}}=x_{\bi_n}, y_{\mathbf{i}}\in D_{\bi}, \text{ and } \forall u\npreceq \bi,  y_{u}=\ast}
\end{equation*} 
Remark that each of the subsets $\tilde{D_\mathbf{i}}$ is isometric to the corresponding bloc $D_\mathbf{i}$.
Then we consider
\begin{equation}
\sG^*(\cD)=\bigcup_{\bi\in \bU}\tilde{D_\mathbf{i}}
\end{equation}
The structure $\sG(\cD)$ is then defined as the closure of $\sG^*(\cD)$ in the space $\ell^1(\cU,\bU,\ast)$. Thanks to Remark~\ref{growing:rem:isometries of l1}, the resulting space (up to isometry) does not depend on the choice of representative for the different decorations. 

For convenience, for any plane tree $\theta$ we also introduce $\sG(\theta,\cD)$ the metric space obtained by only gluing the decorations that are indexed by the vertices in $\theta$, i.e.
\begin{equation}
\sG(\theta,\cD):= \bigcup_{\bi\in \theta} \tilde{D}_\bi
\end{equation} We do not need to complete it since it is already compact, as a union of a finite number of compact metric spaces.

\paragraph{Identification of the leaves.}
Suppose that $\cD$ is such that $\sG(\cD)$ is compact. Then, in this setting, the map $\iota_{\cD}:\partial\bU\rightarrow \sG(\cD)$ defined in \eqref{growing:eq:def de iota} has the following form: for any $\mathbf{i}=i_1i_2\dots \in\partial\bU$, 
\begin{align*}
\iota_{\cD}(\mathbf{i})= (y_u)_{u\in\bU} \quad \text{with } y_{\bi_n}&=x_{\bi_{n+1}} \quad \text{ for all } n\geq 0,\\
y_u&=\ast \quad \text{whenever} \quad u\nprec\bi.
\end{align*}
\subsection{Proof of Theorem \ref{growing:thm:continuité du recollement}}
Before proving the theorem, let us state a lemma that ensures that the gluing operation is continuous when considering a finite number of decorations. 
\begin{lemma}\label{growing:lem:continuité finie dimensionnelle}
	For any $\theta$ finite plane tree, and $\cD$ and $\cD'$ decorations, we have
	\[
	\mathrm{d_{GH}}\left(\sG\left(\theta,\cD\right),\sG\left(\theta,\cD'\right)\right)\leq 2\cdot \sum_{u\in\theta} \mathrm{d_{GH}}^{(\deg^+_\theta(u))}\left(\cD(u),\cD'(u)\right).
	\]
\end{lemma}
\begin{proof}
	For all $u\in\theta$, and thanks to the compact homogeneity of $\cU$, we can find an isometry $\phi_u:\cU\rightarrow\cU$ such that $\phi_u(\ast)=\ast$ and 
	\begin{equation}\label{growing:eq:phiu isometry}
	\mathrm{d_H}(\phi_u\left(D'_u)\right),D_u)\vee \max_{1\leq i\leq \deg^+_\theta(u)} \delta(\phi_u(x_{ui}'),x_{ui})\leq 2 \mathrm{d_{GH}}^{(\deg^+_\theta(u))}\left(\cD(u),\cD'(u)\right).
	\end{equation}
	Then let $\phi_u=\mathrm{id}_\cU,$ for every $u\notin \theta$, and let $\phi = \prod_{u\in\bU}\phi_u$ be the corresponding isometry of $\ell^1(\cU,\bU,\ast)$. Then let us show that we control the Hausdorff distance between
	\begin{align*}
	\sG(\theta,\cD)&=\bigcup_{\bi\in \theta}\enstq{(y_{u})_{u\in\bU}}{y_{\emptyset}=x_{\bi_1},\  y_{\bi_1}=x_{\bi_2}, \ \dots,\ y_{\bi_{n-2}}=x_{\bi_{n-1}},\ y_{\bi_{n-1}}=x_{\bi}, \text{ and } \forall u\npreceq \bi,  y_{u}=\ast},\\
	&=\bigcup_{\bi\in \theta} \tilde{D_\bi},
	\end{align*}
	and 
	\begin{align*}
	\phi(\sG(\theta,\cD'))&=\bigcup_{\bi\in \theta}\enstq{(y_{u})_{u\in\bU}}{y_{\emptyset}=\phi_{\emptyset}(x_{\bi_1}'),\dots, y_{\bi_{n-1}}=\phi_{\bi_{n-1}}(x_{\bi}'), y_{\mathbf{i}}\in \phi_{\bi}(D_{\bi}'), \text{ and } \forall u\nprec \bi,  y_{u}=\ast}\\
	&=\bigcup_{\bi\in \theta} \phi\left(\tilde{D_\bi'}\right)
	\end{align*}
	Now for any $\bi=i_1i_2\dots i_n\in\theta$, any $y=(y_u)_{u\in\bU}\in \tilde{D}_\bi$ and $z=(z_u)_{u\in\bU}\in \phi\left(\tilde{D_\bi'}\right)$, we can write
	\begin{align*}
	\dist(y,z)&=\delta(y_\bi,z_\bi)+\sum_{\ell=1}^{n} \delta(x_{\bi_\ell},\phi_{\bi_{\ell-1}}(x_{\bi_{\ell}}') ),
	\end{align*}
	with $y_\bi\in D_\bi$ and $z_\bi\in \phi(D_\bi')$.
	Now using equation \eqref{growing:eq:phiu isometry}, we get that 
	\begin{align}\label{growing:eq:distance de hausdorff bloc par bloc}
	\mathrm{d_H}\left(\tilde{D}_\bi,\phi\left(\tilde{D_\bi'}\right)\right)&\leq \mathrm{d_H}(D_\bi,\phi_\bi(D_\bi'))+\sum_{\ell=1}^{n} \delta(x_{\bi_\ell},\phi_{\bi_{\ell-1}}(x_{\bi_{\ell}}') ) \notag \\
	&\leq 2\cdot \sum_{u\in\theta} \mathrm{d_{GH}}^{(\deg^+_\theta(u))}\left(\cD(u),\cD'(u)\right).
	\end{align}
	The last inequality is true for any $\bi\in\theta$, hence taking a union yields,
	\begin{align*}
	\mathrm{d_H}\left(\bigcup_{\bi\in\theta}\tilde{D}_\bi,\bigcup_{\bi\in\theta}\phi\left(\tilde{D_\bi'}\right)\right)=\mathrm{d_H}\left(\sG(\theta,\cD),\phi(\sG(\theta,\cD'))\right)\leq 2\cdot \sum_{u\in\theta} \mathrm{d_{GH}}^{(\deg^+_\theta(u))}\left(\cD(u),\cD'(u)\right),
	\end{align*}
	which finishes to prove the lemma.
\end{proof}

\begin{proof}[Proof of Theorem \ref{growing:thm:continuité du recollement}]
	Let $n\geq 1$ be an integer and $\theta$ a finite plane tree and $y=(y_u)_{u\in\bU}\in \sG(\cD_n)$. From our construction of $\sG(\cD_n)$, we know that the indices $v$ for which $y_v\neq \ast$ are all contained in an infinite ray in $\bU$, meaning that there exists $u\in \partial\bU$ such that for all $v\nprec u, y_v=\ast$. Now we can check that
	\begin{align*}
	\dist(y,\sG(\theta,\cD_n))=\underset{v\notin\theta}{\sum_{v\prec u}}\delta(y_v,\ast)
	&\leq \underset{v\notin\theta}{\sum_{v\prec u}}\sup_{n\geq 1}\diam(\cD_n(v)),\\
	&\leq \sup_{u\in\bU}\underset{v\notin\theta}{\sum_{v\prec u}}\sup_{n\geq 1}\diam(\cD_n(v)).
	\end{align*}
	Since it holds for any $y\in \sG(\cD_n)$ and the bound on the right-hand side is uniform for all such $y$, we have
	\begin{equation*}
	\mathrm{d_H}\left(\sG(\theta,\cD_n),\sG(\cD_n)\right)	\leq \sup_{u\in\bU}\underset{v\notin\theta}{\sum_{v\prec u}}\sup_{n\geq 1}\diam(\cD_n(v)).
	\end{equation*}
	Now we can write
	\begin{align*}
	\mathrm{d_{GH}}\left(\sG(\cD_n),\sG(\cD)\right)&\leq \mathrm{d_{GH}}\left(\sG(\cD_n),\sG(\theta,\cD_n)\right)+\mathrm{d_{GH}}\left(\sG(\theta,\cD_n),\sG(\theta,\cD)\right)+\mathrm{d_{GH}}\left(\sG(\theta,\cD),\sG(\cD)\right)\\
	&\leq 2\sup_{u\in\bU}\underset{v\notin\theta}{\sum_{v\prec u}}\sup_{n\geq 1}\diam(\cD_n(v))+\mathrm{d_{GH}}\left(\sG(\theta,\cD_n),\sG(\theta,\cD)\right)
	\end{align*}
	Now using the non-explosion of the function $(u\mapsto \sup_{n\geq 1}\diam(\cD_n(u)))$ we can make the first term as small as we want by taking the appropriate $\theta$, and when $\theta$ is fixed, the second term vanishes as $n\rightarrow\infty$ thanks to Lemma \ref{growing:lem:continuité finie dimensionnelle}. This finishes the proof of \ref{growing:it:convergence gh}.
	
	Now let us prove point \ref{growing:it:convergence ghp}. For simplicity, we write $\mu_n=\sum_{u\in\bU}\nu_{u,n}$ and also $\mu_\infty=(\iota_{\cD_\infty})_*\nu_\infty$. Let $\epsilon>0$. From the non-explosion condition we know that we can find a plane tree $\theta$ such that 
	\begin{equation}\label{growing:eq:relative compacity appliqué dans la preuve}\sup_{u\in\bU}\left( \underset{v\notin \theta}{\sum_{v\prec u}} \sup_{n\geq 1} \diam(\cD_n(v))\right)<\epsilon.
	\end{equation}
	Now, we construct another finite plane tree $\theta'$, such that $\theta\subset\theta'$, by adding only children of vertices of $\theta$. We do so in such a way that:
	\begin{equation*}
	\sum_{v\in\theta'\setminus \theta}\nu_\infty\left(T(v)\right)\geq 1-\epsilon/2,
	\end{equation*}
	as it is always possible since, if we construct $\theta'$ by adding every children of every vertices of $\theta$, the last sum would be $1$. Remark that from \eqref{growing:eq:relative compacity appliqué dans la preuve}, for any $v\in\theta'\setminus\theta$ and any $n\geq 1$, we have $\diam\left(\cD_n(v)\right)<\epsilon$.
	
	Introduce the projection $p_{\theta'}:\ell^1(\cU,\bU,\ast)\rightarrow \ell^1(\cU,\bU,\ast)$, such that for any $(y_u)_{u\in\bU}$, the image $(z_u)_{u\in\bU}=p_{\theta'}\left((y_u)_{u\in\bU}\right)$ is such that $z_u=y_u$ for any $u \in \theta'$ and $z_u=\ast$ otherwise.
	Using \eqref{growing:eq:relative compacity appliqué dans la preuve}, we can check that for any $n\geq 1$ and for any $y\in\sG(\cD_n)$, we have
	\begin{equation*}
	\dist\left(p_{\theta'}(y),y\right)<\epsilon.
	\end{equation*}
	This observation suffices to show that for any $n\in \N\cup\{\infty\}$,
	\begin{equation*}
	\mathrm{d_{GHP}}\left(\left(\sG(\cD_n),\mu_n\right), \left(p_{\theta'}\left(\sG(\cD_n)\right),(p_{\theta'})_*\mu_n\right)\right)<\epsilon.
	\end{equation*}
	Then,
	\begin{align*}
	\mathrm{d_{GHP}}\left(\left(\sG(\cD_n),\mu_n\right),\left(\sG(\cD_\infty),\mu_\infty\right)\right)&\leq 
	\mathrm{d_{GHP}}\left(\left(\sG(\cD_n),\mu_n\right), \left(p_{\theta'}\left(\sG(\cD_n)\right),(p_{\theta'})_*\mu_n\right)\right)\\
	&+ \mathrm{d_{GHP}}\left(\left(\sG(\cD_\infty),\mu_\infty\right), \left(p_{\theta'}\left(\sG(\cD_\infty)\right),(p_{\theta'})_*\mu_\infty\right)\right)\\
	&+
	\mathrm{d_{GHP}}\left(\left(p_{\theta'}\left(\sG(\cD_n)\right),(p_{\theta'})_*\mu_n\right),\left(p_{\theta'}\left(\sG(\cD_\infty)\right),(p_{\theta'})_*\mu_\infty\right)\right).
	\end{align*}
	The first two terms of the right-hand side are smaller than $\epsilon$ from what precedes, we only have to prove that the last one is also small whenever $n$ is large enough. Remark that for any $\cD$, $p_{\theta'}\left(\sG(\cD)\right)=\sG(\theta',\cD)$. Let us fix $n\geq 1$ large enough such that 
	\begin{align*}
	2\cdot \sum_{u\in\theta'} \mathrm{d_{GH}}^{(\deg_{\theta'}(u ))}\left(\cD_n(u),\cD_\infty(u)\right)<\epsilon,
	\end{align*}
	and
	\begin{align}\label{growing:eq:couplage mun mu}
	\sum_{v\in\theta'\setminus \theta} \abs{\nu_n(T(v))-\nu_\infty(T(v))}<\epsilon.
	\end{align}
	From \eqref{growing:eq:distance de hausdorff bloc par bloc} in the proof of Lemma \ref{growing:lem:continuité finie dimensionnelle}, we can find an isometry $\phi$ such that for all $\bi\in \theta'$,
	\begin{align}\label{growing:eq:distance de hausdorff blocks correspondants}
	\mathrm{d_H}\left(\tilde{D}_{\infty,\bi},\phi\left(\tilde{D}_{n,\bi}\right)\right)<\epsilon. 
	\end{align}
	Now because of \eqref{growing:eq:couplage mun mu}, we know that we can find a coupling $(X_n,X_\infty)$ of random variables with values in $\overline{\bU}$ having distribution $\nu_n$ and $\nu_\infty$, such that with probability $>1-\epsilon$, they both fall in the same $T(v)$ for $v\in\theta'\setminus\theta$. From this coupling, we can construct another one between $(Y_n,Y_\infty)$ of random variables on respectively $\sG(\theta',\cD_n)$ and $\sG(\theta',\cD_\infty)$ such that one has distribution $(p_{\theta'})_*\mu_n$ and the other $(p_{\theta'})_*\mu_\infty$ and such that the probability that there exists $v\in\theta'\setminus\theta$ such that $Y_n\in \tilde{D}_{n,v}$ and $Y_\infty\in \tilde{D}_{\infty,v}$ is greater than $1-\epsilon$. 
	Using this plus \eqref{growing:eq:distance de hausdorff blocks correspondants} shows that the couple $(Y_n,\phi(Y_\infty))$ is at distance at most $\epsilon$ with probability at least $1-\epsilon$. This shows that the Lévy-Prokhorov distance between $(p_{\theta'})_*\mu_\infty$ and $\phi_*((p_{\theta'})_*\mu_n)$ is smaller than $\epsilon$. In this end, we just showed that
	\begin{align*}
	\mathrm{d_{GHP}}\left(\left(p_{\theta'}\left(\sG(\cD_n)\right),(p_{\theta'})_*\mu_n\right),\left(p_{\theta'}\left(\sG(\cD_\infty)\right),(p_{\theta'})_*\mu_\infty\right)\right)<\epsilon,
	\end{align*}
	which finishes the proof of the proposition.
\end{proof}

\subsection{Sufficient condition for non-explosion}
Let us finish this section by proving a result that ensures non-explosion for some type of real-valued decorations. 
Let $(x_n)_{n\geq 1}$ be a sequence of non-negative real numbers. We define a real-valued decoration on the Ulam tree $\ell:\bU\rightarrow \R_+$ using this sequence and a sequence $(u_n)_{n\geq 1}$ of distinct elements of $\bU$ as
\begin{align*}
\ell(u_k)&=x_k \quad \text{for all $k\geq 1$},\\
\ell(u)&=0  \quad \text{for any $u\notin\enstq{u_k}{k\geq 1}$}
\end{align*}
The following lemma ensures the non-explosion of $\ell$ under some assumptions that are often met in our cases of application. 
\begin{lemma}\label{growing:lem:sufficient condition for non-explosion}
	If there exists constants $\epsilon>0$ and $K>0$ such that for all $n\geq 1$
	\begin{equation*}
	x_n\leq (n+1)^{-\epsilon+\petito{1}}\quad \text{as $n\rightarrow\infty$} \qquad \text{and} \qquad \haut(u_n)\leq K\cdot \log n,
	\end{equation*}
	then the function $\ell$ defined above is non-explosive.
\end{lemma}
\begin{proof}
	Let $i\in\N$. For any $u\in \bU$ we have
	\begin{align*}
	\underset{v\in \{u_k,\ 2^i<k\leq 2^{i+1}\}}{\sum_{v\prec u}}\ell(v)&
	\leq \#\enstq{k\in \intervalleentier{2^i+1}{2^{i+1}}}{u_k\prec u} \cdot \left(\max_{2^{i}<k\leq 2^{i+1}} \ell(u_k)\right)\\
	&\leq K\cdot\log 2 \cdot (i+1)\cdot (2^i)^{-\epsilon+\petito{1}},
	\end{align*}
	where the last display is independent of $u$. Now, if we consider any sequence of plane trees $(\tau_i)_{i\geq 1}$ such that for every $i\geq 1$ the tree $\tau_i$ contains all the vertices $\{u_1,u_2,\dots,u_{2^i}\}$, then we have for any $u\in \bU$,
	\begin{align*}
	\underset{v\notin \tau_i}{\sum_{v\prec u}}\ell(v) \leq \sum_{j=i}^\infty \underset{v\in \{u_k,\ 2^i<k\leq 2^{i+1}\}}{\sum_{v\prec u}}\ell(v)&\leq \sum_{j=i}^\infty K\cdot\log 2 \cdot (j+1)\cdot (2^j)^{-\epsilon+\petito{1}}
	\end{align*}
	and the last display converges to $0$ as $i\rightarrow\infty$, which proves the lemma.
\end{proof}
\section{Preferential attachment and weighted recursive trees}\label{growing:sec:pa and wrt}
In this section we recall some results about preferential attachment trees with initial fitnesses and weighted recursive trees that are proved in the companion paper \cite{senizergues_geometry_2019}. 
\subsection{Definitions}\label{growing:subsec:definition pa wrt}
\paragraph{Weighted recursive trees (WRT).}
For any sequence of non-negative real numbers $(w_n)_{n\geq 1}$ with $w_1>0$, the distribution $\wrt((w_n)_{n\geq 1})$ of the \emph{weighted recursive tree with weights $(w_n)_{n\geq 1}$} is defined on sequences of growing plane trees.
A sequence $(\ttT_n)_{n\geq 1}$ having this distribution is constructed iteratively starting from $\ttT_1$ containing only one vertex $u_1=\emptyset\in \bU$, in the following manner: the tree $\ttT_{n+1}$ is obtained from $\ttT_n$ by adding a vertex $u_{n+1}$. The parent of this new vertex is chosen to be any of the vertices $u_k\in \ttT_n$ with probability proportional to $w_k$, and $u_{n+1}$ is added to the tree so that it is the rightmost child of its parent.
Whenever we consider a random sequence of weight $(\mathsf{w}_n)_{n\geq 1}$, the distribution $\wrt((\mathsf{w}_n)_{n\geq 1})$ denotes the law of the random tree obtained by first sampling the sequence $(\mathsf{w}_n)_{n\geq 1}$ and then, conditionally on $(\mathsf{w}_n)_{n\geq 1}$, running the above construction with this sequence of weights.

\paragraph{Preferential attachment trees (PA).}
For any sequence $\mathbf a=(a_n)_{n\geq 1}$ of real numbers, with $a_1>-1$ and $a_n\geq 0$ for $n\geq 2$, we define another distribution on growing sequences $(\ttP_n)_{n\geq 1}$ of plane trees called the \emph{affine preferential attachment tree with initial fitnesses $(a_n)_{n\geq 1}$} which is denoted $\pa((a_n)_{n\geq 1})$. 
The construction goes on as before: $\ttP_1$ contains only one vertex $u_1$ and $\ttP_{n+1}$ is obtained from $\ttP_n$ by adding a vertex $u_{n+1}$, whose parent is chosen to be any $u_k\in \ttT_n$ with probability proportional to $\deg^{+}_{\ttP_n}(u_k)+a_k$,
where $\deg^{+}_{\ttP_n}(\cdot)$ denotes the number of children in the tree $\ttP_n$. By convention if $n=1$, the second vertex $u_2$ is always defined as a child of $u_1$.

\subsection{Properties of preferential attachment and weighed recursive trees}
Let us state the properties proved in the companion paper \cite{senizergues_geometry_2019} that will be needed in our analysis. Let us suppose here that we consider a sequence $\mathbf{a}=(a_{n})$ such that
\begin{align}\label{growing:eq:assum An=cn}\tag{$H_{c,c'}$}
	A_n:=\sum_{i=1}^{n}a_i=c\cdot n + \grandO{n^{1-\epsilon}} \quad \text{and} \quad a_n\leq n^{c'+\petito{1}}
\end{align}
for some constants $c>0$, some $0\leq c'<\frac{1}{c+1}$ and some $\epsilon>0$.
\paragraph{Convergence of degrees and representation theorem.}
A first result concerns the scaling limit of the degrees of the vertices in their order of creation and the distribution of the sequence of trees conditionally on the limit sequence; it can be read from \cite[Theorem~1, Proposition~2 and Proposition~5]{senizergues_geometry_2019}. We have the following convergence in the product topology to a random sequence
\begin{align}\label{growing:eq:convergence of degrees pa}
	n^{-\frac{1}{c+1}}\cdot (\deg_{\ttP_n}^+(u_1),\deg_{\ttP_n}^+(u_2),\dots)\overset{\text{a.s.}}{\underset{n\rightarrow\infty}{\longrightarrow}}(\mathsf m^\mathbf{a}_1,\mathsf m^\mathbf{a}_2,\dots),
\end{align}
and conditionally on the sequence $(\mathsf{m}^\mathbf a_k)_{k\geq 1}$ the sequence $(\ttP_n)$ has distribution $\wrt((\mathsf{m}^\mathbf a_k)_{k\geq 1})$. 
Also, the limiting sequence $(\mathsf{m}^\mathbf a_k)_{k\geq 1}$ has the following behaviour, which depends on the parameters $c$ and $c'$ 
\begin{equation*}
\mathsf M^\mathbf a_k:=\sum_{i=1}^{k}\mathsf m^\mathbf{a}_i \underset{k\rightarrow\infty}{\sim}(c+1)\cdot k^{\frac{c}{c+1}} \qquad \text{and} \qquad \mathsf{m}^\mathbf a_k\leq (k+1)^{c'-\frac{1}{c+1}+o_\omega(1)}.
\end{equation*}
for a random function $o_\omega(1)$ which only depends on $k$ and tends to $0$ as $k\rightarrow \infty$. 
The convergence \eqref{growing:eq:convergence of degrees pa} is such that for all $n$ large enough
\begin{align}\label{growing:eq:controle degnuk}
	\forall k\geq 1, \quad \deg_{\ttP_n}^+(u_k)\leq n^{\frac{1}{c+1}}\cdot (k+1)^{c'-\frac{1}{c+1}+o_\omega(1)},
\end{align}
also for a random function $o_\omega(1)$ of $k$.

\paragraph{Distribution of $(\mathsf{M}^\mathbf a_k)_{k\geq 1}$.}
In some very specific cases for the sequence $\mathbf{a}$, the process $(\mathsf{M}^\mathbf a_k)_{k\geq 1}$ has an explicit distribution. In particular if $\mathbf{a}=a,b,b,b\dots$ then the sequence $(\mathsf{M}^\mathbf a_k)_{k\geq 1}$ has the Mittag-Leffer Markov chain distribution $\MLMC(\frac{1}{b+1},\frac{a}{b+1})$, which we define below.

Let $0<\alpha<1$ and $\theta>-\alpha$. The generalized Mittag-Leffler $\mathrm{ML}(\alpha, \theta)$ distribution has $p$th moment
\begin{align}\label{growing:eq:moments mittag-leffler}
\frac{\Gamma(\theta) \Gamma(\theta/\alpha + p)}{\Gamma(\theta/\alpha) \Gamma(\theta + p \alpha)}=\frac{\Gamma(\theta+1) \Gamma(\theta/\alpha + p+1)}{\Gamma(\theta/\alpha+1) \Gamma(\theta + p \alpha+1)}
\end{align}
and the collection of $p$-th moments for $p \in \N$ uniquely characterizes this distribution. 
Then, a Markov chain $(\mathsf M_n)_{n\geq 1}$ has the distribution $\MLMC(\alpha,\theta)$ if for all $n\geq 1$,
\begin{equation*}
\mathsf M_n\sim \mathrm{ML}\left(\alpha,\theta+n-1\right),
\end{equation*}
and its transition probabilities are characterised by the following equality in law:
\begin{equation*}
\left(\mathsf M_n,\mathsf M_{n+1}\right)=\left(B_n\cdot \mathsf M_{n+1},\mathsf M_{n+1}\right),
\end{equation*}
where $B_n\sim \mathrm{Beta}\left(\frac{\theta+k-1}{\alpha}+1,\frac{1}{\alpha}-1\right)$ is independent of $\mathsf M_{n+1}^{\alpha,\theta}$.

For any sequence $\mathbf{a}=a,k_1,k_2,\dots k_p,k_1,k_2,\dots k_p,\dots$ that is periodic starting from the second term, with $k_1,k_2,\dots k_p$ being non-negative integers, the Markov chain $(\mathsf M^\mathbf{a}_n)_{n\geq 1}$ also has a rather explicit distribution, see \cite[Proposition~28]{senizergues_geometry_2019}, involving a product of independent Gamma random variables.
\paragraph{Height.}
In this setting, we know that the height of the tree $\ttP_n$ grows logarithmically in $n$ using \cite[Theorem~3]{senizergues_geometry_2019}. We only need here the following weak version: there exists some constant $K$ such that 
\begin{equation}\label{growing:eq:height in log}
\haut(\ttP_n)\leq K\cdot \log n,
\end{equation}
almost surely for all $n$ large enough. This estimate is also true for any sequence $(\ttT_n)_{n\geq 1}$ of weighted random trees with weights $(w_n)_{n\geq 1}$ as soon as $W_n:=\sum_{i=1}^{n}w_i$ has at most a polynomial growth, which we always assume. 

\paragraph{Measures.}
For a sequence of trees $(\ttT_n)_{n\geq 1}$ evolving under the distribution $\wrt((w_n)_{n\geq 1})$ for any weight sequence $(w_n)$, the result \cite[Theorem~4]{senizergues_geometry_2019} ensures that the probability measures $(\mu_n)_{n\geq 1}$, defined in such a way that for all $k\in\{1,\dots n\}$ we have $\mu_n(u_k)=\frac{w_k}{W_n}$, converge almost surely weakly on $\overline{\bU}$ towards a limiting measure $\mu$. 

Under the conditions $\sum_{n=1}^{\infty}w_n=\infty$ and  $\sum_{n=1}^{\infty}\left(\frac{w_n}{W_n}\right)^2<\infty$, which are almost surely satisfied by our sequence $(\mathsf{m}^\mathbf a_n)_{n\geq 1}$, the limiting measure $\mu$ is carried on $\partial \bU$, and other sequences of probability measures $(\eta_n)_{n\geq 1}$ and $(\nu_n)_{n\geq 1}$, which we define below, also converge almost surely weakly towards $\mu$. 

For any $n\geq1$, the measure $\nu_n$ is just defined as the uniform measure on the set $\{u_1,\dots,u_n\}$.
The second sequence of measures $(\eta_n)_{n\geq1}$ depends on a sequence $(b_n)_{n\geq 1}$ of real numbers which satisfies $b_1>-1$ and $b_n\geq 0$ for all $n\geq 2$. We suppose that $b_n=\grandO{n^{1-\epsilon}}$ for some $\epsilon>0$ and that $B_n:=\sum_{i=1}^{n}b_i=\grandO{n}$. 
The measures are then defined in such a way that $\eta_1$ only charges the vertex $u_1$, and for every $n\geq 2$, the measure $\eta_n$ charges only the vertices $\{u_1,u_2,\dots,u_n\}$ where for any $1\leq k \leq n$,
\begin{equation}\label{growing:eq:mesures par les degrés}
\eta_n(u_k)=\frac{b_k+\deg^{+}_{\ttT_n}(u_k)}{B_n+n-1}.
\end{equation}
We recall the following result of \cite[Proposition~8]{senizergues_geometry_2019}.
\begin{proposition}\label{growing:prop:convergence measures on wrt}
Under the assumptions $\sum_{n=1}^{\infty}w_n=\infty$ and  $\sum_{n=1}^{\infty}\left(\frac{w_n}{W_n}\right)^2<\infty$, the sequences $(\mu_n)_{n\geq 1}$, $(\nu_n)_{n\geq 1}$ and $(\eta_n)_{n\geq 1}$ almost surely converge towards the same limit $\mu$.
\end{proposition}
\paragraph{Other description of the measure $\mu$ in the case of constant sequence $\mathbf{a}$.}
Suppose now that $\mathbf{a}$ is constant from the second term, say $a_1=a>-1$ and $a_n=b>0$ for all $n\geq 2$, so that it satisfies \eqref{growing:eq:assum An=cn} with $c=b$ and $c'=0$.  
For all $u\in\bU$, and all $i\geq 1$, we define using the limiting measure $\mu$ the quantities 
\begin{align}\label{growing:eq:def Pu}
\mathsf p_{ui}=\frac{\mu(T(ui))}{\mu(T(u))},
\end{align}
which describe how the mass above every vertex $u$ is split into the subtrees above its children. In this case we can explicitly describe the law of the $(\mathsf p_{u})_{u\in\bU}$ and hence also the law of $\mu$.

Moreover, let $\ell:\bU\rightarrow \R_+$ defined as
\begin{align*}
 \ell(u_n) =\mathsf m^\mathbf{a}_n \qquad \forall n \geq 1.
\end{align*}
Remark that almost surely this defines $\ell$ on all vertices of $\bU$ and that for any $u\in\bU$ we have
 \[\ell(u):=\lim_{n\rightarrow\infty}n^{-1/(b+1)} \deg^+_{\ttP_n}(u).\] 
Thanks to \cite{janson_random_2017}, the values $(\ell(u))_{u\in\bU}$ can also be expressed from the one of $(\mathsf p_{u})_{u\in\bU}$.
The following proposition describes the joint distribution of those random variables, see Section~\ref{growing:subsection:chinese restaurant process} in the appendix for the definition of the distributions involved. 
\begin{proposition}\label{growing:prop:pa tree are split trees}
	In this setting we have 
	\[(\mathsf p_{i})_{i\geq 1}\sim \GEM\left(\frac{1}{b+1},\frac{a}{b+1}\right) \quad \text{and} \quad \forall u \in \bU\setminus\{\emptyset\}, \ (\mathsf p_{ui})_{i\geq 1}\sim \GEM\left(\frac{1}{b+1},\frac{b}{b+1}\right), \]
	and they are all independent. Denote for all $u\in\bU$, 
	\begin{align*}
	S_u:=\Gamma\left(\frac{b}{b+1}\right)\cdot \lim_{i\rightarrow\infty}i\cdot \mathsf{p}_{ui}^{\frac{1}{b+1}}, 
	\end{align*}the $\frac{1}{b+1}$-diversity of the sequence $(\mathsf p_{ui})_{i\geq 1}$. Then for all $u\in \bU$,
	\[\ell(u)=\left(\prod_{v\preceq u}\mathsf p_v\right)^{\frac{1}{b+1}}\cdot S_u.\]
\end{proposition}
\begin{proof}
	This result almost follows from \cite[Theorem~1.5]{janson_random_2017} and the adaptation to our case is left to the reader.
\end{proof}

\section{Distributions on decorations}\label{growing:sec:distributions on decorations}
In this section, we define two families of distributions on decorations on the Ulam tree that will arise as limits of our discrete models.  
\subsection{The iterative gluing construction}
Let $(\bB_n,\bD_n,\bRho_n,(X_{n,i})_{i\geq 1})_{n\geq 1}$ be a sequence of independent random variables in $\M^{\infty \bullet}$, meaning compact pointed metric spaces endowed with a sequence of points. Let also $(w_n)_{n\geq1}$ and $(\lambda_n)_{n\geq1}$ be two sequences of non-negative real numbers, which we call respectively the weights and scaling factors. 
The model is the following: first sample $(\ttT_n)_{n\geq 1}$ with distribution $\wrt((w_n)_{n\geq 1})$. Then for all $n\geq1$, denoting $u_n$ the $n$-th created vertex in the trees $(\ttT_n)_{n\geq 1}$, we set 
\begin{equation}
\cD(u_n)=(D_{u_n},d_{u_n},\rho_{u_n},(x_{u_ni})_{i\in\N}):=(\bB_n,\lambda_n\cdot \bD_n,\bRho_n,(X_{n,i})_{i\geq 1}),
\end{equation}
and for all $u\notin \enstq{u_n}{n\geq 1}$, we set 
\begin{equation*}
\cD(u)=\left(\{\star\},0,\star,(\star)_{i\geq 1}\right).
\end{equation*}
Let us assume that the cumulated sum of the weights $W_n$ does not grow faster to infinity than polynomially, so that the height of the tree grows at most logarithmically (see \cite[Theorem~3]{senizergues_geometry_2019}). 
We also assume that there exists $\alpha>0$ and $p>1$ with $\alpha p>1$  such that \[\lambda_n\leq n^{-\alpha +\petito{1}}\quad \text{and} \quad \sup_{n\geq 1}\Ec{\diam(B_n)^p}<\infty,\] 
then we have almost surely $\diam \cD(u_n)\leq n^{-\epsilon+o_\omega(1)}$, with $\epsilon=\alpha -\frac{1}{p}>0$.
This is easily derived using Markov inequality and the Borel-Cantelli lemma.
Using Lemma~\ref{growing:lem:sufficient condition for non-explosion}, the function $(u\mapsto \diam(\cD_n(u)))$ is then almost surely non-explosive so $\sG(\cD)$ is almost surely compact. 

Assuming that the sequence $(w_n)_{n\geq 1}$ has an infinite sum, the limit $\mu$ of the weight measure associated to the trees $(\ttT_n)_{n\geq 1}$ is almost surely carried on $\partial\bU$ and the random metric space $\sG(\cD)$ can a.s. be endowed with a probability measure $(\iota_{\cD})_*\mu$ and this yields the random measured metric space $\sG(\cD,\mu)$, as constructed in \eqref{growing:eq:def GDmu}. 

We call this procedure the \emph{iterative gluing construction} with blocks $(\bB_n,\bD_n,\bRho_n,(X_{n,i})_{i\geq 1})_{n\geq 1}$, scaling factors $(\lambda_n)_{n\geq 1}$ and weights $(w_n)_{n\geq1}$. 
We allow the sequences $(\lambda_n)_{n\geq 1}$ and $(w_n)_{n\geq1}$ to be random and in this case we assume that they are independent of the blocks $(\bB_n,\bD_n,\bRho_n,(X_{n,i})_{i\geq 1})_{n\geq 1}$ .

\paragraph{The case of exchangeable distinguished points.}
A special case of the above construction is given when $(\bB_n,\bD_n,\bRho_n,(X_{n,i})_{i\geq 1},\bNu_n)_{n\geq 1}$ is a sequence in $\K^{\infty \bullet}$ and that for all $n\geq 1$, conditionally on $\bNu_n$, the points $(X_{n,i})_{i\geq 1}$ are i.i.d. with law $\bNu_n$, independent of everything else.
In this case we still call this distribution the iterative gluing construction with blocks $(\bB_n,\bD_n,\bRho_n,\bNu_n)_{n\geq 1}$, scaling factors $(\lambda_n)_{n\geq 1}$ and weights $(w_n)_{n\geq1}$. This is the setting studied in \cite{senizergues_random_2019}.

\subsection{The self-similar case}\label{growing:subsec:the self-similar case}
Let us describe particular cases of models that are self-similar in distribution.
This setting is adapted from the one studied by Rembart and Winkel in \cite{rembart_recursive_2018}, which deals with several models of self-similar random trees. Let us define a model inspired from theirs. 
We fix $\beta>0$ and the law of a couple $\left((\bB,\bD,\bRho,(X_{i})_{i\geq1}), (P_{i})_{i\geq 1}\right)$, where the first coordinate is a random variable in $\M^{\infty\bullet}$ and $(P_{i})_{i\geq 1}$ is a random variable in, say, $\intervalleff{0}{1}^\N$.
In order to use mimic the notation of \cite{rembart_recursive_2018}, we let $\Xi=\M^{\infty\bullet}\times\intervalleff{0}{1}^\N$.
We consider a family
\begin{align*}
(\xi_u)_{u\in\bU}=\left((\bB_u,\bD_u,\bRho_u,(X_{ui})_{i\geq1}), (P_{ui})_{i\geq 1}\right)_{u\in\bU}
\end{align*}  of r.v.\ in $\Xi$ which are i.i.d., with the same law as some $\xi=\left((\bB,\bD,\bRho,(X_{i})_{i\geq1}), (P_{i})_{i\geq 1}\right)$. We set $P_\emptyset=1$ and 
\begin{align*}
\lambda_u=\left(\prod_{v\preceq u} P_v\right)^\beta
\end{align*}
We then define our random decorations as, for all $u\in\bU$,
\begin{equation}
\cD(u):= (\bB_u,\lambda_u \cdot \bD_u,\bRho_u,(X_{ui})_{i\geq1}).
\end{equation}
We say that $\cD$ is a self-similar decoration with exponent $\beta$ and \emph{base} distribution given by $\xi$.

We want to show that, under suitable assumptions, the resulting $\sG(\cD)$ is almost surely compact.
In our examples, the distribution of $(P_i)_{i\geq 1}$ will always be $\GEM(\alpha,\theta)$ for some parameters $\alpha\in\intervalleoo{0}{1}$ and $\theta>-\alpha$, see Section~\ref{growing:subsection:chinese restaurant process} in the Appendix for references about this process, but the arguments presented here are still valid in greater generality. 

\paragraph{The function $\phi_\beta$.}
\begin{figure}
	\begin{center}
		\includegraphics[height=4cm]{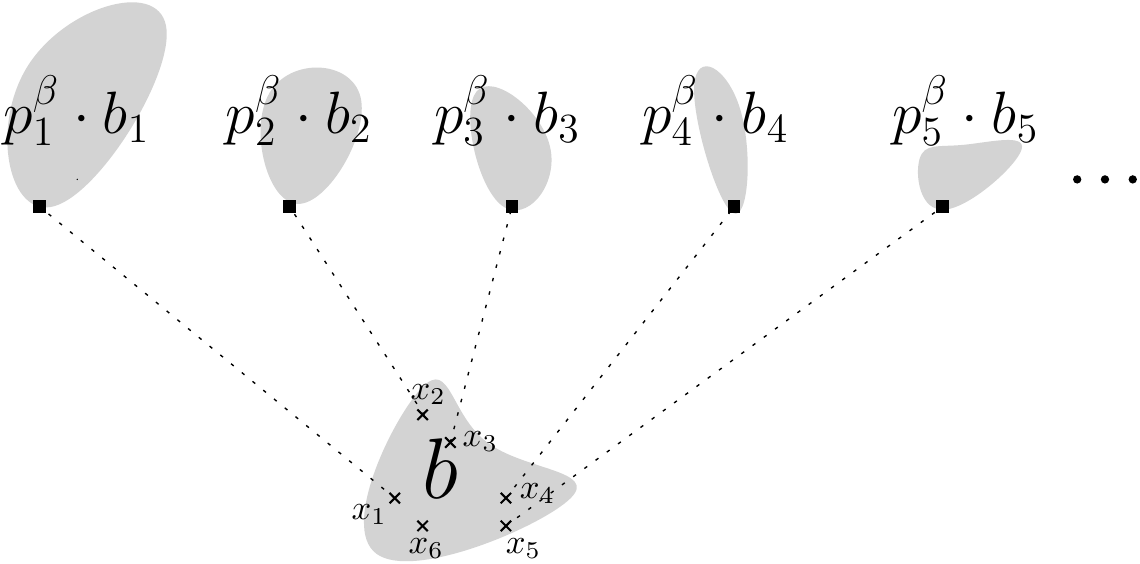} 
		\caption{The function $\phi_\beta$}\label{growing:fig:function phi_beta} 
	\end{center}
\end{figure}
We first define, for any $n\geq 1$, the function \[\phi_\beta^{(n)}:\Xi\times (\M^{\bullet})^\N\rightarrow \M^{\bullet},\] as follows:
$\phi_\beta^{(n)}((b,d,\rho,(x_i)_{i\geq 1}),(p_i)_{i\ge 1}, (b_i,d_i,\rho_i)_{i\geq 1})$ is the metric space obtained after gluing the $n$ first $b_i$ with distances scaled by $p_i^\beta$ by identifying their root $\rho_i$ with the point $x_i\in b$. For any $n\geq 1$ this operation is continuous with respect to the product topology on the starting space, hence it is measurable. 

Now we define $\phi_\beta$ as $\lim_{n\rightarrow\infty} \phi_\beta^{(n)}$ on the set where this limit exists, and constant equal to $(\{\ast\},0,\ast)$ on the complementary set. Since $\M^\bullet$ is Polish, the function $\phi_\beta$ is measurable. 
Remark that the condition for the limit to exists is
\begin{align*}
p_i \diam(b_i)\underset{i\rightarrow\infty}{\longrightarrow}0.
\end{align*}

\paragraph{The contraction $\Phi_\beta$.}
Consider the set of probability measures $\mathcal P(\M^\bullet)$ on the space $\M^\bullet$ and for $p\ge 1$, the subset $\mathcal P_p \subset \mathcal P(\M^\bullet)$ given by
\begin{equation} \mathcal P_{p}:=\enstq{\eta\in\mathcal P(\M^\bullet)}{\Ec{\diam(\tau)^{p}} < \infty \text{ for } \tau \sim \eta}. \end{equation}
We equip $\mathcal P_p$ with the \emph{Wasserstein metric} of order $p \geq 1$, which is defined by
\begin{equation}
W_p\left(\eta, \eta'\right) :=\left(\inf \mathbb E\left[\left|d_{\rm GH}\left(\tau, \tau'\right)\right|^p\right]\right)^{1/p}, \qquad \eta, \eta' \in \mathcal P_p,
\end{equation}
where the infimum is taken over all joint distributions of $(\tau,\tau^\prime)$ on $(\M^\bullet)^2$ with marginal distributions $\tau\sim\eta$ and $\tau^\prime\sim\eta'$. The space $(\mathcal P_p, W_{p})$ is complete since $d_{\rm GH}$ is a complete metric on $\M^\bullet$. Convergence in $(\mathcal P_p,W_p)$ implies weak convergence on $\M^\bullet$ and convergence of $p$th diameter moments.

Then we define the function $\Phi_\beta: \cP_p\rightarrow\cP_p$ where for any $\eta\in \cP_p$, the image $\Phi_\beta(\eta)$ is the distribution of
\begin{align*}
\phi_\beta(\xi,(\tau_i)_{i\geq 1}),
\end{align*}
where the $(\tau_i)_{i\geq 1}$ are i.i.d.\ random variables with law $\eta$, independent of $\xi$. Now let us state a result that was stated in the context of trees but remains valid in our case.
\begin{lemma}[Lemma~3.4 of \cite{rembart_recursive_2018}]\label{growing:lem:rembart winkel} Let $\beta>0$, $p\ge 1$ and $\left((\bB,\bD,\bRho,(X_{i})_{i\geq1}), (P_{i})_{i\geq 1}\right)$ such that $\Ec{\diam(\bB)^p}<\infty$ and $\Ec{\sum_{j \geq 1}P_j^{p\beta}} < 1$. Then the map $\Phi_\beta \colon \mathcal P_p \rightarrow \mathcal P_p$ associated with 
	$\phi_\beta$ is a strict contraction with respect to the Wasserstein metric of order $p$, i.e. 
	\begin{equation}
	\sup \limits_{\eta, \eta' \in \mathcal P_p, \eta \neq \eta'} \frac{W_p \left(\Phi_\beta(\eta), \Phi_\beta(\eta')\right)}{W_p(\eta, \eta')} < 1.
	\end{equation}
\end{lemma}
Now using Banach fixpoint theorem, we know that their exists in $\cP_p$ a unique fixed point of this function $\Phi_\beta$. 

\paragraph{Compactness.}Finally, the almost sure compactness of our structure $\sG(\cD)$ is ensured by \cite[Prop. 3.5]{rembart_recursive_2018}, and actually the distribution of $\sG(\cD)$ is exactly the fixpoint of $\Phi_\beta$, and this fixpoint is attractive. 
\paragraph{Measure on the leaves.}
If we restrict ourselves to the case where the sequence $(P_i)_{i\geq 1}$ is such that $\sum_{i=1}^{\infty}P_i=1$ almost surely, we can define a measure $\mu$ on $\partial\bU$ as follows:
\begin{align*}
\forall u\in\bU, \quad \mu(T(u))=\prod_{v\preceq u} P_v.
\end{align*}
Then we can consider the measured metric space $\sG(\cD,\mu)$ by endowing $\sG(\cD)$ with the measure $\tilde\mu=(\iota_\cD)_*\mu$. Under the condition $\Pp{\exists i\geq 1,\ \bD(\rho, X_i)> 0 \text{ and } P_i>0}>0$, one can check that this measure is almost surely carried on the set of leaves $\sL(\cD)$.

\paragraph{Hausdorff dimension of the leaves.}
Under some mild hypotheses on the distribution of our blocks, we can compute the Hausdorff dimension of $\sL(\cD)$ almost surely. 
\begin{proposition}
	Let $\beta>0$, $p\ge 1$ and $\left((\bB,\bD,\bRho,(X_{i})_{i\geq1}), (P_{i})_{i\geq 1}\right)$ such that $\Ec{\diam(\bB)^p}<\infty$ and $\Ec{\sum_{j \geq 1}P_j^{p\beta}} < 1$. Suppose furthermore that almost surely $\sum_{j \geq 1}P_j=1$ and that $\Pp{\exists i\geq 1,\ \bD(\rho, X_i)> 0 \text{ and } P_i>0}>0$. Then the Hausdorff dimension of $\sL(\cD)$ is almost surely:
	\begin{align*}
	\dim_{\mathrm H}(\sL(\cD))=\frac{1}{\beta}.
	\end{align*}
\end{proposition}
\begin{proof}
	We prove this by providing an upper-bound and a lower-bound for the dimension. The upper-bound follows from the proof of \cite[Lemma 4.6]{rembart_recursive_2018} which adapts to our new setting. For the lower-bound, we provide a direct argument, which uses crucially the assumption that $\sum_{j \geq 1}P_j=1$ a.s. Indeed, in this case, the preceding paragraph ensures the existence of a measure $\tilde\mu$ on  $\sL(\cD)$. Let us show that for $\tilde\mu$-almost every point $x$, we have:
	\begin{align}\label{growing:eq:mass dist principle}
	\liminf_{r\rightarrow 0} \frac{\log \tilde\mu(B(x,r)) }{-\log r} \leq -\frac{1}{\beta},
	\end{align} 
	which will prove the proposition, using the mass distribution principle (see \cite{falconer_fractal_2014} for example). 
	Actually, it is easy to see that, for \eqref{growing:eq:mass dist principle} to hold, it is enough to provide a sequence $(r_n)_{n\geq 1}$ tending to $0$ such that $\frac{\log r_n}{\log r_{n+1}}\rightarrow 1$ and
	\begin{align}\label{growing:eq:mass dist principle seq}
	\liminf_{n\rightarrow \infty} \frac{\log \tilde\mu(B(x,r_n)) }{-\log r_n} \leq -\frac{1}{\beta},
	\end{align} 
	
	Let us prove that \eqref{growing:eq:mass dist principle seq} holds almost surely for a point $L$ taken under the measure $\tilde\mu$ and a random sequence $R_n$. Using the product definition of $\mu$, it is straightforward to see that if $\mathbf{I}=I_1I_2,\dots\in\partial\bU$ is taken under the measure $\mu$, then the sequence $I_1,I_2,\dots$ is i.i.d.\ with the same distribution as $I$ given by $\Ppsq{I=i}{(P_j)_{j\geq 1}}=P_i$, where the sequence $(P_j)_{j\geq 1}$ has the distribution of that in the theorem. We can compute $\Ec{\log P_I}=\Ec{\Ecsq{\log P_I}{(P_j)_{j\geq 1}}}=\Ec{\sum_{i=1}^{\infty}P_i\log P_i}$. 
	
	Let $R_n:=\dist(\rho_{\mathbf{I}_n},\iota_{\cD}(\mathbf{I}))$ be the distance of the random leaf $L:=\iota_{\cD}(\mathbf{I})$ to the root $\rho_{\mathbf{I}_n}$ of the block $\cD(\mathbf{I}_n)$. Remark that the open ball $\Ball(L,R_n)$ of centre $L$ and radius $R_n$ only contains points that come from decorations with indices $u\succeq \mathbf I_n$ so that $\mu(\Ball(L,R_n))\leq \mu(T(\mathbf I_n))$.
	
Now write
	\begin{align*}
\log \mu(\Ball(L,R_n)) \leq	\log \mu(T(\mathbf I_n))=\sum_{i=1}^{n} \log P_{\mathbf I_i} \underset{n\rightarrow\infty}{\sim} n \cdot \Ec{\sum_{i=1}^{\infty}P_i\log P_i},
	\end{align*}
	almost surely, because of the law of large numbers. Now it suffices to prove that almost surely
	\begin{align}\label{growing:eq:log rn sim}
	\log R_n\underset{n\rightarrow\infty}{\sim}n \beta \cdot \Ec{\sum_{i=1}^{\infty}P_i\log P_i},
	\end{align}
	and \eqref{growing:eq:mass dist principle seq} would follow for the random leaf $L$ thanks to the two last displays. In order to prove \eqref{growing:eq:log rn sim} we write
	\begin{align*}
	R_n:=\dist(\rho_{\mathbf{I}_n},\iota_{\cD}(\mathbf{I}))&=\sum_{k=n}^{\infty} \left(\prod_{i=1}^{k}P_{\mathbf I_i}\right)^\beta \bD_{\mathbf{I}_k}(\rho, X_{\mathbf{I}_{k+1}}),
	\end{align*}
	using the definition of the distances in $\sG(\cD)$. 
Then let us fix $\delta>0$ such that $\Pp{\bD(\rho,X_I)>\delta}>\delta$, and let $\tau_n=\inf\enstq{i\geq n}{\bD_{\mathbf{I}_i}(\rho, X_{\mathbf{I}_{i+1}})>\delta}$. Then we have
	\begin{align*}
	\Pp{\tau_n\geq n+\sqrt{n}}\leq (1-\delta)^{\sqrt{n}},
	\end{align*}  
	which is summable in $n$, so that using Borel-Cantelli lemma, we have almost surely: $n\leq \tau_n \leq  n+\sqrt{n}$. Then for all $n$ large enough
	\begin{align*}
	R_n\geq \left(\prod_{i=1}^{n+\sqrt{n}}P_{\mathbf I_i}\right)^\beta\cdot \delta,
	\end{align*}
	and this proves that $\log R_n\geq \beta \sum_{i=1}^{n+\sqrt{n}}\log P_{\mathbf{I}_i}+\log \delta$. For an upper bound, remark that 
	\begin{align*}
	R_n=\left(\prod_{i=1}^{n}P_{\mathbf I_i}\right)^\beta \cdot \underbrace{\left(\bD_{\mathbf{I}_n}(\rho, X_{\mathbf{I}_{n+1}}) +P_{\mathbf I_{n+1}}\sum_{k=n+1}^\infty \left(\prod_{i=n+1}^{k}P_{\mathbf I_i}\right)^\beta \bD_{\mathbf{I}_i}(\rho, X_{\mathbf{I}_{i+1}})\right)}_{R_n'},
	\end{align*}
	where $R_n'$ has the same law as $R_0$, which admits a finite first moment. Using Markov inequality and Borel-Cantelli lemma, we get that almost surely for any $n$ large enough, $R_n'\leq n^p$ for some $p>1$. Then for all $n\geq 1$ large enough
	\begin{align*}
	\log R_n &= \log \left(\prod_{i=1}^{n}P_{\mathbf I_i}\right)^\beta + \log R_n'\leq \beta \sum_{i=1}^{n} \log P_{\mathbf I_i} + p \log n.
	\end{align*}
	In the end, using the upper and lower bound on $R_n$ and the law of large numbers we get \eqref{growing:eq:log rn sim}, which finishes the proof of the proposition.
\end{proof}
\paragraph{Almost-self-similar decorations.}
For our needs, we define a slight variation of this model where we only suppose that the random variables $(\xi_u)_{u\in\bU\setminus \{\emptyset\}}$ have the same law as $\xi$, and $\xi_\emptyset=\left((\bB_\emptyset,\bD_\emptyset,\bRho_\emptyset,(X_{i})_{i\geq1}), (P_{i})_{i\geq 1}\right)$ is independent of the variables $(\xi_u)_{u\in\bU\setminus \{\emptyset\}}$ but can possibly have a different law. 

In this case, we say that the obtained $\cD$ is almost-self-similar with exponent $\beta$ and base distributions $\xi_\emptyset$ and $\xi$. 
If $\xi_\emptyset$ satisfies the conditions of Lemma~\ref{growing:lem:rembart winkel} as well as $\xi$, the above arguments still hold and the obtained random metric space has the law of $\phi_\beta(\xi_\emptyset,(\tau_i)_{i\geq 1})$ where the $(\tau_i)_{i\geq 1}$ are i.i.d.\ with distribution $\eta$ which is the unique fixed point of $\Phi_\beta$. 
\subsection{Some decorations constructed by iterative gluing are also self-similar }
Some random decorations that are described using an iterative gluing construction also belong to the family of almost-self-similar decorations. The following proposition ensures that this is the case for a particular family of iterative gluing constructions.
\begin{proposition}\label{growing:prop:iterative gluing constructions can be ass}
	Suppose that $\cD$ is defined as an iterative gluing construction using	\begin{enumerate}
\item a sequence of weight $(\mathsf m_n)_{n\geq1}$ defined as the increments of a Mittag-Leffler Markov chain $(\mathsf M_n)_{n\geq 1} \sim\MLMC(\frac{1}{b+1},\frac{a}{b+1})$, with $a>-1$ and $b>0$,
\item a sequence of scaling factors taken as $(\mathsf m_n^\gamma)_{n\geq 1}$ for some $\gamma>0$,
\item a sequence of independent blocks $(\bB_n,\bD_n,\rho_n,(X_{n,i})_{i\geq 1})$, with the same distribution for $n\geq 2$, such that their diameter admits a $p$-th moment with $p>1$.
\end{enumerate} 
Then $\cD$ is an almost-self-similar decoration with exponent $\frac{\gamma}{b+1}$ and base distributions $\xi_\emptyset$ and $\xi$ such that 
\begin{itemize}
	\item  $\displaystyle\xi_\emptyset \overset{(d)}{=}\left((\bB_1,S_\emptyset^\gamma\cdot \bD_1,\rho_1,(X_{1,i})_{i\geq 1}),(P_{i})_{i\geq 1}\right),$
with $(P_{i})_{i\geq 1}\sim \mathrm{GEM}\left(\frac{1}{b+1},\frac{a}{b+1}\right)$ that is independent of $\bB_1$ and $S_\emptyset$ is its $\frac{1}{b+1}$-diversity,
\item $\displaystyle\xi \overset{(d)}{=}\left((\bB_2,S^\gamma\cdot \bD_2,\rho_2,(X_{2,i})_{i\geq 1}),(\tilde{P}_{i})_{i\geq 1}\right),$
with $(\tilde{P}_{i})_{i\geq 1}\sim \mathrm{GEM}\left(\frac{1}{b+1},\frac{b}{b+1}\right)$ that is independent of $\bB_2$ and $S$ is its $\frac{1}{b+1}$-diversity.
\end{itemize}
\end{proposition}
\begin{proof}
Recall the definition of $\mu$ the probability measure on $\partial\bU$ obtained as the weak limit of the mass measure of the weighted recursive tree used for this iterative construction.
If we denote for all $u\in\bU$ and $i\in\N$,
\begin{equation*}
\mathsf p_{ui}=\frac{\mu(T(ui))}{\mu(T(u))},
\end{equation*}
then from Proposition~\ref{growing:prop:pa tree are split trees} we have a complete description of the distribution of $\left(\mathsf p_u\right)_{u\in\bU}$ using $\mathrm{GEM}$ distributions. For every $u\in\bU$, we let $S_u$ be the $\frac{1}{b+1}$-diversity of the sequence $(\mathsf p_{ui})_{i\geq 1}$. 
Hence, denoting 
\begin{align*}
\xi_\emptyset&:=\left((\bB_1,S_\emptyset^\gamma\cdot \bD_1,\rho_1,(X_{1,i})_{i\geq 1}),(\mathsf p_{i})_{i\geq 1}\right)\\
\xi_{u_k}&:=\left((\bB_k,S_{u_k}^\gamma\cdot\bD_k,\rho_k,(X_{k,i})_{i\geq 1}),(\mathsf p_{u_ki})_{i\geq 1}\right)\qquad \forall k\geq 2.
\end{align*}
it is immediate from the previous section that the $\left(\xi_u\right)_{u\in\bU}$ are independent and $\left(\xi_u\right)_{u\in\bU\setminus \{\emptyset\}}$ are i.i.d.. Hence the distribution of $\cD$ coincides with that of an almost-self-similar decoration with scaling exponent $\frac{\gamma}{b+1}$ with these base distributions.
\end{proof}

\section{Application to models of growing random graphs}\label{growing:sec:applications}
Let us use this framework of random decorations to prove scaling limits for models of growing random graphs. We first present a general proof that will apply to all our different applications. 
Every example that we treat is of the following form: we start with a model of objects defined iteratively $\left(H_n\right)_{n\geq 1}$, that can be considered as measured metric spaces, in our case it will always be graphs.
We then interpret this construction in our framework of random decorations by constructing a sequence of decorations $(\cD^{(n)})_{n\geq 1}$ and measures on those decoration $(\bnu^{(n)})_{n\geq 1}$, such that 
the distribution of the sequence $\left(\sG(\cD^{(n)}, \bnu^{(n)})\right)_{n\geq 1}$ coincides with that of $\left(H_n\right)_{n\geq 1}$ as measured metric spaces. 
In the particular examples that we are studying, the evolution of the sequence of decoration $(\cD^{(n)})_{n\geq 1}$ is quite easy to understand and, in particular, we can prove that these decorations admit a scaling limit.
This allows us to use Theorem~\ref{growing:thm:continuité du recollement} on this sequence of decorations and hence obtain a convergence in the scaling limit for the sequence $\left(H_n\right)_{n\geq 1}$. 

\subsection{An abstract result that handles all our applications}
Let us describe a particular type of sequence of decoration $(\cD^{(n)})_{n\geq 1}$ for which we can state a general scaling limit result. This setting will be rather abstract but is intended to be general enough to encompass all of our examples.

In this setting, we study a sequence $(\cD^{(n)})_{n\geq 1}$ of decorations endowed with measures $(\bnu^{(n)})_{n\geq 1}$ which are constructed using a increasing sequence of trees $(\ttP_n)_{n\geq 1}$ and a collection of processes $(\mathcal{A}_k(m),m\geq0)_{k\geq 1}$ with values in $\M^{\infty \bullet}$ that are jointly independent and independent of $(\ttP_n)_{n\geq 1}$. We assume that the following properties hold.
\begin{enumerate}
	\item\label{growing:it:Pn preferential attachment} The sequence $(\ttP_n)_{n\geq 1}$ evolves as a preferential attachment tree with some sequence of fitnesses $\mathbf a=(a_n)_{n\geq 1}$, as described in Section~\ref{growing:subsec:definition pa wrt}, which satisfies \eqref{growing:eq:assum An=cn} for some $c>0$ and $0\leq c'<\frac{1}{c+1}$.
\item\label{growing:it:Dn processus en deguk} For all $n\geq 1$, the decoration $\cD^{(n)}$ is such that for all $k\in\{1,\dots,n\}$,
\begin{equation*}
	\cD^{(n)}(u_k)=\mathcal A_k(\deg_{\ttP_n}^+(u_k)).
\end{equation*}
and for all $u\notin \{u_1,\dots,u_n\}$, the associated block is trivial i.e. $\cD^{(n)}(u)=(\star,0,\star,(\star)_{i\geq 1})$.
\item\label{growing:it:convergence des processus A}There exists $\gamma>0$ such that for all $k\geq 1$,
\begin{align*}
m^{-\gamma}\cdot \cA_k(m) \overset{\text{a.s.}}{\underset{m\rightarrow\infty}{\longrightarrow}} (\bB_k, \bD_k,\bRho_k,(X_{k,i})_{i\geq 1}) \quad \text{in } \M^{\bullet \infty}.
\end{align*} 
\item\label{growing:it:moment condition on the sup of A} There exists a sequence $(c_k)_{k\geq 1}$ such that for all $p>0$ we have
\begin{equation*}
\sup_{k\geq 1} \Ec{\sup_{m\geq 1}\left(\frac{\diam(\cA_k(m))}{(m+c_k)^{\gamma}}\right)^p}<\infty,
\end{equation*}
and the sequence $(c_k)_{k\geq 1}$ is such that $c_k\leq k^{s+\petito{1}}$ for some $s<\frac{1}{c+1}$. 
\item\label{growing:it:measures converge} The sequence of measures $(\nu^{(n)})_{n\geq 1}$ on $\bU$ that corresponds to the sequence of measure-valued decorations $(\bnu^{(n)})_{n\geq 1}$ almost surely converges towards the probability measure $\mu$ on $\partial\bU$ which is associated by Proposition~\ref{growing:prop:convergence measures on wrt} to the sequence of preferential attachment trees $(\ttP_n)_{n\geq 1}$. 

\end{enumerate}

\begin{remark}
	The assumptions \ref{growing:it:Pn preferential attachment} and \ref{growing:it:Dn processus en deguk} characterize the law of the process $(\cD^{(n)})_{n\geq 1}$ from the ones of the processes $\cA_k$ for $k\geq 1$ and the sequence $\mathbf{a}$.
	On the contrary, the details of the measures $\bnu^{(n)}$ on the different blocks are not specified and can be anything as long as the associated measures $\nu^{(n)}$ on $\bU$ converge towards $\mu$. 
\end{remark}

Under all those assumptions we have a scaling limit result. For a sequence $\mathbf{a}=(a_n)_{n\geq 1}$, recall the definition of the random sequence $(\mathsf{m}_n^\mathbf{a})_{n\geq 1}$ from \eqref{growing:eq:convergence of degrees pa}.
\begin{theorem}\label{growing:thm:metatheorem}
	Suppose that the decorations $\cD^{(n)}$ are constructed as above. Then we have the following almost sure convergence
	\begin{align*}
	\sG(n^{-\frac{\gamma}{c+1}}\cdot \cD^{(n)}, \bnu^{(n)}) \underset{n\rightarrow\infty}{\longrightarrow} \sG(\cD,\mu),
	\end{align*}
	 in the Gromov--Hausdorff--Prokhorov topology, where the limit is described as an iterative construction with blocks $(\bB_k, \bD_k,\bRho_k,(X_{k,i})_{i\geq 1})_{k\geq 1}$, scaling factors $((\mathsf m^\mathbf a_k)^\gamma)_{k\geq 1}$ and weights $(\mathsf m^\mathbf a_k)_{k\geq 1}$.
\end{theorem}
\begin{proof}
First we check that for all $k\geq 1$ and all $n\geq k$ we have 
\begin{align*}
n^{-\frac{\gamma}{c+1}}\cdot \cD^{(n)}(u_k)&=n^{-\frac{\gamma}{c+1}}\cdot \mathcal A_k(\deg_{\ttP_n}^+(u_k))\\
&=\left(n^{-\frac{1}{c+1}}\deg_{\ttP_n}^+(u_k)\right)^\gamma\cdot (\deg_{\ttP_n}^+(u_k))^{-\gamma} \cdot \mathcal A_k(\deg_{\ttP_n}^+(u_k))\\
&\underset{n\rightarrow\infty }{\longrightarrow} (\mathsf m ^\mathbf a_k)^\gamma\cdot (\bB_k, \bD_k,\bRho_k,(X_{k,i})_{i\geq 1}),
\end{align*}
in the topology of $\M^{\infty\bullet}$, using the assumption \ref{growing:it:convergence des processus A}, the convergence \eqref{growing:eq:convergence of degrees pa} of degrees in the sequence of increasing trees $(\ttP_n)_{n\geq 1}$, which is a preferential attachment tree by the assumption \ref{growing:it:Pn preferential attachment}. 
By \ref{growing:it:convergence des processus A}, for any $u\notin\{u_1,u_2,\dots\}$ the block $\cD^{(n)}(u)$ is constant equal to the trivial block. So $n^{-\frac{\gamma}{c+1}}\cdot\cD^{(n)}$ converges to some limiting decoration $\cD$ in the product topology.

Now, for any $k\geq 1$ and $n\geq k$ we have
\begin{align*}
	\diam\left(n^{-\frac{\gamma}{c+1}}\cdot \cD^{(n)}(u_k)\right)=\left(n^{-\frac{1}{c+1}}\left(\deg_{\ttP_n}^+(u_k)+c_k\right)\right)^\gamma\cdot\frac{\diam (\mathcal A_k(\deg_{\ttP_n}^+(u_k)))}{(\deg_{\ttP_n}^+(u_k)+c_k)^{\gamma}}.
\end{align*}
The first term can be shown to be smaller than some random bound $(k+1)^{-\epsilon+o_\omega(1)}$ uniformly in $n\geq k$, for some $\epsilon>0$, using \eqref{growing:eq:controle degnuk} and the assumption that $c_k\leq k^{s+\petito{1}}$ for $s<\frac{1}{c+1}$. The second term is bounded above by some $k^{o_\omega(1)}$ thanks to \ref{growing:it:moment condition on the sup of A} (using the Markov inequality and the Borel-Cantelli lemma). 
In the end, we almost surely have the following control
\begin{align}
		\diam\left(n^{-\frac{\gamma}{c+1}}\cdot \cD^{(n)}(u_k)\right)\leq k^{-\epsilon+o_\omega(1)}.
\end{align}
Thanks to \eqref{growing:eq:height in log}, the height of $\ttP_n$ is almost surely bounded above by some $K\log n$ for some constant $K$, so Lemma~\ref{growing:lem:sufficient condition for non-explosion} ensures that the function $\ell:u\mapsto \sup_{n\geq 1}\diam\left(n^{-\frac{\gamma}{c+1}}\cdot \cD^{(n)}(u)\right)$ is almost surely non-explosive.  
Thanks to Theorem~\ref{growing:thm:continuité du recollement}\ref{growing:it:convergence gh}, this ensures the almost sure Gromov-Hausdorff convergence of the spaces $\sG(n^{-\frac{\gamma}{c+1}}\cdot \cD^{(n)})$ towards $\sG(\cD)$.

Finally, assumption \ref{growing:it:measures converge} ensures that we are in the conditions of application of Theorem~\ref{growing:thm:continuité du recollement}\ref{growing:it:convergence ghp} and so we can improve the last convergence into the GHP convergence 
	\begin{align*}
\sG(n^{-\frac{\gamma}{c+1}}\cdot \cD^{(n)}, \bnu^{(n)}) \underset{n\rightarrow\infty}{\longrightarrow} \sG(\cD,\mu),
\end{align*}
 which finishes our proof.

\end{proof}
\paragraph{How to apply this theorem.}
This theorem may seem abstract at that point, but it encompasses all our specific examples of growing random graphs. 
Now for all our sequences of graphs $(H_n)_{n\geq 1}$, the goal will be to provide a sequence $(a_n)_{n\geq 1}$ satisfying \eqref{growing:eq:assum An=cn} for some parameters $c$ and $c'$ and processes $(\cA_k)_{k\geq1}$ so that the decorations $(\cD^{(n)})_{n\geq 1}$ satisfying \ref{growing:it:Pn preferential attachment} and \ref{growing:it:Dn processus en deguk} indeed evolve in such a way that $\sG(\cD^{(n)})$ coincides with our process $(H_n)_{n\geq 1}$.
Then we check that the other assumptions are also satisfied in order to get the scaling limit.

\paragraph{Particular form of processes $\cA$.}
First, remark that for any $k\geq 1$ and $m\geq 0$, all the distinguished points in $\cA_k(m)$ that matter for the construction are only the first $m$ ones. 
All the others can be set equal to the root vertex without changing the distribution of $\left(\sG(\cD^{(n)})\right)_{n\geq 1}$, so we can always suppose that at each step $m\geq 0$, the metric space $\cA_k(m)$ is endowed with only $m$ distinguished points in addition to the root and can hence be seen as an element of $\M^{m \bullet}$.

Second, in all our examples, the different processes $\cA_k$ for $k\geq 0$ all evolve under the same Markovian transitions, possibly starting from different states $\cA_k(0)$ for different values of $k\geq 1$. 
These transitions are often more naturally defined on \emph{weighted graphs}, in which each of the vertices and edges are given some weight. The dynamics involve taking a vertex or edge at random proportionally to its weight, do some local transformation of the graph at that point by possibly adding one or several vertices and edges to the graph. 
The list of distinguished points is then updated by appending some vertex to the end of the existing list. 
\paragraph{Almost self-similar limits.}
Theorem~\ref{growing:thm:metatheorem} describes the limiting space as the result of an iterative construction. In our examples, it will be often the case that Proposition~\ref{growing:prop:iterative gluing constructions can be ass} applies to the limiting space and hence that it is almost-self-similar in the sense of Section~\ref{growing:subsec:the self-similar case}. 
It happens in particular whenever the sequence $\mathbf{a}$ is of the form $\mathbf a=(a,b,b,\dots)$ and all the processes $(\cA_k)_{k\geq 2}$ have the same law.
We will not discuss further this type of construction and it is left to the reader to apply Proposition~\ref{growing:prop:iterative gluing constructions can be ass} when possible to obtain this other description of the limiting object. 

\subsection{Generalised Rémy's algorithm}
Recall the construction described in the introduction. Consider $(G_n,o_n)_{n\geq 1}$ a sequence of finite rooted graphs with number of edges given by the sequence $\mathbf a=(a_n)_{n\geq 1}$ which satisfies \eqref{growing:eq:assum An=cn} for some $c>0$ and $c'<\frac{1}{c+1}$. 
We construct the sequence $(H_n)_{n\geq 1}$ recursively as follows. 
Let $H_1=G_1$. 
Then, for any $n\geq 1$, conditionally on the structure $H_n$ already constructed, take an edge in $H_n$ uniformly at random, split it into two edges by adding a vertex "in the middle" of this edge, and glue a copy of $G_{n+1}$ to the structure by identifying $o_{n+1}$ the root vertex of $G_{n+1}$ with the newly created vertex. 
Call the obtained graph $H_{n+1}$. 
See Figure~\ref{growing:fig:decomposition le long d'un arbre} for a realisation of $H_5$ using the sequence $(G_n)_{n\geq 1}$ of Figure~\ref{growing:fig:exemple graphes gn}. 

\paragraph{Uniform edge-splitting process.}
Before decomposing this construction as a process on decorations on the Ulam tree, let us introduce a simpler process.
For any connected, rooted graph $(G,\rho)$ with at least one edge, we introduce the following process $(\cA_G(n))_{n\geq 0}$, called the \emph{uniform edge-splitting} process started from $G$. 
The initial value for the process $\cA_G(0)$ is just (the set of vertices of) the graph $G$ endowed the corresponding graph distance, rooted at $o$, with an empty list of distinguished points.
Then $\cA_G(n+1)$ is obtained from $\cA_G(n)$ by duplicating an edge uniformly at random by adding some point $x_{n+1}$ in its centre. The vertex $x_{n+1}$ is then appended at the end of the list of distinguished points, now becoming of length $n+1$. At every step, the obtained object \[\cA_G(n)=(A_G(n),\dist_{\mathrm{gr}},\rho,(x_i)_{1\leq i\leq n})\] can then be considered as an element of $\M^{\bullet n}$, a metric space with $n$ distinguished points, which we also see as an element of $\M^{\bullet \infty}$ by the usual identification. By construction, $\cA_G(n)$ is also a graph, and we will sometimes also consider it as such.

We introduce $\cC_G$ a continuous version of $G$ as a random element of $\M^{\bullet \infty}$,
\begin{equation*}
	\cC_G=(C_G,\dist,\rho,(X_i)_{i\geq 1}),
\end{equation*}
that is constructed in the following way. If we arbitrarily label $e_1,\dots,e_{\abs{E(G)}}$ the edges of $G$, then $\cC_G$ is obtained from $G$ by replacing each edge $e$ with a segment of length $L(e)$ where the lengths are such that
\begin{equation}
(L(e_1),L(e_2),\dots,L(e_{\abs{E(G)}}))\sim \mathrm{Dir}(1,1,\dots,1),
\end{equation} 
so that the total length is $1$. The $(X_{i})_{i\geq 1}$ are then obtained conditionally on this construction as i.i.d.\ points taken under the length measure.

We have the following convergence result for graphs undergoing a uniform edge-splitting process.
\begin{lemma}[Convergence of the uniform edge-splitting process]\label{growing:lem:convergence edge-splitting}
	Suppose $(G,\rho)$ is a connected, rooted graph with at least one edge. If we consider the process $(\mathcal{A}_G(n))_{n\geq 0}$ defined as above as a process on pointed metric spaces $(A_G(m),\dist_{\mathrm{gr}},\rho,(x_i)_{1\leq i\leq m})$ then we have the following convergence in $\M^{\infty \bullet}$ as $m\rightarrow\infty$,
	\[(A_G(n),\frac{1}{n}\dist_{\mathrm{gr}},\rho,(x_i)_{1\leq i\leq n}) \underset{n\rightarrow\infty}{\longrightarrow} \cC_G=(C_G,\dist,\rho,(X_i)_{i\geq 1}),\]
	where $\cC_G$ is described as above. 
\end{lemma}

\begin{proof}
We give here a sketch of the proof of the statement. 
For any edge $e$ of the original graph $G=(V,E)$, and any $n\geq 1$, we say that the edges in $\cA_G(n)$ that were created through splitting events along this edge $e$ \emph{originate} from $e$. 
As illustrated in Figure~\ref{growing:fig:edge-splitting process}, where the blue edges in $\cA_G(6)$ originate from $e$, all the edges originating from $e$ form a path in the graph $\cA_G(n)$. Let us denote $L(e,n)$ the number of edges in that path.
For an arbitrary labelling $e_1,e_2,\dots,e_{\abs{E}}$ of the edges, it is easy to see that the vector $(L(e_1,n),L(e_2,n),\dots,L(e_{\abs{E}},n))$ evolves as the weights of different colours in a P\'olya urn, as described in Theorem~\ref{growing:thm:classical Polya urn}. From this theorem, we then have the almost sure convergence 
\begin{align*}
	\frac{1}{n}\cdot (L(e_1,n),L(e_2,n),\dots,L(e_{\abs{E}},n))\underset{n\rightarrow\infty}{\longrightarrow} (L(e_1),L(e_2),\dots,L(e_{\abs{E}})),
\end{align*}
where the limit has a Dirichlet distribution $\mathrm{Dir}(1,1,\dots,1)$. This is enough to prove the convergence of the graphs $\cA_G(n)$ to the limiting metric space $\cC_G$ as rooted metric spaces. 
The convergence of the position of the points $x_1,x_2,x_3,\dots$ along their respective path is also obtained by an urn interpretation: whenever a vertex $x_i$ is created along the path originating from an edge $e$, the number of edges along that path on the left and on the right of $x_i$ evolves also like a (time-changed) P\'olya urn and hence once rescaled the position of the point along that edge converges almost surely to some point $X_i$ in the limiting space. 
Last, we have to prove that the sequence $(X_i)_{i\geq 1}$ is i.i.d.\ uniform along the length of the limiting structure $\cC_G$. This follows from the fact that for any time $n\geq 1$, the labels $x_1,x_2,\dots,x_n$ of the vertices created in the process are exchangeable. 
\end{proof}
\begin{figure}
	\centering
	\begin{tabular}{ccc}
		\subfloat[The initial state $\cA_G(0)=G$]{\includegraphics[height=3cm,page=1]{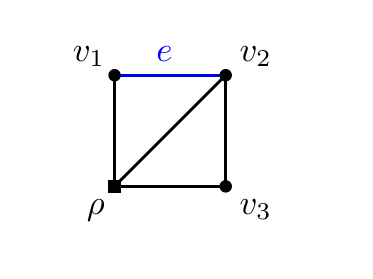}} &\subfloat[A realisation of $\cA_G(6)$]{\includegraphics[height=3cm,page=2]{Figures/GrowingGraphs/edge-splitting_process.pdf}} &\subfloat[The limiting space $\mathcal{C}_G$]{\includegraphics[height=3cm,page=3]{Figures/GrowingGraphs/edge-splitting_process.pdf}}
	\end{tabular}
\caption{An illustration of the edge-splitting process started from some graph $G$.}
\label{growing:fig:edge-splitting process}
\end{figure}

\paragraph{Construction as a gluing of decorations.}
Let us now provide a construction of a sequence $(\cD^{(n)})_{n\geq 1}$ of decorations, endowed with measures $(\bnu^{(n)})_{n\geq1}$, that satisfies the assumptions of Theorem~\ref{growing:thm:metatheorem} and for which the process $(\sG(\cD^{(n)},\bnu^{(n)}))_{n\geq 1}$ coincides with $(H_n)_{n\geq 1}$ endowed with its graph distance and the uniform measure on its vertices.  
For this, let $(\ttP_n)_{n\geq 1}$ be a preferential attachment tree with fitnesses $(a_n)_{n\geq1}$ and let the processes $\cA_k$ for $k\geq1$ be independent with the same law as $\cA_{G_k}$, defined in the above paragraph. 
Now, consider the measures $\bnu^{(n)}$ such that for all $u\in\bU$, $\nu^{(n)}_{u}$ charges every point of $\cD^{(n)}(u)$ except its root if $u\neq \emptyset$, with the same mass, normalised in such a way that the associated measure $\nu^{(n)}$ on the Ulam tree is a probability measure.
It is now an exercise to check that the sequence of graphs $(H_n)_{n\geq 1}$ seen as measured metric spaces has the same distribution as $(\sG(\cD^{(n)},\bnu^{(n)}))_{n\geq 1}$.
\paragraph{Applying the theorem.}
Assumptions \ref{growing:it:Pn preferential attachment} and \ref{growing:it:Dn processus en deguk} are satisfied by construction, so let us verify that the other ones hold as well. 
The convergence \ref{growing:it:convergence des processus A} is obtained for $\gamma=1$ by the result of Lemma~\ref{growing:lem:convergence edge-splitting}, so that for any $k\geq 1$, the limiting block $(\bB_k, \bD_k,\bRho_k,(X_{k,i})_{i\geq 1})$ has the distribution of $\cC_{G_k}=(C_G,\dist,\rho,(X_i)_{i\geq 1})$.
The control \ref{growing:it:moment condition on the sup of A} is immediate with $(c_k)_{k\geq 1}=(a_k)_{k\geq 1}$ because the diameter of a graph is smaller than its number of edges so we have the deterministic upper-bound for all $k\geq 1$ and $m\geq 0$, $\diam(\cA_k(m))\leq a_k+m$. 

The last point \ref{growing:it:measures converge} is obtained by checking that the measure $\nu^{(n)}$ has the form \eqref{growing:eq:mesures par les degrés}. 
Indeed, let $(b_n)_{n\geq 1}$ be defined such that $b_1$ is the number of vertices of $G_1$ and for $n\geq 2$, $b_n$ is the number of vertices minus $1$ of the graph $G_n$. 
Then the measures $\nu^{(n)}$ on the Ulam tree are probability measures of the form $\nu^{(n)}(u_k)\propto b_n+\deg_{\ttP_n}^+(u_k)$ for all $k\leq n$ and $\nu^{(n)}(u)=0$ on other vertices $u$.
Because the graphs $G_n$ for $n\geq 1$ are connected, their number of vertices is smaller than their number of edges minus $1$, so that for all $n\geq 1$ we have $b_n\leq a_n$ which is enough to check that $\nu^{(n)}$ is of the form \eqref{growing:eq:mesures par les degrés}. 

In the end, this yields a proof of Proposition~\ref{growing:prop:convergence rémy généralisé}.

\subsection{Generalised Rémy's algorithm, version 2}
\begin{figure}
	\begin{center}
		\begin{tabular}{cc}
			\subfloat[A realisation of $H_5$, where the leaves have been labelled by their time of creation]{\includegraphics[height=6cm,page=1]{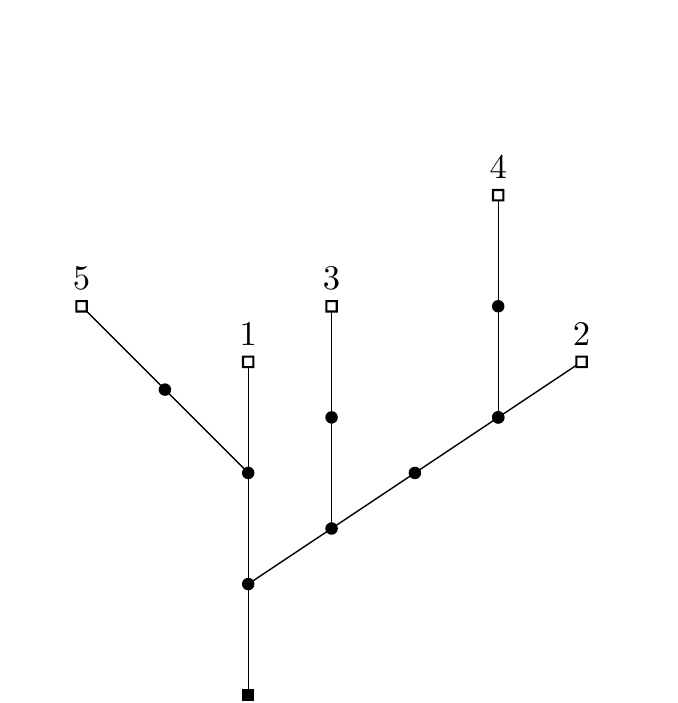}\label{growing:subfig:double edge remy}} & \subfloat[Its decomposition of the second type]{\includegraphics[height=6cm,page=2]{Figures/GrowingGraphs/double-edge_remy_algo2.pdf}\label{growing:subfig:double edge remy decompostion}} \\
		\end{tabular}
		\caption{A realisation of the tree $H_5$ and its decomposition as a decoration}\label{growing:fig:double edge remy} 
	\end{center}
\end{figure}
Let us consider the former sequence of graphs $(H_n)_{n\geq1}$ constructed from a particular sequence of graphs $(G_n)_{n\geq 1}$ with $G_1$ equal to the single-edge graph and constant starting from the second term, equal to a line with two edges, rooted at one end.
This time, we are going to use a different decomposition of this model, which will lead to another description of the limit, which will this time be described as an iterative gluing of rescaled i.i.d.\ Brownian trees. 

We decompose this process along a decoration in a slightly different way than before. Indeed, every time that we glue a new copy of the two-edge-line graph to the structure, we consider that the lower edge is a part of the block to which it attached, and only the upper edge is the newly created block of the decoration (see Figure~\ref{growing:fig:double edge remy}). 

In this decomposition, every time that an edge belonging to some block is selected by the algorithm, the graph corresponding to that block undergoes a step of the original Rémy's algorithm ,and hence its number of edges increases by two.
In order to stay in a setting where the weight of each block is reinforced by one every time that it is selected, we now see every edge as having a weight $\frac{1}{2}$. 
From this observation (and Figure~\ref{growing:fig:double edge remy}), we take $(a_n)_{n\geq 1}=(\frac12,\frac12,\dots)$ and independent processes $(\cA_k)_{k\geq1}$ that all have the same distribution as a the standard Rémy algorithm started from a single edge, such that the distinguished points correspond to the added leaves in order of creation, and on can check that the corresponding sequence of decorations $(\cD^{(n)})_{n\geq 1}$ actually describes the process $(H_n)_{n\geq 1}$.

We also add the measures $\bnu^{(n)}$ in the same way as before, that is to say that for all $n\ge1$, over all $u\in \bU$, the measure $\nu_u^{(n)}$ charges every vertex (say, including the leaves, and excluding the root if $u\neq \emptyset$) with the same mass, in such a way that for any $n\geq 1$ the sum over $u\in\bU$ of the total mass of the $\nu_u^{(n)}$'s is $1$. In this way, the measured metric space $\sG(\cD^{(n)},\bnu^{(n)})$ coincides with $H_n$ endowed with its graph distance and its uniform measure on the vertices.

Using \cite[Theorem~5]{curien_stable_2013}, for any $k\geq 1$, we have the following almost sure convergence
\begin{align}
m^{-\frac{1}{2}}\cdot\cA_k(m)\underset{m\rightarrow\infty}{\rightarrow} (\bB_k,\bD_k, \bRho_k, \left(X_{k,i}\right)_{i\geq 1}) \qquad \text{in } \M^{\infty \bullet},
\end{align}
where the limiting metric space $(\bB_k,\bD_k, \bRho_k, \left(X_{k,i}\right)_{i\geq 1})$ has the distribution of ($2$ times) the Brownian tree, endowed with an i.i.d. sequence of points taken under its mass measure.
The condition \ref{growing:it:moment condition on the sup of A} is satisfied thanks to Lemma~\ref{growing:lem:sup hauteur algo remy queue gaussienne}, proved in the appendix.
It is easy to check that the measure have the form \eqref{growing:eq:mesures par les degrés}, so that \ref{growing:it:measures converge} is satisfied and so Theorem~\ref{growing:thm:metatheorem} applies, which proves the following proposition.
\begin{proposition}\label{growing:prop:proposition generalised remy version 2}
	Under these conditions, we have the following almost sure convergence in Gromov--Hausdorff--Prokhorov topology
	\begin{align*}
			(H_n,n^{-\frac{2}{3}}\cdot\dist_{\mathrm{gr}},\mu_{\mathrm{unif}})\underset{n\rightarrow\infty}{\longrightarrow} \cH.
	\end{align*}
	The limiting space $\cH$ can be constructed by an iterative gluing construction using scaling factors $((\mathsf m^\mathbf a_n)^{\frac12})_{n\geq 1}$ and weights $(\mathsf m^\mathbf{a}_n)_{n\geq 1}$, where the sequence $(\mathsf m^\mathbf{a}_n)_{n\geq 1}$ has the distribution of the increments of a $\MLMC\left(\frac{2}{3},\frac{1}{3}\right)$ process, and i.i.d.\ blocks $((\bB_k,\bD_k, \bRho_k, \left(X_{k,i}\right)_{i\geq 1}))_{k\geq 1}$ that have the distribution of $2$ times the Brownian tree endowed with a sequence of i.i.d.\ leaves taken under its mass measure.
\end{proposition}

Remark that Proposition~\ref{growing:prop:convergence rémy généralisé} describes the limit as an iterative gluing construction using blocks that are all equal to the $\intervalleff{0}{1}$ interval rooted at $0$ endowed with i.i.d. random points.
The associated sequence of scaling factors and weights would then be both equal to a sequence $(\mathsf m_n)_{n\geq 1}$ having the distribution of the increment of a $\MLMC(\frac{1}{3},\frac{1}{3})$ Mittag-Leffler Markov chain. 
Proposition~\ref{growing:prop:proposition generalised remy version 2} proves in particular the non-trivial fact that the two iterative gluing constructions with segments or Brownian trees lead to the same object. 

\subsection{Marchal algorithm started from an arbitrary seed}
Let us define Marchal's algorithm started from a rooted connected multigraph $G$, as introduced in \cite{goldschmidt_stable_2018}, following the same idea as \cite{marchal_note_2008}. 
Fix $\alpha\in\intervalleoo{1}{2}$. We let $H^\alpha_1=G$ and for each $n\geq 1$ define $H^{\alpha}_{n+1}$ recursively. 
If $H^{\alpha}_n$ is defined then 
take a vertex or an edge with probability proportional to their weight, where their weight are defined as
\begin{itemize}
	\item $\alpha-1$ for any edge,
	\item $\deg(v)-1-\alpha$ for a vertex $v$ with degree $3$ or more,
	\item $0$ for a vertex of degree $2$ or less.
\end{itemize} 
Then if it is an edge, split this edge into $2$ edges with a common endpoint and add an edge linking that newly created vertex to a new leaf. 
Otherwise, attach an edge linking the selected vertex to a new leaf.
The obtained graph is then $H^\alpha_{n+1}$.

This construction has recently been studied in \cite{goldschmidt_stable_2018}, whose results already ensure that these graphs appropriately rescaled converge in the GHP topology to some random object that is constructed using an iterative gluing construction. 
In this case, there are two natural interpretation of this graph process in terms of decorations on the Ulam tree, which give two different descriptions of the limiting object: one of them coincides with the one given in \cite{goldschmidt_stable_2018}, but the other one is different. 

Let us describe the difference between the two using Figure~\ref{growing:fig:different description for marchal}. Consider Marchal's algorithm started from the single-edge graph and label the non-root leaves by their order of appearance. 
One way of describing the process in a way that is handled by Theorem~\ref{growing:thm:metatheorem} is the following, which is represented in Figure~\ref{growing:subfig:weight along the segments}: every time that a new leaf is added, the newly created block consists of just one edge with Marchal weight $\alpha-1$, whose two extremities have no weight. 
The edge or vertex that was selected when adding this new leaf belongs to some block; if it was a vertex then the weight of this vertex is reinforced by $1$ in its own block; if it was an edge then a new vertex of weight $2-\alpha$ and a new edge of weight $\alpha-1$ are created in the block. In both cases, the total Marchal weight of that block is reinforced by $1$. Using this description, in the limit, the blocks (other than the first one possibly, if we start with an arbitrary seed) are described as segments with a countable number of atoms of weight along them.

The second way of describing the process is the following, represented in Figure~\ref{growing:subfig:weight on the roots}. Every time that a new leaf is added, the newly created block consists one edge with Marchal weight $\alpha-1$, rooted at a point of weight $2-\alpha$. 
The edge or vertex that was selected when adding this new leaf belongs to some block; if it was a vertex then the weight of this vertex is reinforced by $\alpha-1$ in its own block; if it was an edge then a new vertex and a new edge are created in the block; the vertex has no weight and the edge has weight $\alpha-1$. 
In both cases, the total Marchal weight of that block is reinforced by $\alpha-1$. Using this description, in the limit, the blocks (other that the first one) are described as segments with an atom at their root. 

\begin{figure}
\begin{tabular}{ccc}
	\subfloat[The tree $H_4^\alpha$\label{growing:subfig:marchal tree}]{\includegraphics[height=4cm,page=1]{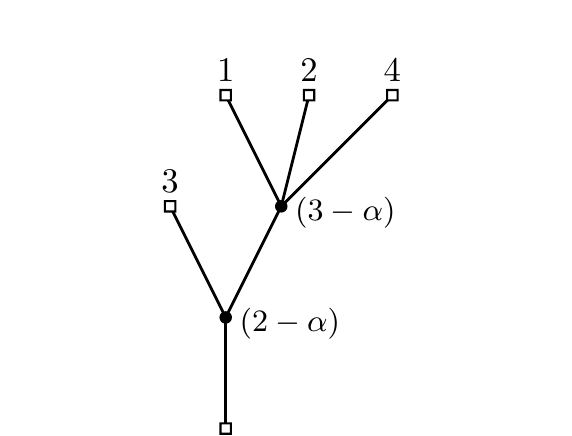}} & 	\subfloat[A decomposition of $H_4^\alpha$\label{growing:subfig:weight along the segments}]{\includegraphics[height=4cm,page=2]{Figures/GrowingGraphs/Marchal_weights.pdf}} &
	\subfloat[Another decomposition of $H_4^\alpha$\label{growing:subfig:weight on the roots}]{\includegraphics[height=4cm,page=3]{Figures/GrowingGraphs/Marchal_weights.pdf}}
\end{tabular}
\caption{Two different decomposition of a tree constructed using Marchal's algorithm as a gluing along a decoration. The filled-in vertices have non-zero Marchal weight, indicated in parenthesis. All the edges have Marchal weight $\alpha-1$.}\label{growing:fig:different description for marchal}
\end{figure}

\paragraph{Notation.} In what follows, we write $w(G)$ for the sum of the weights of all vertices and edges of a multigraph $G$. Note that if $G$ has surplus $s$ and $\ell$ vertices of degree $1$ and $m$ vertices of degree $2$ then $w(G)=(\ell-1)\alpha+m(\alpha-1)+s(\alpha+1)-1.$
Also we denote by the symbol $-$ the graph with only one edge and two endpoints.

\subsubsection{Splitting the width}\label{growing:subsec:splitting the width}
In the first decomposition, we take $(a_n)=(w(G),\alpha-1,\alpha-1,\dots)$ and $(\ttP_n)_{n\geq 1}$ taken as a preferential attachment tree $\pa(\mathbf a)$. 
The processes $(\cA_k,k\geq 1)$ all follow the same Markov transitions on weighted graphs, starting from the rooted graph $G$ in the case of $\cA_1$ and from the single-edge graph, which we denote $-$, in the case of $\cA_k$ for any $k\geq 2$. 
We denote $(\cA_G^\alpha(n),n\geq 1)$ such a process started from the seed graph $G$.
 
In this setting, the weight of an edge is always $\alpha-1$ but the weight of vertices evolves with time. At time $0$, for any seed graph $G$, every vertex with degree $d\geq 3$ is given weight $d-1-\alpha$ and other vertices have weight $0$.
The process evolves then in the following way: to obtain $\cA_G^\alpha(m+1)$ from $\cA(m)$ we choose at random a an edge or a vertex proportionally to their weights:
\begin{itemize}
\item if it is a vertex $x$ then $\cA_G^\alpha(m+1)$ is obtained from $\cA_G^\alpha(m)$ by setting $x_{m+1}$ its $(m+1)$-st distinguished point to $x$, and incrementing the weight of $x$ by one,
\item if it is an edge, then $\cA_G^\alpha(m+1)$ is obtained from $\cA_G^\alpha(m)$ by splitting this edge in $2$ by adding a new vertex $x$, giving weight $2-\alpha$ to this vertex, and setting $x_{m+1}$ its $(m+1)$-st distinguished point to $x$.
\end{itemize}

\begin{lemma}\label{growing:lem:convergence width splitting}
With this dynamics, for $\cA_G^\alpha$ for any connected graph $G$ with at least one edge, we have the following convergence almost surely in $\M^{\infty \bullet}$,
\begin{align}\label{growing:eq:convergence block marchal width splitting}
m^{-(\alpha -1)}\cdot \cA_G^\alpha(m)\underset{n\rightarrow\infty}{\longrightarrow} \mathcal{C}^\alpha_G=(C^\alpha_G,\dist,\rho,(X_i)_{i\geq 1}),
\end{align}
almost surely as in $\M^{\infty\bullet}$, where the distribution of the limiting object is described in the paragraph below.
Moreover, for any $p\geq 0$, we have
\begin{align}\label{growing:eq:control diameter width splitting}
\Ec{\left(\sup_{m\geq 0} \frac{\diam\cA_G^\alpha(m)}{m^{\alpha -1}}\right)^p}< + \infty.
\end{align}
\end{lemma}

\paragraph{Limiting block.}
For any connected multigraph $G$ with at least one edge, let us describe the law of the random metric space \[\mathcal{C}^\alpha_G=(C^\alpha_G,\dist,\rho,(X_i)_{i\geq 1}).\] 
It is a continuous version of $G$ meaning that we define it by replacing every edge of $G$ by a segment of some length. We label its edges $e_1,e_2,\dots e_{\abs{E}}$ in arbitrary order and replace each edge $e$ with a segment of length $L(e)$, whose distribution is characterised by what follows.
We let $I$ be the set of vertices of $G$ that have degree greater than $3$ and write $I=\{v_1, v_2,...,v_{\abs{I}}\}$.
All the random variables used in the construction are supposed to be independent.
\begin{itemize}
	\item We let \[ \left(W_{E},W_{v_1},\dots,W_{v_{\abs{I}}}\right)\sim \mathrm{Dir}\left(\abs{E}\cdot(\alpha -1),d_{v_1}-1-\alpha,\dots,d_{v_{\abs{V}}}-1-\alpha\right)\]
	\item We let
	\[(Q_{j})_{j\geq1}\sim \mathrm{PD}\left(\alpha-1,\abs{E}\cdot (\alpha-1)\right), \quad \text{and} \quad S \quad \text{its $(\alpha-1)$-diversity},\]
	\item The length of the edges are defined as 
	\[(L(e_1),L(e_2),\dots,L(e_{\abs{E}}))= W_E^{\alpha-1} \cdot S\cdot (B_1,B_2,\dots,B_{\abs{E}}),\]
	with $(B_1,B_2,\dots,B_{\abs{E}})\sim \mathrm{Dir}(1,1,\dots 1)$.
	\item Conditionally on the lengths, let $(Z_{j})_{j\geq1}$ be independent and uniformly distributed on the total length of the graph.
	\item We set \[\mu=\sum_{v\in I}W_v \delta_{v}+ W_E\cdot \sum_{j=1}^{\infty}Q_{j}\delta_{Z_{j}},\]
	and conditionally on all the rest, the points $(X_i)_{i\geq 1}$ are obtained as i.i.d.\ samples under the probability measure $\mu$.
\end{itemize}
\begin{proof}[Proof of Lemma~\ref{growing:lem:convergence width splitting}]
The proof of this convergence is based on convergence results for P\'olya's urn and Chinese Restaurant Processes, recalled in Section~\ref{growing:subsection:classical polya urn} and Section~\ref{growing:subsection:chinese restaurant process} of the appendix.
Indeed, let us study the evolution of the process $(\cA(m))_{m\geq 0}$ started from $\cA(0)=G$. For any internal vertex $v\in I$ and $n\geq 0$, we let $W_v(n)$ be the weight of that vertex in the weighted graph $\cA(n)$, and $W_E(n)$ be the sum of the weight of all the other parts of the graph $\cA(n)$, that is, the sum of the weight of all edges and vertices of degree $2$. 

From the definition of the process, the vector  $\left(W_{E}(n),W_{v_1}(n),\dots,W_{v_{\abs{I}}}(n)\right)$ evolves as the weight of colours in a P\'olya urn (whose definition is recalled in Section~\ref{growing:subsection:classical polya urn}) with starting proportions $\left(\abs{E}\cdot(\alpha -1),\deg(v_1)-1-\alpha,\dots,\deg(v_{\abs{I}})-1-\alpha\right)$ so thanks to Theorem~\ref{growing:thm:classical Polya urn}, 
\[\frac{1}{n}\cdot \left(W_{E}(n),W_{v_1}(n),\dots,W_{v_{\abs{I}}(n)}\right)\underset{n\rightarrow\infty}{\rightarrow} \left(W_{E},W_{v_1},\dots,W_{v_{\abs{I}}}\right),\]
with $\left(W_{E},W_{v_1},\dots,W_{v_{\abs{I}}}\right)\sim \mathrm{Dir}\left(\abs{E}\cdot(\alpha -1),d_{v_1}-1-\alpha,\dots,d_{v_{\abs{V}}}-1-\alpha\right)$.
Also, conditionally on $\left(W_{E},W_{v_1},\dots,W_{v_{\abs{I}}}\right)$ the choice at each step of the construction is i.i.d such that any vertex interval $v$ is chosen with probability $W_{v}$ and something else (edge or vertex of degree $2$) is chosen with probability $W_{E}$.

Then, let us define $y_1,y_2,\cdots$ the vertices of degree $2$ created by the process in order of appearance. For any $n\geq 0$ and $i\geq 1$ we let $N_i(n)=\#\enstq{1\leq k\leq n}{x_k=y_i}$, the number of time that $y_i$ appears in the list $x_1,x_2,\dots,x_n$ of distinguished points of $\cA(n)$.  
Up to a time-change that depends on the sequence $(W_{E}(n))_{n\geq 1}$, the sequence $(N_i(n),i\geq 1)$ evolves with time like the number of customers sitting at each table in a Chinese Restaurant Process with seating plan $(\alpha-1,\abs{E}\cdot (\alpha-1))$, so thanks to Theorem~\ref{growing:thm:convergence chinese restaurant process}, we have the following convergence almost surely in $\ell^1$ 
\[\frac{1}{W_E(n)}\cdot (N_i(n))_{i\geq 1}\underset{n\rightarrow\infty}{\rightarrow} (P_i)_{i\geq 1},\]
with  $(P_i)_{i\geq 1}\sim \mathrm{GEM}\left(\alpha-1,\abs{E}\cdot (\alpha-1)\right).$

Conditionally on the sequence $(Q_j)_{j\geq 1}$, which we define as the decreasing rearrangement of the sequence $(P_i)_{i\geq 1}$, Theorem~\ref{growing:thm:exhangeability chinese restaurant process} gives a description of this Chinese Restaurant Process: at any moment, when a new customer enters the restaurant they receive a label in such a way that each label $j\geq 1$ has probability $Q_j$, independently of the other costumers. If there is already a customer with that label in the restaurant, the new customer joins them, otherwise they sit at a new table.    

The total number of edges $L_E(n)$ in the graph $\cA(n)$, up to an additive constant, is equal to the number of vertices created before that time. This corresponds to that number of tables in the Chinese Restaurant Process defined above. Using Theorem~\ref{growing:thm:convergence chinese restaurant process}, we get
\begin{align*}
	\frac{L_E(n)}{n^{\alpha-1}}=\frac{L_E(n)}{(W_E(n))^{\alpha-1}}\cdot \frac{W_E(n)^{\alpha -1}}{n^{\alpha-1}} \underset{n\rightarrow\infty}{\rightarrow} W_E^{\alpha-1} \cdot S,
\end{align*} 
where $S$ is the $(\alpha-1)$-diversity of the sequence $(P_i)_{i\geq 1}$. By using Lemma~\ref{growing:lem:number of tables in a CRP} we also have, for any $p\geq 1$, 
\begin{align*}
	\Ec{\left(\frac{L_E(n)}{n^{\alpha-1}}\right)^p}<\infty.
\end{align*}
Since $L_E(n)$ is an upper-bound for the diameter of $\cA(n)$, this already proves the claim \eqref{growing:eq:control diameter width splitting}.

Now, just looking at the shape of the graph $\cA(n)$, we notice that, up to a time-change that is deduced from the sequence $(L_E(n))_{n\geq 0}$, this graph evolves under the uniform edge-splitting process described in Lemma~\ref{growing:lem:convergence edge-splitting}. Using the result of this lemma, we deduce that almost surely 
\begin{align*}
	\frac{1}{L_E(n)}(L(e_1,n),L(e_2,n),\dots,L(e_E,n)) \underset{n\rightarrow\infty}{\rightarrow} (B_1,B_2,\dots,B_{\abs{E}}),
\end{align*}
where $(B_1,B_2,\dots,B_{\abs{E}})\sim \mathrm{Dir}(1,1,\dots 1)$. 
Also, the vertices created during the process ranked in order of creation, which we denote $(Y_i)_{i\geq 1}$, are in the limit positioned i.i.d.\ uniformly along the length of the limiting space, independently of $(P_i)_{i\geq 1}$
We define the points $(Z_j)_{j\geq 1}$ as the points such that $((Q_j,Z_j))_{j\geq 1}$ is the non-increasing reordering of $((P_i,Y_i))_{i\geq 1}$ with respect to the first coordinate (there are almost surely no ties), hence by definition $\sum_{i=1}^{\infty}Q_{j}\delta_{Z_{j}}=\sum_{i=1}^{\infty}P_{i}\delta_{Y_{i}}$.

In the end, we just have to justify that the points $(X_i)_{i\geq 1}$ are indeed i.i.d.\ taken under the measure $\mu=\sum_{v\in I}W_v \delta_{v}+ W_E\cdot\sum_{i=1}^{\infty}Q_{j}\delta_{Z_{j}}$, conditionally on all the rest.
This follows from the description of the evolution as $n\rightarrow\infty$ of the sequences $\left(W_{E}(n),W_{v_1}(n),\dots,W_{v_{\abs{I}}}(n)\right)$ and $(N_i(n),i\geq 1)_{n\geq 1}$ when conditioned on their limit. 
Indeed, at any time $n\geq 1$, conditionally on $\left(W_{E},W_{v_1},\dots,W_{v_{\abs{I}}}\right)$ and $(Q_j)_{j\geq 1}$, the next distinguished vertex $x_{n+1}$ is either any $v\in I$ with probability $W_v$, or a vertex of degree $2$ with probability $W_E$. 
If it is a vertex of degree $2$, then it corresponds to any of the tables $j\geq 1$ with asymptotic size $Q_j$ with probability $Q_j$, independently of the other ones. 
Thanks to the above paragraph, the limiting positions $(Z_j)_{j\geq 1}$ of the vertices corresponding to those tables are i.i.d. along the length of the limiting graph, which concludes the proof. 
\end{proof}
\paragraph{Convergence result.}
We get the following convergence.
\begin{proposition}\label{growing:prop:convergence marchal algo 1} 
	The graphs $H_n^\alpha$ endowed with the uniform measure on their vertices converge almost surely in the scaling limit for the Gromov-Hausdorff-Prokhorov topology,
	\begin{align*}
	(H_n^\alpha,n^{1-1/\alpha}\cdot\dist_{\mathrm{gr}},\mu_{\mathrm{unif}})  \underset{n\rightarrow\infty}{\longrightarrow} \cH^\alpha_G.
	\end{align*}
The distribution of the limiting space $\cH^\alpha_G$ is obtained as an iterative gluing construction with blocks 
	\begin{align*}
	(\bB_1,\bD_1,\bRho_1,(X_{1,i})_{i\geq 1})\overset{\mathrm{(d)}}{=}\mathcal{C}^\alpha_G \quad \text{and} \quad \forall n\geq 2,\  (\bB_n,\bD_n,\bRho_n,(X_{n,i})_{i\geq 1})\overset{\mathrm{(d)}}{=}\mathcal{C}^\alpha_-
	\end{align*}
	and sequence of scaling factors $(\mathsf m_k^{(\alpha-1)})_{k\geq 1}$ and weights $(\mathsf m_k)_{k\geq 1}$, where $(\mathsf m_k)_{k\geq 1}$ is obtained as the increments of a $\MLMC(\frac{1}{\alpha},\frac{w(G)}{\alpha})$. 
\end{proposition}
\begin{proof}
The proof of this proposition is another application of Theorem~\ref{growing:thm:metatheorem}. Conditions \ref{growing:it:Pn preferential attachment} and \ref{growing:it:Dn processus en deguk} are satisfied thanks to the discussion at the beginning of the section, conditions \ref{growing:it:convergence des processus A} and \ref{growing:it:moment condition on the sup of A} are satisfied thanks to Lemma~\ref{growing:lem:convergence width splitting}.

Last, we have to verify that \ref{growing:it:measures converge} is satisfied for some measure-valued decoration $\bnu^{(n)}$ such that $\sG(\cD^{(n)},\bnu^{(n)})$ coincides with $(H_n)_{n\geq 1}$ endowed with its graph distance and its uniform measure on the vertices.
A choice for $\bnu^{(n)}$ is achieved, as in the other examples, by letting for every $n\geq 1$, for all $u\in\bU$, $\nu_u^{(n)}$ charges every vertex of $\cD^{(n)}(u)$, except its root if $u\neq \emptyset$, with the same mass. 

The slight technical difficulty here is that the total number of vertices $\abs{V(H_n^\alpha)}$ in $H_n^\alpha$ is random, and the convergence of $\nu^{(n)}$ towards $\mu$ is not directly ensured by Proposition~\ref{growing:prop:convergence measures on wrt}. 
To handle this, we are going to prove that the number of vertices in this algorithm grows asymptotically linearly like a constant multiple of $n$. 
Indeed, let us define $X_n$ as the total weight of edges in $H_n^\alpha$ and $Y_n$ the total weight of vertices. 
Every time that an edge is picked by the algorithm, two edges of weight $(\alpha-1)$ and one vertex of weight $(2-\alpha)$ are created; when a vertex is picked, one edge of weight $(\alpha-1)$ is created and the weight of a vertex is reinforced by $1$. This indicates that the couple $(X_n,Y_n)$ evolves like a generalised P\'olya urn with replacement matrix \[\begin{bmatrix}
2(\alpha -1) &2-\alpha\\ \alpha -1 &1
\end{bmatrix}\] as described in Section~\ref{growing:subsection:classical polya urn} of the appendix. Using Lemma~\ref{growing:lem:convergence generalised polya urn}, the weight of edges almost surely grows like $\alpha(\alpha-1)n$ so the number of edges is asymptotically $\alpha n$ and so is the number of vertices, hence almost surely
\begin{align*}
	\abs{V(H_n^\alpha)}\underset{n\rightarrow\infty}{\sim} n \alpha.
\end{align*}

Now, consider $u\in \bU$, and let us investigate the behaviour of $\nu^{(n)}(T(u))$, where we recall the notation $T(u)=\enstq{v\in\bU}{u\preceq v}$.
We let $\nu_n$ be the uniform measure on the vertices of $\ttP_n$, in such a way that $n\nu_n(T(u))$ is the number of vertices in $\ttP_n$ that are above $u$. Thanks to Proposition~\ref{growing:prop:convergence measures on wrt} the sequence $(\nu_n)_{n\geq 1}$ almost surely converges towards the measure $\mu$ associated to $(\ttP_n)_{n\geq 1}$.
We want to show that almost surely, for any choice of $u\in\bU$,
\begin{align}\label{growing:eq:number of vertices number of leaves}
\nu_n(T(u))\underset{n\rightarrow\infty}{\sim} \nu^{(n)}(T(u)), \qquad \text{and} \qquad \nu^{(n)}(\{u\})\underset{n\rightarrow\infty}{\rightarrow}0,
\end{align}
which is enough to prove that the sequence $(\nu^{(n)})_{n\geq 1}$ converges to the same limit $\mu$ as $(\nu_n)_{n\geq 1}$ almost surely. Note that the second requirement is immediate from the fact that the number of vertices in $\cD^{(n)}(u)$ corresponds up to a constant to the out-degree $\deg_{\ttP_n}^{+}(u)$, which grows sub-linearly.

Now, the quantity $\abs{V(H_n^\alpha)} \cdot\nu^{(n)}(T(u))$ is the number of vertices in all the blocks $\cD^{(n)}(v)$ for $v\succeq u$. By self-similarity of the process, the reasoning made above for the total weight of edges also applies, and so the number of vertices in decorations $\cD^{(n)}(v)$ for $u\preceq v$ is asymptotically equivalent to the number of times that the algorithm picked an element in those decorations, which corresponds to the number $n\nu_n(T(u))$ of vertices of $\ttP_n$ above $u$. Hence
\begin{align*}
		\alpha n \nu^{(n)}(T(u))\underset{n\rightarrow\infty}{\sim} \abs{V(H_n^\alpha)} \cdot\nu^{(n)}(T(u))\underset{n\rightarrow\infty}{\sim} \alpha n \nu_n (T(u)),
\end{align*}
which leads to \eqref{growing:eq:number of vertices number of leaves}, and hence finishes the proof.

\end{proof}
\subsubsection{Another decomposition}
Using another decomposition into decorations we retrieve the other description of the limiting space, which was proved in \cite{goldschmidt_stable_2018}. It is obtained as an iterative gluing construction with blocks (the distribution of which we define below) 
	\begin{align*}
	(\bB_1,\bD_1,\bRho_1,\bNu_1) \overset{\mathrm{(d)}}{=}\mathcal{C}^{\mathrm{len},\alpha}_G \quad \text{and} \quad \forall n\geq 2,\  (\bB_n,\bD_n,\bRho_n,\bNu_n)\overset{\mathrm{(d)}}{=} \mathcal{C}^{\mathrm{len},\alpha}_{\bullet -},
	\end{align*}
	and sequence of weights and scaling factors $(\mathsf m_k)_{k\geq 1}$, where $(\mathsf m_k)_{k\geq 1}$ is obtained as the increments of a $\MLMC(1-\frac{1}{\alpha},\frac{w(G)}{\alpha})$. 

Let us describe the random metric space $\mathcal{C}^{\mathrm{len},\alpha}_G=(C^{\mathrm{len},\alpha}_G,\dist,\rho,(X_i)_{i\geq 1})$, for any rooted connected multigraph $G$. 
As before, we let $I$ be the set of vertices of $G$ that have degree greater than $3$ and write $I=\{v_1, v_2,...,v_{\abs{I}}\}$ and we arbitrarily label its edges $e_1,e_2,\dots e_{\abs{E}}$. Then
\begin{itemize}
	\item we define \[ \left(L(e_1),L(e_2)\dots,L(e_{\abs{E}}),L(v_1),\dots,L(v_{\abs{I}})\right)\sim \mathrm{Dir}\left(1,1,\dots,1,\frac{d_{v_1}-1-\alpha}{\alpha-1},\dots,\frac{d_{v_{\abs{I}}}-1-\alpha}{\alpha-1}\right),\]
	\item and set \[\nu=\sum_{v\in I}L(v)\cdot \delta_{v}+ \mu_{\mathrm{len}},\]
	where $\mu_{\mathrm{len}}$ is the length measure on the structure. Conditionally on all the rest, the sequence $(X_i)_{i\geq1}$ is i.i.d. with distribution $\nu$.
\end{itemize}
For the single-edge graph, this yields a segment of unit length endowed with the uniform measure. We also introduce a variant of this one. We define  $\mathcal{C}^{\mathrm{len},\alpha}_{\bullet -}=(C^{\mathrm{len},\alpha}_{\bullet -},\dist,\rho,(X_i)_{i\geq 1})$ as follows
\begin{itemize}
	\item we set \[ \left(L(e),L(\rho)\right)\sim \mathrm{Dir}\left(1,\frac{2-\alpha}{\alpha-1}\right)\]
	\item and \[\nu=L_\rho \delta_{\rho}+ \mu_{\mathrm{len}},\]
	where $\mu_{\mathrm{len}}$ is the length measure on the structure. Conditionally on all the rest, the sequence $(X_i)_{i\geq1}$ is i.i.d. with distribution $\nu$.
\end{itemize}
We do not provide another proof of the convergence because the result is already know from \cite{goldschmidt_stable_2018}. However, it could be done quite easily using Lemma~\ref{growing:lem:convergence edge-splitting} and arguments involving P\'olya urns and Theorem~\ref{growing:thm:classical Polya urn}. This is left to the reader.

\subsection{Scaling limits for growing trees and/or their looptrees.}
The looptree $\mathrm{Loop}(\tau)$ of a plane tree $\tau$ is a multigraph constructed from $\tau$ as follows: we first place a blue vertex in the middle of every edge of the tree $\tau$. Then, we connect two blue vertices if they correspond to two consecutive edges according to the cyclic ordering around vertices that are not the root. Then $\mathrm{Loop}(\tau)$ is obtained by removing all the vertices and edges that belong to the tree $\tau$, see Figure~\ref{growing:subfig:realisation of T6} and Figure~\ref{growing:subfig:construction looptree T6} for an example. 
An informal way of describing that construction is to say that every (non-root) vertex of the tree is replaced by a \emph{loop} that has the same length as the degree of that vertex. 

Whenever we work with a model of tree that has degrees that grow to infinity, studying the associated looptrees may allow to pass this information to the limit in terms of metric scaling limits. 
Among the two models that we present here, one of them admits scaling limits for both the tree itself and its associated looptree. For the other one, only the looptree behaves well in this sense.

\subsubsection{The $\alpha-\gamma-$growth model}
Fix $\alpha\in \intervalleoo{0}{1}$ and $\gamma\in \intervalleof{0}{\alpha}$. The $\alpha-\gamma$-growth model is defined as follows: $T_1^{\alpha,\gamma}$ is a  
tree with a single edge. 
Then if $T_n^{\alpha,\gamma}$ is already constructed, take an edge or a vertex at random with probability proportional to
\begin{itemize}
\item $1-\alpha$ for edges that are adjacent to a leaf,
\item $\gamma$ for edges that are not adjacent to a leaf,
\item $(d-2)\alpha -\gamma$ for every vertex of degree $d\geq 3$. 
\end{itemize} 
Then as in Marchal's algorithm, if an edge is chosen, it is split into two edges and a new edge leading to a new leaf is grafted at the newly created vertex. If a vertex is chosen, add an edge connecting it to a new leaf. Remark that for $\gamma=1-\alpha$, this algorithm corresponds to Marchal's algorithm with parameter $\frac{1}{\gamma}$.
We consider a planar variation of this algorithm, where every time that we attach a new leaf to a vertex we attach it in a uniform corner around this vertex. This makes it possible to consider $T_n^{\alpha,\gamma}$ as a plane tree and hence also define its corresponding looptree $\mathrm{Loop}(T_n^{\alpha,\gamma})$. 
\paragraph{Decomposition along a preferential attachment tree.}
\begin{figure}
\centering
\begin{tabular}{cc}
	\subfloat[A realisation of $T_6^{\alpha,\gamma}$, with leaves labelled with their time of creation \label{growing:subfig:realisation of T6}]{\includegraphics[height=6cm,page=1]{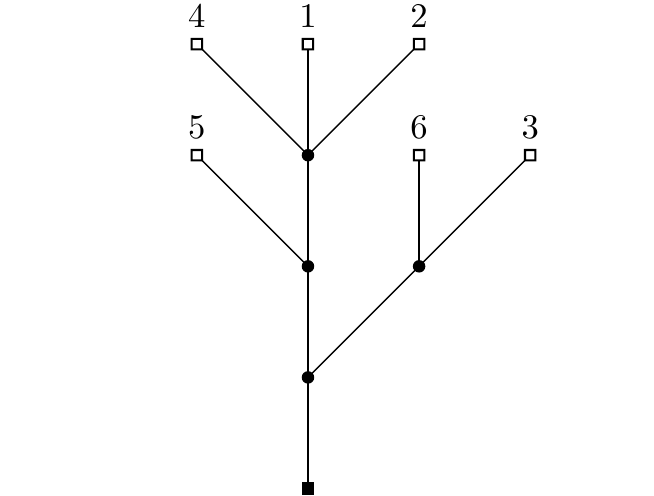}} & \subfloat[The corresponding decomposition as a gluing of decoration]{\includegraphics[height=6cm,page=2]{Figures/GrowingGraphs/looptree_decomposition}}\\
		\subfloat[The looptree $\mathrm{Loop}(T_6^{\alpha,\gamma})$ with the tree $T_6^{\alpha,\gamma}$ shown in gray\label{growing:subfig:construction looptree T6}]{\includegraphics[height=6cm,page=4]{Figures/GrowingGraphs/looptree_decomposition}} & \subfloat[Its decomposition as a gluing of a decoration]{\includegraphics[height=6cm,page=5]{Figures/GrowingGraphs/looptree_decomposition}}
\end{tabular}
	\caption{Decomposition of $T_n^{\alpha,\gamma}$ and its associated looptree}\label{growing:fig:decomposition alpha gamma tree looptree}
\end{figure}
Let us use the same decomposition of our trees $T_n^{\alpha,\gamma}$ (illustrated in Figure~\ref{growing:fig:decomposition alpha gamma tree looptree}) as we did for Marchal's algorithm in Section~\ref{growing:subsec:splitting the width}: every time that a new leaf is added, we create a new block that consists of just one edge with weight $1-\alpha$, whose two extremities have no weight. 
The edge or vertex that was selected when adding this new leaf belongs to some block; if it was a vertex then the weight of this vertex is reinforced by $\alpha$ in its own block; if it was an edge then a new vertex of weight $\alpha-\gamma$ and a new internal edge of weight $\gamma$ are created in the block. In both cases, the total weight of that block is reinforced by $\alpha$. 

As for the corresponding looptrees, we use a similar decomposition: for every block in the decomposition of $T_n^{\alpha,\gamma}$, we replace every vertex that is not the root of its block by the corresponding loop in $\mathrm{Loop}(T_n^{\alpha,\gamma})$, as displayed in Figure~\ref{growing:fig:decomposition alpha gamma tree looptree}. This gives rise to two decorations $\cD^{\mathrm{tree},(n)}$ and $\cD^{\mathrm{tree},(n)}$ that are constructed jointly and have the same support, as defined in Section~\ref{growing:subsec:decorations on the ulam tree}.

In the end, in order to be in the situation described by Theorem~\ref{growing:thm:metatheorem}, we let $\mathbf{a}=(\frac{1-\alpha}{\alpha},\frac{1-\alpha}{\alpha},\frac{1-\alpha}{\alpha},\dots)$. 
Let us describe the corresponding Markov processes $\cA^{\mathrm{tree}}$ and $\cA^{\mathrm{loop}}$ that govern the evolution of the blocks in one decoration and the other, and it is quite natural to describe them jointly.
 
At time $0$, we have $\cA^{\mathrm{tree}}(0)$ is a single-edge with weight $1-\alpha$, rooted at one end, and $\cA^{\mathrm{loop}}(0)$ is a single self-loop on one vertex. Then, we choose at random a an edge or a vertex proportionally to their weights:
\begin{itemize}
	\item if it is a vertex $x$ then $\cA^{\mathrm{tree}}(m+1)$ is obtained from $\cA^{\mathrm{tree}}(m)$ by setting $x_{m+1}^{\mathrm{tree}}$ its $(m+1)$-st distinguished point to $x$, and incrementing the weight of $x$ by $\alpha$. In the loop corresponding to that vertex in $\cA^{\mathrm{loop}}(m)$, an edge is chosen uniformly at random, split in two by the addition of a new vertex $y$, and this vertex $y$ becomes $x_{m+1}^{\mathrm{loop}}$ the $(m+1)$-st distinguished point of $\cA^{\mathrm{loop}}(m+1)$.
	\item If it is an edge, then $\cA^{\mathrm{tree}}(m+1)$ is obtained from $\cA^{\mathrm{tree}}(m)$ by splitting this edge in $2$ by adding a new vertex $x$, giving weight $\alpha-\gamma$ to this vertex, and setting its $(m+1)$-st distinguished point to $x$. Among the two edges that result from the splitting, the one further from the root has weight $1-\alpha$ and the other one has weight $\gamma$.
	In $\cA^{\mathrm{loop}}(m)$, the addition of this new vertex corresponds to the creation of a new loop of length $3$ in between the two loops that correspond to the two endpoints of the edge that was duplicated. The $(m+1)$-st distinguished point $x_{m+1}^{\mathrm{loop}}$ of $\cA^{\mathrm{loop}}(m+1)$ is the new vertex of degree two that was created in the process.
\end{itemize}

\begin{lemma}\label{growing:lem:convergence block tree looptree}
With this dynamics, the processes $\cA^{\mathrm{tree}}(m)$ and $\cA^{\mathrm{loop}}(m)$ admit an almost sure joint scaling limit in $\M^{\infty\bullet}$ as $m\rightarrow\infty$: 
\begin{align*}
	m^{-\frac{\gamma}{\alpha}}\cdot\cA^{\mathrm{tree}}(m) &\underset{n\rightarrow\infty}{\rightarrow}\cS^{\alpha,\gamma}, \quad \text{and} \quad
	m^{-1}\cdot \cA^{\mathrm{loop}}(m) \underset{n\rightarrow\infty}{\rightarrow}\cB^{\alpha,\gamma},
\end{align*}
where the joint law of the limiting objects $(\cS^{\alpha,\gamma},\cB^{\alpha,\gamma})$ is defined below. Moreover, for any $p\geq 0$, we have
\begin{align}\label{growing:eq:control diameter loopspine}
\Ec{\left(\sup_{m\geq 0} \frac{\diam\cA^{\mathrm{tree}}(m)}{m^{\frac{\gamma}{\alpha}}}\right)^p}<+ \infty.
\end{align}
\end{lemma}
\paragraph{Joint construction of the limiting blocks.}
\begin{figure}
	\centering
	\begin{tabular}{c}
			\subfloat[The block $\cB^{\alpha,\gamma}$ constructed from a countable number of circles $(C_n)_{n\geq 1}$]{\includegraphics[height=4cm,page=2]{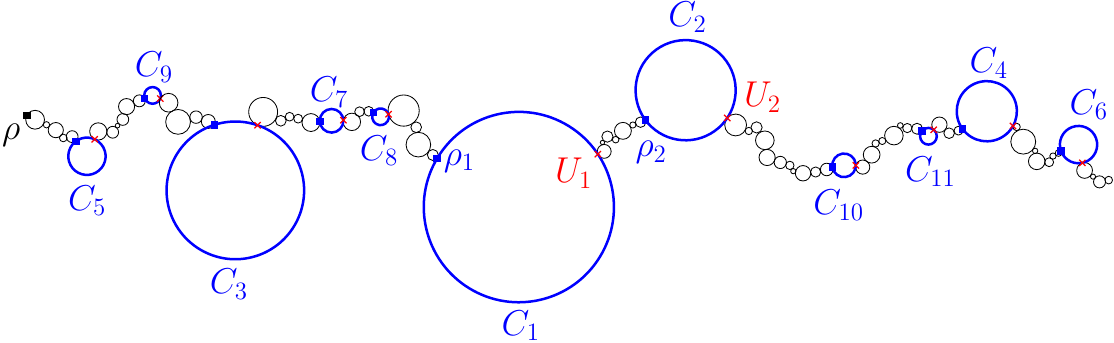}}\\
		\subfloat[The points $(Y_n)_{n\geq 1}$ along the segment $\intervalleff{0}{1}$ ]{\includegraphics[height=4cm,page=3]{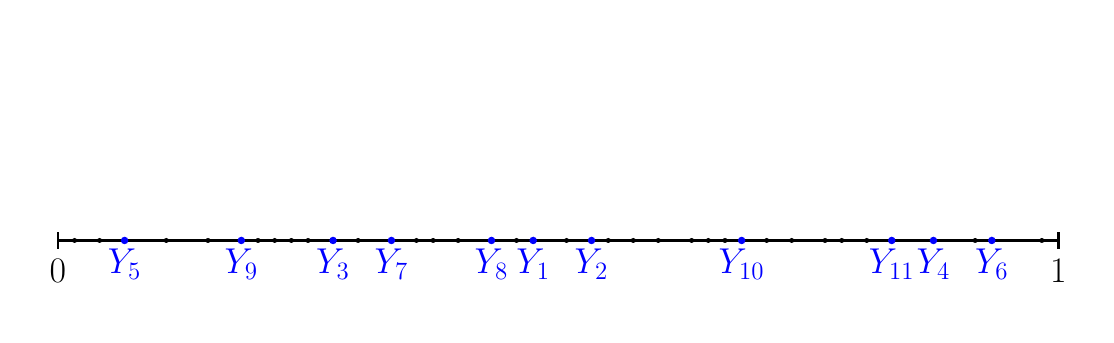}}
	\end{tabular}
\caption{The block $\cB^{\alpha,\gamma}$ is constructed by agglomerating countably many circles $(C_n)_{n\geq 1}$, in the order given by the relative position of points $(Y_n)_{n\geq 1}$. Since they are dense in $\intervalleff01$, no two circles are ever adjacent. }\label{growing:fig:spine looptree}
\end{figure}
Let us define a random sequence $(Y_n)_{n\geq 1}$ on $\intervalleff{0}{1}$ as follows
\begin{itemize}
	\item Let $Y_1\sim\mathrm{Beta}(1,\frac{1-\alpha}{\gamma})$.
	\item Recursively, if $(Y_1,\dots, Y_n)$ are already defined then conditionally on them the point $Y_{n+1}$ is distributed uniformly on $\intervalleff{0}{\max_{1\leq i\leq n}Y_i}$ with probability $(\max_{1\leq i\leq n}Y_i)$ and as $1-R_n\cdot (1-\max_{1\leq i\leq n}Y_i)$ with complementary probability, with $R_n\sim \mathrm{Beta}(1,\frac{1-\alpha}{\gamma})$ independent of everything else. 
\end{itemize}
Then, if $\gamma=\alpha$, the couple $(\mathcal{S}^{\alpha,\gamma},\mathcal{B}^{\alpha,\gamma})$ is such that $\mathcal{S}^{\alpha,\gamma}=\mathcal{B}^{\alpha,\gamma}$ which are just defined as the interval $\intervalleff{0}{1}$, rooted at $0$ and endowed with the points $(Y_n)_{n\geq 1}$. 

If $\gamma\neq\alpha$, we define the following random variables, independently of the sequence $(Y_n)_{n\geq 1}$. 
\begin{itemize}
	\item We let $(P_i)_{i\geq 1}\sim\mathrm{GEM}(\frac{\gamma}{\alpha},\frac{1-\alpha}{\alpha})$  and $S$ denote its $\frac{\gamma}{\alpha}$-diversity .
	\item Define recursively the sequence $(D_k)_{k\geq 1}$ starting from $D_1=1$. Conditionally on $(D_1,\dots,D_k)$ we have
	\begin{align*}
		D_{k+1}&=i \quad \text{with probability }P_i\quad \text{for any }i\in\{1,\dots,\max_{1\leq i\leq k}D_i\},\\
		&=1+\max_{1\leq i\leq k}D_i\quad \text{with complementary probability.}
	\end{align*}
\item The sequence $(X_k)_{k\geq 1}$ is then defined on the interval $\intervalleff{0}{S}$ as $(S\cdot Y_{D_k})_{k\geq 1}$.
\end{itemize}
The block $\cS^{\alpha,\gamma}$ is then defined as the interval $\intervalleff{0}{S}$ rooted at $0$ endowed with the sequence $(X_k)_{k\geq 1}$. In order to construct $\cB^{\alpha,\gamma}$, we introduce a sequence $(C_i)_{i\geq 1}$ of circles such that for all $i\geq 1$, $(C_i,d_i,\rho_i)$ is a circle with circumference $P_i$ endowed with its path distance and rooted at some point $\rho_i$. Conditionally on that, we take on each $C_i$ a point $U_i$ and a sequence $(V_{i,j})_{j\geq 1}$ of i.i.d. uniform random points on $C_i$.  
Then we consider their disjoint union
\begin{equation}
	\bigsqcup_{i=1}^\infty C_i,
\end{equation}
which we endow with the distance $\dist$ characterised by
\begin{align*}
	\dist(x,y) &=d_i(x,y) \qquad \text{if} \quad x,y\in C_i,\\
	&=d_i(x,U_i) + \sum_{k: Y_i<Y_k<Y_j}d_k(\rho_k,U_k)+d_j(\rho_j,y) \qquad \text{if} \quad x\in C_i,\ y\in C_j,\ Y_i<Y_j.
\end{align*}
Then $\cB^{\alpha,\gamma}$ is defined as the completion of $\bigsqcup_{i=1}^\infty C_i$ equipped with this distance, with distinguished points $(V_{N_k,k})_{k\geq 1}$. Its root $\rho$ can be obtained as a limit $\rho=\lim_{i\rightarrow\infty}\rho_{\sigma_i}$ for any sequence $(\sigma_i)_{i\geq 1}$ for which $Y_{\sigma_i}\rightarrow 0$. 

\begin{proof}[Proof of Lemma~\ref{growing:lem:convergence block tree looptree}]
Recall that at any time $n$ the object \[\cA^{\mathrm{tree}}(n)=(A^{\mathrm{tree}}(n), \dist_{n}^{\mathrm{tree}},\rho^{\mathrm{tree}},(x_i^{\mathrm{tree}})_{1\leq i \leq n})\] is endowed with a list of distinguished points $x_1^{\mathrm{tree}},x_2^{\mathrm{tree}},\dots x_n^{\mathrm{tree}}$, which are not necessarily distinct.  
Let us drop the superscript for readability.
For every $k\geq 1$ let us call $z_k$ the $k$-th vertex of degree $2$ in $\cA^\mathrm{tree}$ in order of creation.
For any $m$, denote $D_m$ the unique integer such that $x_m=z_{D_m}$. Also for any $i\geq 1$, denote 
\begin{align*}
	N_i(n)=\#\enstq{k\in\intervalleentier{1}{n}}{x_k=z_i}, 
\end{align*}
the number of distinguished points among $x_1,x_2,\dots,x_n$ that are equal to $z_i$ and $K_n=\#\enstq{x_i}{1\leq i\leq n}$ the total number of vertices created until time $n$. 
Suppose $\gamma< \alpha$ (the case $\gamma=\alpha$ is easier and follows using only a subset of the following arguments), then from the dynamics of $\cA^{\mathrm{tree}}$, the numbers $(N_i(n),i\geq 1)$ evolve as the number of customers seated at each table in order of creation in a Chinese Restaurant Process with parameters $(\frac{\gamma}{\alpha},\frac{1-\alpha}{\alpha})$.  
Thanks to Theorem~\ref{growing:thm:convergence chinese restaurant process}, the following convergences hold almost surely
 \begin{align*}
 	 \left(\frac{N_i(n),i \geq 1}{n}\right)  \underset{n \longrightarrow \infty}{\overset{\mathrm{a.s.} \text{\emph{ in} }\ell^1}\longrightarrow} \left(P_i,i\geq 1\right) \quad \text{and}\quad \frac{K_n}{n^{\gamma/\alpha} } \underset{n \rightarrow \infty}{\overset{\mathrm{a.s.}}\longrightarrow} S,
 \end{align*}
where $\left(P_i,i\geq 1\right)\sim \mathrm{GEM}(\frac{\gamma}{\alpha},\frac{1-\alpha}{\alpha})$ and $S$ denotes its $\frac{\gamma}{\alpha}$-diversity. Still thanks to Theorem~\ref{growing:thm:exhangeability chinese restaurant process}, conditionally on $\left(P_i,i\geq 1\right)$, the distribution of the sequence $(D_1,D_2,\dots )$ is exactly the one described in the description of $\cS^{\alpha,\gamma}$. 

Since $\cA^{\mathrm{tree}}(n)$ is just a line made of $K_n$ vertices (and so $K_n+1$ edges), the last convergence is enough to prove that, as rooted metric spaces we have the following almost sure convergence in $\M^{\bullet}$ 
\begin{align*}
(A^{\mathrm{tree}}(n),n^{-\frac{\gamma}{\alpha}}\cdot \dist_{n}^{\mathrm{tree}},\rho^{\mathrm{tree}}) \underset{n\rightarrow\infty}{\rightarrow} (\intervalleff{0}{S},\dist_{\intervalleff{0}{S}}, 0),
\end{align*} 
where $\dist_{\intervalleff{0}{S}}$ is the usual Euclidian distance on that interval.

Now let us handle the position of the created vertices along the line. 
When $z_1$ the first vertex of degree $2$ is created, it is adjacent to one edge of weight $1-\alpha$ on the leaf-side and one edge of weight $\gamma$ on the root-side. Every time that a new vertex is created, on either side, it reinforces the weight of that side by creating a new edge of weight $\gamma$. 
We recognize the dynamics of a P\'olya urn, hence the proportion of the number of vertices on the root-side of $z_1$ converges almost surely to some random variable $Y_1\sim \mathrm{Beta}(1,\frac{1-\alpha}{\gamma})$. This means that almost surely
\begin{align*}
	\frac{\dist_{n}^{\mathrm{tree}}(\rho,z_1)}{n^{\frac{\gamma}{\alpha}}}=\frac{\dist_{n}^{\mathrm{tree}}(\rho,z_1)}{K_n}\cdot \frac{K_n}{n^{\frac{\gamma}{\alpha}}} \underset{n\rightarrow\infty}{\rightarrow} Y_1\cdot S.
\end{align*}
Conditionally on $Y_1$, the created vertices $z_2,z_3,z_4,\dots$ are inserted independently on the root-side or the leaf-side of $z_1$, with probability distribution $(Y_1,1-Y_1)$. 
Then, only considering what happens to the root-side of $z_1$, what we observe is a uniform edge-splitting process (all the edges have weight $\gamma$) and so thanks to Lemma~\ref{growing:lem:convergence edge-splitting} the relative position of all those vertices converges and it is uniform on that segment.
Finally, what happens on the leaf-side of $z_1$ has exactly the same distribution has the whole process starting from only one edge with weight $1-\alpha$. 
For every $i\geq 1$, the limit $Y_i=\lim_{k\rightarrow\infty}\frac{\dist_{n}^{\mathrm{tree}}(\rho,z_i)}{K_n}$ then almost surely exists and one can check that this sequence has the same distribution as the sequence of random variables $(Y_1,Y_2,\dots)$ described in the construction of $\mathcal{S}^{\alpha,\gamma}$. Because the sequence $(x_k)_{k\geq 1}$ is exactly given by $(z_{D_k})_{k\geq 1}$ by construction, using all of the above, we get the almost sure convergence in $\M^{\infty\bullet }$
\begin{align*}
\cA^{\mathrm{tree}}(n)=(A^{\mathrm{tree}}(n),n^{-\frac{\gamma}{\alpha}}\cdot \dist_{n}^{\mathrm{tree}},\rho^{\mathrm{tree}},(x_i^{\mathrm{tree}})_{1\leq i \leq n}) \underset{n\rightarrow\infty}{\rightarrow} \cS^{\alpha,\gamma}=(\intervalleff{0}{S},\dist_{\intervalleff{0}{S}}, 0, (S\cdot Y_{D_i})_{i\geq 1}).
\end{align*} 

Let us understand what happens for the corresponding string of loops $\cA^{\mathrm{loop}}$. 
Recall that $\cA^{\mathrm{loop}}(n)$ is also endowed with distinguished points $x_1^{\mathrm{loop}},x_2^{\mathrm{loop}},\dots, x_n^{\mathrm{loop}}$ and let us keep the superscripts for the rest of the proof. 
By construction, the loop that corresponds to vertex $z_i$ at time $n$ contains all the $N_i(n)$ vertices $\enstq{x_j^{\mathrm{loop}}}{x_j^{\mathrm{tree}}=z_i, j\leq n}$ that are of degree $2$, and two other vertices, which are respectively shared with the loop just above and the one just below. 

Now three observations:
first, in the limit $n\rightarrow\infty$, the size of this loop is such that $(N_i(n)+2)\sim nP_i$ almost surely, so when scaling the distance by $n^{-1}$, a circle of length $P_i$ is going to appear in the limit.
Second, starting from the creation of the loop, the addition of every vertex around the loop follows exactly the dynamic of a uniform edge-splitting process, and hence thanks to Lemma~\ref{growing:lem:convergence edge-splitting}, the limiting positions of all the vertices created around the loop are uniform along the length of the circle.
Last, the $\ell^1$ convergence $\left(\frac{N_i(n),i \geq 1}{n}\right)  \underset{n \longrightarrow \infty}{\overset{\mathrm{a.s.} \text{\emph{ in} }\ell^1}\longrightarrow} \left(P_i,i\geq 1\right)$ and the convergence of the number of loops $\frac{K_n}{n^{\gamma/\alpha}}\rightarrow S$ ensure that the total (normalised) length of all the loops created after time $t$ tends to $0$ uniformly in $n$ as $t\rightarrow\infty$, which ensure the a.s.\ relative compactness of the sequence $n^{-1}\cdot \cA^{\mathrm{loop}}(n)$.
From these observations, the identification of the limit is quite straightforward and the details are left to the reader. 
\end{proof}

\paragraph{Convergence results}
We can now express our scaling limit convergences for our processes.
\begin{proposition}\label{growing:prop:convergence alpha gamma tree looptree}
	We have the following joint convergence almost surely in the Gromov--Hausdorff--Prokhorov topology, 
	\begin{align*}
	(T_n^{\alpha,\gamma},n^{-\gamma}\cdot\dist_{\mathrm{gr}},\mu_{\mathrm{unif}}) &\underset{n\rightarrow\infty}{\longrightarrow} \cT^{\alpha,\gamma},\\
	(\mathrm{Loop}(T_n^{\alpha,\gamma}),n^{-\alpha}\cdot \dist_{\mathrm{gr}},\mu_{\mathrm{unif}})&\underset{n\rightarrow\infty}{\longrightarrow} \cL^{\alpha,\gamma}.
	\end{align*}
	The limiting objects can be constructed using an iterative gluing construction with i.i.d.\ blocks using a weight sequence $(\mathsf{m}_n)_{n\geq 1}$ obtained as the increment of a Mittag-Leffler Markov chain $\MLMC(\alpha,1-\alpha)$. 
	The scaling factors are taken as $(\mathsf{m}_n^{\gamma/\alpha})_{n\geq 1}$ for the first one and $(\mathsf{m}_n)_{n\geq 1}$ for the second one, using block i.i.d. block with the same joint distribution as $(\mathcal{S}^{\alpha,\gamma},\mathcal{B}^{\alpha,\gamma})$ defined above. 
\end{proposition}

\begin{proof}[Proof of Proposition~\ref{growing:prop:convergence alpha gamma tree looptree}]
The proof of this proposition is another application of Theorem~\ref{growing:thm:metatheorem}. Conditions \ref{growing:it:Pn preferential attachment} and \ref{growing:it:Dn processus en deguk} are satisfied thanks to the discussion in the first paragraph of the section,
condition \ref{growing:it:convergence des processus A} is satisfied thanks to Lemma~\ref{growing:lem:convergence block tree looptree}. 
Condition \ref{growing:it:moment condition on the sup of A} is satisfied for $\cA^{\mathrm{tree}}$ thanks to the second part of Lemma~\ref{growing:lem:convergence block tree looptree}; for  $\cA^{\mathrm{loop}}$, it comes from the fact that the total number of edges in $\cA^{\mathrm{loop}}(n)$ is deterministically smaller than $3n+1$, hence also its diameter.

It remains to check that \ref{growing:it:measures converge} is satisfied for some measure-valued decoration $\bnu^{\mathrm{tree},(n)}$ and $\bnu^{\mathrm{loop},(n)}$ such that $\left(\sG(\cD^{\mathrm{tree},(n)},\bnu^{\mathrm{tree},(n)}), \sG(\cD^{\mathrm{loop},(n)},\bnu^{\mathrm{loop},(n)})\right)$ coincides with $(T_n^{\alpha,\gamma},\mathrm{Loop}(T_n^{\alpha,\gamma}))$ endowed with their graph distance and uniform measure on the vertices. This is, as for all other examples, achieved by charging every vertex of every block $\cD^{(n)}(u)$ of the decoration with the same weight, except its root vertex if $u\neq \emptyset$.

Then the proof goes as the proof of Proposition~\ref{growing:prop:convergence marchal algo 1}: the total weight, of vertices (on one side) and of edges (on the other side) in $T_n^{\alpha,\gamma}$ evolves as a balanced generalised P\'olya urn with replacement matrix
\begin{align*}
\begin{bmatrix}
\alpha & 1-\alpha\\
\alpha-\gamma & 1-\alpha +\gamma
\end{bmatrix}.
\end{align*}
Using Lemma~\ref{growing:lem:convergence generalised polya urn} the total weight of edges is asymptotically $\frac{1-\alpha}{1-\gamma}n$ almost surely. 
Since at time $n$, the total weight of edges adjacent to a leaf is exactly $(1-\alpha)n$, it means that that the total weight of the internal edges is asymptotically $\frac{\gamma(1-\alpha)}{1-\gamma}n$. In the end, the asymptotic number of edges in $T_n^{\alpha,\gamma}$ (and hence also of vertices) is $(1+\frac{1-\alpha}{1-\gamma})n$. This is also true for the number of vertices in $\mathrm{Loop}(T_n^{\alpha,\gamma})$ because by construction, it corresponds to the number of edges in $T_n^{\alpha,\gamma}$.

Going along the same proof as for Proposition~\ref{growing:prop:convergence marchal algo 1}, this is enough to show that asymptotically almost surely, the measures $\nu^{\mathrm{tree},(n)}$ and $\nu^{\mathrm{loop},(n)}$ on $\bU$ satisfy 
\begin{align*}
	\nu^{\mathrm{tree},(n)}(T(u)) \underset{n\rightarrow\infty}{\sim}\nu^{\mathrm{loop},(n)}(T(u)) \underset{n\rightarrow\infty}{\sim} \nu_n(T(u)),
\end{align*}
together with
\begin{align}
	 \nu^{\mathrm{tree},(n)}({u}) \underset{n\rightarrow\infty}{\rightarrow}0 \qquad \text{and} \qquad \nu^{\mathrm{loop},(n)}({u}) \underset{n\rightarrow\infty}{\rightarrow}0 
\end{align}
for every $u\in \bU$, and $\nu_n$ being the uniform measure on the preferential attachment tree $\ttP_n$ associated to the construction. 
This is enough to prove the a.s. convergence of $\nu^{\mathrm{tree},(n)}$ and $\nu^{\mathrm{loop},(n)}$ to $\mu$, which is the weak limit of $\nu_n$. This finishes the proof.
\end{proof}

\paragraph{Remarks on the limiting space.}
When $\gamma=1-\alpha$ then the limiting spaces $(\cT^{\alpha,\gamma},\cL^{\alpha,\gamma})$ are respectively (a constant multiple of) the $\frac{1}{\gamma}$-stable tree and its associated $\frac{1}{\gamma}$-stable looptree, thanks to the convergence results \cite{curien_random_2014,curien_stable_2013}.

\subsubsection{Looptrees constructed using affine preferential attachment}
This last model is very similar to the case of the generalised Rémy's algorithm. We mention it separately because it was the object of a conjecture by Curien, Duquesne, Kortchemski and Manolescu \cite{curien_scaling_2015}, to which we provide a positive answer.

First, for any $\delta>-1$, let us define the model $\mathrm{LPAM}^\delta$ which produces sequences $(T_n^\delta)_{n\geq 1}$ of plane trees. Start with $T_1^\delta$ containing a unique vertex connected to a root by an edge (the root is not considered as a real vertex here). Then, if $T_n^\delta$ is already constructed, take a vertex in the tree at random (the root does not count) with probability proportional to its degree plus $\delta$, then add an edge connected to a new vertex in uniformly chosen corner around this vertex. This yields $T_{n+1}^\delta$. 
\begin{proposition}
We have the following almost sure convergence 
	\begin{align*}
 (\mathrm{Loop}(T_n^{\delta}),	n^{-\frac{1}{2+\delta}}\cdot\dist_{\mathrm{gr}},\mu_{\mathrm{unif}} )\underset{n\rightarrow\infty}{\longrightarrow}\cL^{\delta}
	\end{align*}
in the Gromov--Hausdorff--Prokhorov topology. The limiting object can be constructed using an iterative gluing construction with deterministic blocks equal to a circle with unit circumference, using a sequence of weights and scaling factors $(\mathsf{m}_n)_{n\geq 1}$ obtained as the increment of a Mittag-Leffler Markov chain $\MLMC\left(\frac{1}{2+\delta},\frac{1+\delta}{2+\delta}\right)$.
\end{proposition}
The Hausdorff dimension of $\cL^{\delta}$ is $2 + \delta$ almost surely using \cite[Theorem~1]{senizergues_random_2019}. The proof of this is really close to the one used for the generalised Remy algorithm, so we omit it. 

\appendix
\section{Some useful definitions and results}\label{growing:sec:computations}
This section is devoted to recalling and proving some definitions and results that are used in some of our applications.

\subsection{P\'olya urns}\label{growing:subsection:classical polya urn}

\paragraph{Dirichlet distributions.}  For parameters $a_1, a_2, \ldots, a_n > 0$, the Dirichlet distribution $\mathrm{Dir}(a_1, a_2, \ldots, a_n)$ has density
\[
\frac{\Gamma(\sum_{i=1}^n a_i)}{\prod_{i=1}^n \Gamma(a_i)} \prod_{j=1}^{n} x_i^{a_j-1}
\]
with respect to the Lebesgue measure on the simplex $\{(x_1, \ldots, x_n) \in [0,1]^n: \sum_{i=1}^n x_i = 1\}$.
In the case $n=2$, a random variable with distribution $\mathrm{Dir}(a_1,a_2)$ can be written as $(B,1-B)$, and $B$ is said to have distribution $\mathrm{Beta}(a_1,a_2)$. 

\paragraph{Convergence and exchangeability.}
\begin{theorem}\label{growing:thm:classical Polya urn}
Consider an urn model with $k$ colours labelled from $1$ to $k$, with initial weights $a_1,\ldots,a_k>0$ respectively. At each step $n\geq 1$, draw a colour with a probability proportional to its weight and add weight $\beta$ to this colour; we let $D_n$ be the label of the drawn colour. 
Let $X^{(1)}_n,\ldots, X^{(k)}_n$ denote the weights of the $k$ colours after $n$ steps. Then,
\begin{enumerate}
	\item we have the following convergence \[
	\left(\frac{X^{(1)}_n}{\beta n},\ldots,\frac{X^{(k)}_n}{\beta n}\right)  \underset{n \rightarrow \infty}{\overset{\mathrm{a.s.}}\longrightarrow} (X^{(1)},\ldots,X^{(k)})
	\]
	where $(X^{(1)},\ldots,X^{(k)}) \sim \mathrm{Dir}(\frac{a_1}{\beta}, \ldots, \frac{a_k}{\beta})$. 
	\item Conditionally on the limiting proportions $(X^{(1)},\ldots,X^{(k)})$, the sequence of draws $D_1,D_2,\dots $ is i.i.d.\ such that
	\begin{align*}
		\Ppsq{D_1=i}{X^{(1)},\ldots,X^{(k)}}=X^{(i)},
	\end{align*}
	for all $1\leq i \leq k$.
\end{enumerate}
\end{theorem}
\paragraph{Balanced generalised P\'olya urns.}
Consider the following urn model with two colours, which depends on four positive real numbers $a,b,c,d>0$. Starting from an initial condition, the weight $(X_n,Y_n)$ of the two colours in the urn evolve in the following way:
at each step, draw a colour from the urn with probability proportional to its weight in the urn. If colour $1$ is drawn, add $a$ to the weight of colour $1$ and $b$ to the weight of colour $2$.  If colour $2$ is drawn, add $c$ to the weight of colour $1$ and $b$ to the weight of colour $2$.
The matrix $M=\begin{bmatrix}
a & b\\
c & d
\end{bmatrix}$ is called the replacement matrix of the urn.

We suppose that the urn is \emph{balanced}, meaning that $a+b=c+d$, and we call $\lambda_1$ the eigenvalue of $M$ with largest modulus, being here equal to $a+b$.
Let us now state a lemma that follows from the application of the results of \cite{athreya_embedding_1968} to our setting.
\begin{lemma}\label{growing:lem:convergence generalised polya urn}
For any initial weight configuration, we have the following almost sure convergence
\begin{align*}
	\left(\frac{X_n}{n(a+b)},\frac{Y_n}{n(a+b)}\right) \longrightarrow v_1,
\end{align*} 
where $v_1$ is the left eigenvector associated to $\lambda_1$, normalized to have components that sum to $1$.
\end{lemma}

\subsection{Chinese Restaurant Processes}\label{growing:subsection:chinese restaurant process}
\paragraph{Generalized Mittag-Leffler distributions.} Let $0<\alpha<1$, $\theta>-\alpha$. The {generalized Mittag-Leffler distribution} $\mathrm{ML}(\alpha,\theta)$ is characterised by its moments. For $M\sim \mathrm{ML}(\alpha,\theta)$ and any $p \in \mathbb R_+$ we have, 
\[
\mathbb E\left[ M^p\right]=\frac{\Gamma(\theta) \Gamma(\theta/\alpha + p)}{\Gamma(\theta/\alpha) \Gamma(\theta + p \alpha)}=\frac{\Gamma(\theta+1) \Gamma(\theta/\alpha + p+1)}{\Gamma(\theta/\alpha+1) \Gamma(\theta + p \alpha+1)}.
\]
\paragraph{GEM and PD distribution.}
Let $0<\alpha<1$, $\theta>-\alpha$ and for $i \ge 1$, let $B_i \sim \mathrm{Beta}(1-\alpha, \theta + i \alpha)$ independently.  Then the sequence $(P_i)_{i\geq 1}$ where $P_i=B_i \prod_{k=1}^{i-1} (1 - B_k)$ has the $\mathrm{GEM}(\alpha,\theta)$ distribution. 
The reordered sequence $(P_i^{\downarrow})_{i \ge 1}$ in non-increasing order is said to have the $\mathrm{PD}(\alpha,\theta)$ distribution.
In this setting the limit $ W := \Gamma(1 - \alpha) \lim_{i \to \infty} i (P_i^{\downarrow})^{\alpha}$ almost surely exists and is said to be the \emph{$\alpha$-diversity} of the sequence $(P_i)_{i\geq 1}$. It has the $\mathrm{ML}(\alpha,\theta)$ distribution (see \cite{pitman_combinatorial_2006}).

\paragraph{Chinese Restaurant Process.}
Fix two parameters $\alpha \in (0,1)$ and $\theta>-\alpha$. Let us introduce the so-called Chinese restaurant process with seating plan $(\alpha,\theta)$. We refer to \cite{pitman_combinatorial_2006} for the definition and properties of this process.
The process starts with one table occupied by one customer and then evolves in a Markovian way as follows: given that at stage $n$ there are $k$ occupied tables with $n_i$ customers at table $i$, a new customer is placed at table $i$ with probability $(n_i-\alpha)/(n+\theta)$ and placed at a new table with probability $(\theta+k\alpha)/(n+\theta)$. 

Let $N_i(n),i\geq 1$ be the number of customers at table $i$ at stage $n$. Let also $K_n$ denote the number of occupied tables at stage $n$ and $D_n$ the number of the table at which the $n$-th costumer sits.
The following theorems follow from \cite[Chapter~3]{pitman_combinatorial_2006}.
\begin{theorem}\label{growing:thm:convergence chinese restaurant process}
In this setting we have the following convergences
\begin{equation*}
\left(\frac{N_i(n),i \geq 1}{n}\right)  \underset{n \longrightarrow \infty}{\overset{\mathrm{a.s.} \text{\emph{ in} }\ell^1}\longrightarrow} \left(P_i,i\geq 1\right) \quad \text{and}\quad \frac{K_n}{n^{\alpha} } \underset{n \rightarrow \infty}{\overset{\mathrm{a.s.}}\longrightarrow} W,
\end{equation*}
where $\left(P_i,i\geq 1\right)$ follows a $\mathrm{GEM}(\alpha,\theta)$-distribution and $W$ is its $\alpha$-diversity.
The sequence $(Q_j,j\geq 1)=\left(P_i^\downarrow,i\geq 1\right)$, defined as the non-increasing rearrangement of $\left(P_i,i\geq 1\right)$ has then the $\mathrm{PD}(\alpha,\theta)$-distribution.
\end{theorem}
The following result allows us to describe the distribution of the process, conditionally on the limiting size of the tables.
\begin{theorem}\label{growing:thm:exhangeability chinese restaurant process}
In the setting of the previous theorem, conditionally on the sequence $\left(P_i,i\geq 1\right)$, the distribution of $(D_n)_{n\geq 1}$ can be described as follows: 
\begin{enumerate}
	\item $D_1=1$ almost surely.
	\item Conditionally on $D_1,\dots,D_n$, we have 
	\begin{align*}
D_{n+1}&=k \qquad \text{with probability } P_k, \text{ for any } 1\leq k\leq \max_{1\leq i\leq n}D_i,\\
&= 1+\max_{1\leq i\leq n}D_i \qquad \text{with complementary probability}.	
	\end{align*}
\end{enumerate}

Also, conditionally on $\left(Q_j, j\geq 1\right)$, the distribution of $(D_n)_{n\geq 1}$ can be described as follows: 
\begin{enumerate}
	\item Let $(I_n)_{n\geq 1}$ be i.i.d.\ with distribution $\Pp{I_1=k}=Q_k$,
	\item $D_1=1$ almost surely,
	\item conditionally on $I_1,\dots,I_n$, we have 
	\begin{align*}
	D_{n+1}=\begin{cases*}
1+\max_{1\leq i\leq n}D_i \qquad \text{ if } I_{n+1}\notin \{I_1,I_2,\dots ,I_n\},\\
D_{J_n} \text{ for }J_n=\inf\enstq{k\geq 1}{I_{n+1}=I_k} \qquad\text{otherwise}.	
\end{cases*}
	\end{align*}
\end{enumerate}
\end{theorem}
For any $\alpha \in (0,1)$ and $\theta>-\alpha$, the law of this evolving configuration of customers around the different tables using the $(\alpha,\theta)$ seating plan is denoted by $\P_{\alpha,\theta}$. The following result is expressed for the canonical process under the probability measure $\P_{\alpha,\theta}$.
\begin{lemma}\label{growing:lem:number of tables in a CRP}
	For every $p\geq 1$, we have
	\begin{equation}
		\mathbb{E}_{\alpha,\theta}\left[\sup_{n\geq 1}\left(\frac{K_n}{n^{\alpha}}\right)^p\right]<\infty.
	\end{equation}
\end{lemma}
\begin{proof}
Let $f_{\alpha,\theta}(k):=\frac{\Gam{\theta/\alpha+k}}{\Gam{\theta/\alpha+1}\Gam{k}}$, and $\cF_n$ the $\sigma$-field generated by the $n$ first steps of the process, then 
\begin{equation*}
\left(	\frac{d\P_{\alpha,\theta}}{d\P_{\alpha,0}}\right)_{\vert_{\cF_n}}=\frac{f_{\alpha,\theta}(K_n)}{f_{1,\theta}(n)}=M_{\alpha,\theta,n},
\end{equation*}
which is a martingale under $\P_{\alpha,0}$ and bounded in $L^p$, for all $p>0$, see \cite[Proof of Theorem~3.28]{pitman_combinatorial_2006}.
There exists a constant $c>1$ such that for any $k,n\geq 1$
\begin{equation*}
\frac{1}{c}\left(\frac{k}{n^\alpha}\right)^\theta	\leq \frac{f_{\alpha,\theta}(k)}{f_{1,\theta}(n)}\leq c\left(\frac{k}{n^\alpha}\right)^\theta.
\end{equation*}
Introduce $M_{\alpha,\theta}^*:=\sup_{n\geq1}M_{\alpha,\theta,n}$. Using Doob's maximal inequality, we get
\begin{equation*}
\E_{\alpha,0}\left[\left(M_{\alpha,\theta}^*\right)^p\right]\leq \left(\frac{p}{p-1}\right)^p \E_{\alpha,0}\left[M_{\alpha,\theta}^p\right].
\end{equation*}
Hence 
\begin{align*}
\E_{\alpha,\theta}\left[\left(\sup_{n\geq 1} \frac{K_n}{n^\alpha}\right)^p\right]\leq c 
\E_{\alpha,\theta}\left[\left(M_{\alpha,1}^*\right)^p\right]&=\E_{\alpha,0}\left[\left(M_{\alpha,1}^*\right)^p\cdot M_{\alpha,\theta}\right] \\
&\leq \sqrt{\E_{\alpha,0}\left[\left(M_{\alpha,1}^*\right)^{2p}\right] \E_{\alpha,0}\left[M_{\alpha,\theta}^2\right]} <\infty,
\end{align*}	
where the last inequality uses Cauchy-Schwarz inequality. This finishes the proof of the lemma.
\end{proof}

\subsection{The supremum of the normalised height in Rémy's algorithm}
Let $(T_n)_{n\geq 1}$ be a sequence of trees evolving using Rémy's algorithm. This sequence is a Markov chain in a state space of binary planted trees. Let us denote $H:=\sup_{n\geq 1}(n^{-1/2}\haut(T_n))$. We prove the following bound on the tail of the distribution of $H$.
\begin{lemma}\label{growing:lem:sup hauteur algo remy queue gaussienne}
There exists constants $C_1$ and $C_2$ such that for all $x>0$,
\begin{equation*}
\Pp{H>x}\leq C_1 \exp (-C_2 x^2).
\end{equation*}
In particular, $H$ admits moments of all orders.
\end{lemma}
\begin{proof} 
Let $\tau_x:=\inf\enstq{n\geq 1}{\haut(T_n)> x n^{1/2}}$. We write
\begin{align*}
\Pp{H>x}&=\Pp{\tau_x<+\infty}\\
&\leq \Pp{\lim_{n\rightarrow\infty}n^{-1/2}\haut(T_n)> \frac{x}{2}}+\Pp{\tau_x<+\infty,\lim_{n\rightarrow\infty}n^{-1/2}\haut(T_n)\leq \frac{x}{2}}
\end{align*}
We know thanks to \cite{curien_stable_2013} that the trees constructed using Rémy's algorithm converge almost surely in Gromov-Hausdorff topology to Aldous' Brownian tree, so $n^{-1/2}\cdot T_n\rightarrow \cT$ as $n\rightarrow\infty$. By continuity we have $\lim_{n\rightarrow\infty}n^{-1/2}\haut(T_n)=\haut(\cT)$. Estimates \cite{kennedy_distribution_1976} on the height of the Brownian tree show that the first term of the above sum is smaller than $C_1 \exp (-C_2 x^2)$, for some choice of constants $C_1$ and $C_2$. Fix some $N_0\geq 1$ that we will chose later. Then compute
\begin{align*}
\Pp{\tau_x<+\infty,\lim_{n\rightarrow\infty}n^{-1/2}\haut(T_n)\leq \frac{x}{2}}
&=\sum_{N=1}^{+\infty}\Pp{\tau_x=N}\Ppsq{\haut(\cT)\leq \frac{x}{2}}{\tau_x=N}\\
&\leq \sum_{N=1}^{N_0}\Pp{\tau_x=N}+\sup_{N\geq N_0} \Ppsq{\haut(\cT)\leq \frac{x}{2}}{\tau_x=N}\\
&\leq N_0 C_1 \exp (-C_2 x^2)+\sup_{N\geq N_0} \Ppsq{\haut(\cT)\leq \frac{x}{2}}{\tau_x=N},
\end{align*}
where in the last inequality we use the fact that for all $N$,
\begin{equation*}
\Pp{\tau_x=N}\leq \Pp{\haut(T_N)\geq xN^{1/2}}\leq C_1 \exp (-C_2 x^2),
\end{equation*} 
using the results of \cite{addario-berry_most_2019}.

Now, let us reason conditionally on the event $\left\lbrace \tau_x=N\right\rbrace$. On that event, there exists in the tree $T_N$ at least one path of length $\lfloor xN^{1/2} \rfloor$ starting from the root ending at a vertex $v$. The height $H_n(v)$ of the vertex $v$ at time $n$ evolves under the same dynamics as the number of balls in a triangular urn model with replacement matrix $ \begin{bmatrix} 1 & 1 \\ 0 & 2\end{bmatrix}$, and starting proportion $(\lfloor xN^{1/2} \rfloor,2N+1-\lfloor xN^{1/2} \rfloor)$, see \cite{janson_limit_2006} for definition and results for those urns. 
Using a theorem of \cite{janson_limit_2006}, as $n\rightarrow\infty$, we have the almost sure convergence
\begin{equation*}
\frac{H_n(v)}{n^{1/2}}\longrightarrow W_N^x,
\end{equation*}
where $W_N^x=\alpha_N^x\cdot M_N$ is the product of two independent variables, with 
\begin{equation}
	\beta_N^x\sim \mathrm{Beta}(\lfloor xN^{1/2} \rfloor,2N+1-\lfloor xN^{1/2} \rfloor) \qquad \text{and} \qquad M_N\sim \ML\left(\frac{1}{2},\frac{2N+1}{2}\right). 
\end{equation}
Then we write
\begin{align*}
\Pp{W_N^x\leq \frac{x}{2}}&=	\Pp{\beta_N^x\cdot M_N\leq \frac{x}{2}}\\
&\leq \Pp{\beta_N^x\cdot M_N\leq \frac{x}{2},\ M_N\geq \frac{3}{2}N^{1/2} } + \Pp{M_N\leq \frac{3}{2}N^{1/2}}\\
&\leq \Pp{\beta_N^x\leq \frac{1}{3}\cdot x \cdot N^{-1/2}} + \Pp{M_N\leq \frac{3}{2}N^{1/2}}.
\end{align*}
We bound the two terms in the last sum using Chebychev inequality, using
\begin{align*}
\Ec{\beta_N^x}=\frac{\lfloor xN^{1/2} \rfloor}{2N+1}\sim \frac{x}{2}N^{-1/2}, \qquad \Var{\beta_N^x}=\frac{\lfloor xN^{1/2} \rfloor (2N+1 - \lfloor xN^{1/2} \rfloor)}{(2N+1)^2(2N+2)}\sim \frac{x}{4}N^{-3/2}.
\end{align*}
We also have
\begin{align*}
\Ec{M_N}=\frac{\Gam{N+1/2}\Gam{2N+2}}{\Gam{2N+1}\Gam{N+1}}\sim 2 N^{1/2},
\end{align*}
and 
\begin{align*} \Var{M_N}=\frac{\Gam{N+1/2}\Gam{2N+3}}{\Gam{2N+1}\Gam{N+3/2}}-\left(\frac{\Gam{N+1/2}\Gam{2N+2}}{\Gam{2N+1}\Gam{N+1}}\right)^2\leq 1.
\end{align*}
Hence,
\begin{align*}
\Pp{\beta_N^x\leq \frac{1}{3}\cdot x \cdot N^{-1/2}}&\leq \Pp{\abs{\beta_N^x-\Ec{\beta_N^x}}\geq \frac{1}{8}\cdot x \cdot N^{-1/2}}\\
&\leq 8 x^2 N \Var{\beta_N^x}\\
&\leq C x^3 N^{-1/2}.
\end{align*}
with $C$ a constant that is independent of $x$ and $N$. We also have
\begin{align*}
\Pp{M_N\leq \frac{3}{2}N^{1/2}}&\leq \Pp{\abs{M_N-\Ec{M_N}}\geq \frac{1}{3}N^{1/2}}\\
&\leq 3 N^{-1} \Var{M_N}\\
&\leq C' N^{-1},
\end{align*}
with $C'$ another constant that does not depend on $x$ or $N$. All this analysis  was done conditionally on the event $\left\lbrace \tau_x=N\right\rbrace$, so in fact, we have for all $N\geq N_0$,
\begin{equation*}
\Ppsq{\haut(\cT)\leq \frac{x}{2}}{\tau_x=N} \leq  C x^3 N^{-1/2} + C' N^{-1}\leq C x^3 N_0^{-1/2} + C' N_0^{-1}
\end{equation*}
Now we just take $N_0=\exp\left(\frac{C_2}{2} x^2\right)$ and the result follows.
\end{proof}

\printbibliography
\end{document}